\documentclass[reqno]{amsart}
\usepackage{mathrsfs}
\usepackage{amsmath,amsfonts,amssymb,amsthm}
\usepackage{upgreek}
\usepackage[usenames,dvipsnames]{xcolor}
\usepackage{enumitem}
\usepackage{graphicx}
\usepackage{overpic}
\usepackage{cancel}
\usepackage[outdir=./]{epstopdf}
\usepackage[colorlinks=true]{hyperref} 
\usepackage{url}
\usepackage[utf8]{inputenc}
\usepackage{booktabs,multirow}
\usepackage[font=footnotesize,width=\textwidth]{caption}
\hypersetup{linkcolor=Violet,citecolor=PineGreen}
\newtheorem{teor}{Theorem}[section]
\newtheorem{lema}[teor]{Lemma}
\newtheorem{prop}[teor]{Proposition}
\newtheorem{coro}[teor]{Corollary}
\theoremstyle{definition}
\newtheorem{defi}[teor]{Definition}
\newtheorem{exa}[teor]{Example}

\newtheorem{nota}[teor]{Remark}
\newtheorem{notas}[teor]{Remarks}
\numberwithin{equation}{section}

\newcommand{\R}{\mathbb R}
\newcommand{\To}{\mathbb T}

\newcommand{\Z}{\mathbb{Z}}
\newcommand{\Q}{\mathbb{Q}}
\newcommand{\C}{\mathbb{C}}
\newcommand{\N}{\mathbb{N}}

\newcommand{\mA}{\mathcal{A}}
\newcommand{\mB}{\mathcal{B}}
\newcommand{\mC}{\mathcal{C}}

\newcommand{\mK}{\mathcal{K}}
\newcommand{\mM}{\mathcal{M}}

\newcommand{\mR}{\mathcal{R}}
\newcommand{\mS}{\mathcal{S}}

\newcommand{\mV}{\mathcal{V}}
\newcommand{\ep}{\varepsilon}

\newcommand{\mI}{\mathcal{I}}
\newcommand{\mJ}{\mathcal{J}}
\newcommand{\mL}{\mathcal{L}}

\newcommand{\pr}{P^1(\R)}

\newcommand{\W}{\Omega}
\newcommand{\WR}{{\W\times\R}}
\newcommand{\RR}{\R\times\R}
\newcommand{\w}{\omega}
\newcommand{\lb}{\lambda}

\newcommand{\ma}{\mathfrak{a}}
\newcommand{\mb}{\mathfrak{b}}
\newcommand{\mc}{\mathfrak{c}}
\newcommand{\md}{\mathfrak{d}}

\newcommand{\ml}{\mathfrak{l}}
\newcommand{\mm}{\mathfrak{m}}
\newcommand{\mn}{\mathfrak{n}}
\newcommand{\muk}{\mathfrak{u}}
\newcommand{\mf}{\mathfrak{f}}
\newcommand{\mg}{\mathfrak{g}}
\newcommand{\mh}{\mathfrak{h}}

\newcommand{\tml}{\tilde\ml}

\newcommand{\tmuk}{\tilde\muk}

\newcommand{\upalfa}{$\upalpha$}
\newcommand{\upomeg}{$\upomega$}

\newcommand{\G}{\Gamma}

\newcommand{\ws}{\w{\cdot}s}
\newcommand{\wt}{\w{\cdot}t}

\newcommand{\bwt}{\bar\w{\cdot}t}

\newcommand{\lsm}{\left[\begin{smallmatrix}}
\newcommand{\rsm}{\end{smallmatrix}\right]}

\newcommand{\merg}{\mathfrak{M}_\mathrm{erg}(\W,\sigma)}
\DeclareMathOperator{\arccot}{arccot}

\begin{document}
\title[Loss-of-hyperbolicity and jump bifurcations]
{Loss-of-hyperbolicity and jump bifurcations in scalar nonautonomous d-concave ODEs}
\author[J. Due\~{n}as]{Jes\'{u}s Due\~{n}as}
\author[C. N\'{u}\~{n}ez]{Carmen N\'{u}\~{n}ez\textsuperscript{*}}
\author[R. Obaya]{Rafael Obaya}
\address{{\rm(J. Due\~{n}as)}.  Departamento de Matem\'{a}tica Aplicada,
Escuela de Ingenier\'{\i}a Infor\-m\'{a}\-tica, Universidad de Valladolid, Paseo de Bel\'{e}n, 15, 47011 Valladolid, Spain.
}
\address{{\rm(C.N\'{u}\~{n}ez, R. Obaya)}. Departamento de Matem\'{a}tica Aplicada, Escuela de Ingenier\'{\i}as
Industriales, Universidad de Valladolid, Paseo Prado de la Magdalena 3-5, 47011 Va\-lla\-dolid, Spain.}
\address{{}The three authors are members of IMUVA: Instituto de Investigaci\'{o}n en Matem\'{a}ticas,
Universidad de Valladolid.}
\email[J.~Due\~{n}as]{jesus.duenas@uva.es}
\email[C.~N\'{u}\~{n}ez]{carmen.nunez@uva.es}
\email[R.~Obaya]{rafael.obaya@uva.es}
\thanks{All the authors were supported by Ministerio de Ciencia, Innovaci\'{o}n y
Universidades (Spain) under project PID2024-156691NB-I00, and by Department of Education of the Junta de Castilla y Le\'{o}n (Spain) and FEDER Funds under project CLU-2025-1-02-IMUVA.}
\thanks{\textsuperscript{*}Corresponding author.}
\begin{abstract}
The paper studies bifurcation for additively parametrized scalar nonautonomous ODEs of the
form $x'=f(t,x)+\lb$, where $f$ is d-concave and coercive in the state variable.
By formulating the problem through the hull of $f$ and the associated skewproduct flow,
the autonomous notion of bifurcation can be extended in terms of loss of hyperbolicity
of copies of the base. A fairly complete classification of the possible bifurcation diagrams
is provided. The notion of a jump bifurcation, linked to abrupt changes in the global attractor,
is introduced and related to the occurrence of critical transitions. Several nontrivial, genuinely
nonautonomous examples are presented, illustrating the different bifurcation diagrams that may arise
from the lack of unique ergodicity on the hull and showing how critical transitions can result from an
underlying jump bifurcation point. The framework has potential applications in climate dynamics,
ecology, circuits, optics, biology, and other models involving abrupt transitions.
\end{abstract}
.
\renewcommand{\subjclassname}{\textup{2020} Mathematics Subject Classification}

\maketitle
\section{Introduction}
Nonautonomous bifurcation theory is a young branch of mathematics that is attracting
growing interest within the scientific community. In fact, most of the significant advances
in this field have occurred in the 21st century.
Even in the scalar case, formulating an appropriate notion of bifurcation for nonautonomous ODEs is
not straightforward. In a scalar autonomous problem, bifurcation theory usually follows equilibria.
A general nonautonomous equation, however, need not possess constant solutions, and no single
trajectory is necessarily distinguished from the others. One of the goals of this article
is to explore in greater depth the idea (already outlined in \cite{dno1,dno2,dno5}) that a change
in the number of uniformly separated hyperbolic solutions is an extension of the autonomous concept
that is perfectly suited to the time-recurrent context. This forms the basis of the concept of
{\em loss-of-hyperbolicity bifurcation}, which appears in the title of the article.

The recent interest in nonautonomous bifurcation theory is intrinsically linked to another burgeoning
branch of scientific analysis: the study of critical transitions.
In fact, vague descriptions of what constitutes a bifurcation point or what a critical transition
is are often difficult to distinguish: both concepts convey the idea of a significant change in
the overall dynamics of a given dynamical system due to a minuscule change in the law that governs
its evolution. Not all critical transitions studied can be explained by an underlying bifurcation
phenomenon, and not all such phenomena involve the occurrence of critical transitions.
Another purpose of this article is to examine this relationship in greater depth,
defining {\em jump bifurcations\/} as those suitable to give rise to a critical transition phenomenon.

We focus on one-parameter bifurcation problems of the type
\begin{equation}\label{eq:1process}
 x'=f(t,x)+\lb, \qquad\lb\in\R\,,
\end{equation}
where the scalar map $f$ is assumed to be coercive with respect to the state variable $x$.
It is also assumed that the derivative of $f$ with respect to $x$ is concave in a weak sense to be
specified below: $f$ is what we refer to as a {\em d-concave} function.

In addition, we will assume suitable regularity
and time-compactness conditions ensuring that the family of time translates of $f$ has compact closure.
More precisely, let $(f{\cdot}s)(t,x):=f(t+s,x)$ and let $\W:=\mathrm{cls}\{f{\cdot}s\mid\,s\in\R\}$ be
the {\em hull\/} of $f$ in the compact-open topology of $C(\R\times\R,\R)$.
Under our hypotheses, $\W$ is a compact metric space, and the time translation defines a continuous
flow $\sigma\colon\R\times\W\to\W$, $(t,\w)\mapsto\w{\cdot}t$.
By setting $\mf(\w,x):=\w(0,x)$, the problem \eqref{eq:1process} is embedded into the family
\begin{equation}\label{eq:1hull}
 x'=\mf(\w{\cdot}t,x)+\lb\,,\qquad\w\in\W,\quad\lb\in\R\,:
\end{equation}
the original equation corresponds to the distinguished element $\w=f$.
This formulation replaces the study of a single nonautonomous process with that of a skewproduct flow
$\tau_\lb$ on $\WR$, making the topological and ergodic methods of nonautonomous dynamical systems natural
tools for the analysis. The flow $\tau_\lb$ preserves the flow on the base $\W$,
and is given on the fiber by the solutions of the equations corresponding to each element of the hull.

Since $\W$ is the closure of the orbit of $f$, the base flow is {\em transitive\/} by construction: the orbit
of $f$ is dense in $\W$. The main results of this paper will be formulated for families \eqref{eq:1hull}
assuming transitivity of the flow on $\W$. This is fully consistent with the main
purpose of the hull framework, namely, the analysis of the variation with respect to $\lambda$
of the processes generated by a parametric family of time-dependent equations like \eqref{eq:1process}.

In this hull framework, uniformly separated hyperbolic solutions are replaced by disjoint
(and hence strictly ordered)
{\em hyperbolic copies of the base}. A copy of the base is the graph of a continuous map $\mc\colon\W\to\R$
such that $t\mapsto \mc(\wt)$ solves \eqref{eq:1hull} for all $\w\in\W$, and it is hyperbolic if
this graph is uniformly exponentially stable for $\tau_\lb$ either as time increases (when it is
attractive) or as time decreases (if it is repulsive).
The persistence of hyperbolicity guarantees that a hyperbolic solution for a given value
of the parameter $\lb$ is part of a branch of such solutions that varies with $\lambda$.
Roughly speaking, $\lb_0$ is said to be a {\em loss-of-hyperbolicity bifurcation point} if
it is the endpoint of an interval of parameter values corresponding to that branch.

The foundations of this work can be found in \cite{dno1,dno5}.
The present paper advances the previously developed theory with a more complete classification
of the different possibilities for the bifurcation diagrams. The conditions we assume guarantee the
existence of the global attractor $\mA_\lb$ for $\tau_\lb$ for all $\lb\in\R$, and that it reduces to a single
hyperbolic attractive copy of the base when $|\lb|$ is large enough. We call
$\lb_u^*\ge-\infty$ (resp.~$\lb_l^*\le\infty$) the lowest (resp.~highest) value of the parameter
for which the upper equilibrium $\muk_\lb$ (resp.~lower equilibrium $\ml_\lb$)
of $\mA_\lb$ is a hyperbolic copy of the base for all $\lb\in(\lb^*_u,\infty)$
(resp.~for all $\lb\in(-\infty,\lb^*_l)$). The
relative position of $\lb_l^*$ and $\lb_u^*$ gives rise to three different possibilities:
{\bf p1} if $\lb_u^*<\lb_l^*$, {\bf p2} if $\lb_u^*=\lb_l^*$ and
{\bf p3} if $\lb_u^*>\lb_l^*$. In {\bf p1} and {\bf p2}, we provide an exhaustive
description of all possible (loss-of-hyperbolicity) bifurcation diagrams
and the underlying dynamics for each value of $\lb$, whereas the number of possibilities
when $\lb_l^*<\lb_u^*$ makes such an exhaustive description an impossible task in case {\bf p3}.
We show that, even if $\W$ is minimal, the lack of ergodic uniqueness of its flow can give rise
to bifurcation patterns and dynamical scenarios that do not occur in autonomous, almost periodic
in time, or, in general, uniquely ergodic cases,
and provide explicit examples that give rise to some of these genuinely nonautonomous bifurcation patterns.

Nonautonomous d-concave scalar equations arise across a wide range of applied sciences.
Equations of this type, together with closely related nonautonomous bistable scalar models,
appear in climate dynamics \cite{BerglundGentz2002,Cessi1994,HankelTziperman2023}, ecological
control problems \cite{fmrh,WangChenZhengYu2024}, nonlinear electrical circuits
\cite{BaltanasEtAl2003,LuchinskyEtAl1999}, optical systems \cite{HohlEtAl1995,JungGrayRoy1990},
and biological population models \cite{LopesAmorGore2024,WackerSchluter2021}, among other areas
of scientific interest. Consequently, the results obtained in this paper may contribute to the
analysis of a broad range of problems in applied mathematics.

As already said, one of the motivations for this work comes from the mathematical theory
of critical transitions: abrupt and often difficult-to-reverse changes in the state of a
system that may be triggered by small or gradual variations in external conditions. They have been studied
in connection with climate change and large-scale climate-system transitions
\cite{AlkhayuonEtAl2019,awvc,Schellnhuber2009}, ecosystem collapse, ecological regime shifts,
and evolutionary rescue \cite{SchefferEtAl2009,VanselowEtAl2022}, biological regulation, cell-fate decisions,
and threshold phenomena in nerve excitation \cite{Hill1936,NeneZaikin2010}, and instability in economic and
financial systems \cite{MayLevinSugihara2008,YukalovEtAl2009}. This collection of works,
while by no means exhaustive, shows that critical transitions constitute a truly
multidisciplinary field of research.

A frequent framework for their mathematical study starts with a $\lambda$-family of ODEs
and replaces the parameter with a time-dependent function, a “parameter shift”
$\Lambda(t)$ \cite{alk2018,asne,apw2017,awvc,dno4,WieczorekXieAshwin2023}. This gives rise to the
{\em transition equation\/} to be analyzed. In many of these studies, as in
\cite{apw2017, WieczorekXieAshwin2023},
it is assumed that $\Lambda(t)$ is asymptotically constant,
and replacing $\Lambda(t)$ with its constant limits yields the {\em past\/} and
{\em future equations}.
The {\em frozen equations} are given by the equations
obtained by replacing $\Lambda(t)$ with each one of the values in the range of $\Lambda$, and
all of them are assumed, as in our case, to have a representative local attractor.
In this setting, a critical transition may appear when the solutions of the
transition equation connecting with that special local attractor of the past equation
stop tracking the special local attractor of the frozen equations.
More recent developments \cite{dno3,dno4,dno5} allow the limiting equations themselves to be
nonautonomous. In particular, in \cite{dno4,dno5}, the role of the past and future equations is played by
the families \eqref{eq:1hull} restricted to the elements $\w$ of the \upalfa-limit and \upomeg-limit
sets of $f$ in $\W$, respectively. Some d-concavity (or concavity) conditions assumed in a weak
(measure) sense just on the limit families permit relating the occurrence of critical transitions
with the occurrence of saddle-node bifurcation phenomena for the initial $\lambda$-family.
Previous and more restrictive approaches to this idea can be found in \cite{dno3,lno2,lnor}.

However, the bifurcation phenomena underlying critical transitions are not limited to the saddle-node type.
In this article, we refine the concept of the bifurcation point due to loss of hyperbolicity with that of
{\em jump bifurcation point\/} $\lb_0$, which adequately describes a certain abrupt change in the shape of the
global attractor $\mA_\lb$ when the parameter $\lb$ crosses $\lb_0$; or, more specifically, an abrupt change in
the upper equilibrium $\muk_\lb$ or in the lower equilibrium $\ml_\lb$ of $\mA_\lb$. Let us
briefly explain the connection
between this type of bifurcation and the emergence of critical transitions by focusing on the case of a
variation in the rate of change of the parameter: when the parameter shift takes the form
$\Lambda(\ep t)$ and $\ep$ is small,
the solutions to the transition equation track $\mA_\lb$, and some of them track the
local attractors determined by the graphs of $\muk_\lb$ and $\ml_\lb$. This behavior, which has been well
known for decades in the autonomous case thanks to the combination of the results of Tikhonov \cite{tikh}
and Fenichel \cite{feni}, has recently been demonstrated in the nonautonomous case under conditions that our
d-concave ordinary differential equations satisfy in \cite{loos}. Clearly, an abrupt variation in these
tracked elements necessarily involves a critical transition.

When the values $\lb_l^*$ and $\lb_u^*$ are real, they are loss-of-hyperbolicity bifurcation points,
but possibly not the only ones. On the contrary, under some fairly mild condition conditions,
if $\lb_0$ is a jump bifurcation point, then it coincides with one of these two (perhaps equal) special values.
In this work, we describe examples exhibiting jump bifurcation in the three cases {\bf p1}, {\bf p2} and
{\bf p3} mentioned above. The well-known cases of double saddle-node bifurcation for \eqref{eq:1hull}
provide examples of case {\bf p1} with two jump bifurcation points.
Regarding case {\bf p2}, we detail the example of a minimal compact base flow carrying two distinct ergodic
measures, based on the well-known Harper-type examples, and show how it gives rise
to a purely nonautonomous bifurcation diagram which in addition involves jump bifurcation.
An example of a cubic polynomial
equation with time-dependent coefficients fitting case {\bf p3} and a jump bifurcation
completes the analysis. The
complexity of these nonautonomous bifurcation diagrams is a direct consequence of the
complexity of the dynamics
corresponding to a transition equation giving rise to a hull on which several ergodic measures coexist.

We conclude this introduction with a brief outline of the paper. The basic concepts and results
used in the subsequent analysis are presented in Section \ref{sec:2}. Section \ref{sec:3} describes
the hypotheses imposed on our d-concave bifurcation problems and explains their motivation.
It also defines the notion of a loss-of-hyperbolicity bifurcation point and focuses on what is,
to the best of our knowledge, the best-known nontrivial bifurcation diagram in this setting:
the double saddle-node bifurcation diagram, which we refer to as {\em S-shaped}.
Section \ref{sec:4} begins with the aforementioned classification into cases {\bf p1},
{\bf p2}, and {\bf p3}. After presenting some common and case-specific properties of these situations,
and describing the possibilities for the bifurcation diagram in cases {\bf p1} and {\bf p2},
we give the definition of a jump bifurcation point and describe a critical transition,
corresponding to an S-shaped bifurcation diagram (and hence to case {\bf p1}) for
the family of frozen equations. Finally, Sections \ref{sec:5} and \ref{sec:6} provide detailed,
highly nontrivial examples of cases {\bf p3} and {\bf p2} with jump bifurcation points, and describe
related models in which these underlying bifurcations give rise to critical transitions.
\section{Preliminaries and some previous results}\label{sec:2}
The basic definitions and properties of maximal solutions of ordinary differential equations,
flows, orbits, invariant sets, \upalfa-limit and \upomeg-limit sets,
and invariant and ergodic (normalized positive Borel) measures can be found,
e.g., in \cite{cano}, \cite{dno1},
\cite{dno4}, and references therein. In this preliminary section,
we recall some less standard concepts, fundamental in the forthcoming analysis.

Let $\sigma\colon\R\times\W\to\W$, $(t,\w)\mapsto\sigma(t,\w)=:\w{\cdot}t$ be a global
continuous flow on a compact metric space $\W$.
We denote by $C^{0,2}(\WR,\R)$
the set of continuous functions $\mh\colon\WR\to\R$ for which the
first and second derivative with respect to the second
variable, $\mh_x$ and $\mh_{xx}$, exist and are continuous.
A map $\mh\in C^{0,2}(\WR,\R)$ provides a family of scalar nonautonomous differential equations
\begin{equation}\label{eq:2fam}
x'=\mh(\w{\cdot}t,x)\,,\quad\w\in\W\,,
\end{equation}
and we denote by $v(t,\w,x)$ the maximal solution of the equation corresponding to $\w$
given by $v(0,\w,x)=x$. So, $v(t+s,\w,x)=v(t,\ws,v(s,\w,x))$ whenever the right-hand term is defined,
and hence there exists a maximal open subset $\mV\subseteq\R\times\WR$ containing
$\{0\}\times\WR$ such that
\begin{equation}\label{def:2tau}
 \tau\colon\mV\to\WR\,,\;\;(t,\w,x)\mapsto (\wt,v(t,\w,x))
\end{equation}
defines a (possibly local) continuous flow on $\WR$, of {\em skewproduct} type.
This framework is a usual tool to analyze a single nonautonomous equation: the {\em hull construction},
briefly described in Section \ref{subsec:2hull},
provides a family of the type \eqref{eq:2fam} including the initial one.

We will often restrict the family \eqref{eq:2fam} to the elements of a $\sigma$-invariant
compact subset $\bar\W\subseteq\W$. It is easy to check that
$C^{0,2}(\WR,\R)\hookrightarrow C^{0,2}(\bar\W\times\R,\R)$, so
that it makes perfect sense to define $\tau$ on $\bar\W\times\R$ as in \eqref{def:2tau}.
In some of these cases, $\bar\W$ will be the closure $\W_\w$
of the $\sigma$-orbit $\{\wt\,|\;t\in\R\}$ of a point $\w\in\W$. Recall that the
base flow $(\W,\sigma)$ is {\em transitive\/} if there exists $\w\in\W$ with $\W_{\w}=\W$,
in which case this happens for a residual subset of points: see, e.g., \cite[Theorem 9.20]{gohe}.
The set $\W$ is {\em minimal\/} if $\W_\w=\W$ for all $\w\in\W$. In addition, a
set $\mM\subseteq\W$ (resp.~$\mK\subset\WR$) is {\em minimal} if it is compact, $\sigma$-invariant
(resp.~$\tau$-invariant), and does not contain properly any other such set.

We denote by $\merg$ the nonemtpy set of $\sigma$-ergodic measures on $\W$.
\subsection{Equilibria, semiequilibria, and upper and lower solutions}
Let $\mb\colon\W\to\R$ be a map, and assume that $v(t,\w,\mb(\w))$
exists for all $(t,\w)\in\R\times\W$. If $\mb(\wt)\le v(t,\w,\mb(\w))$
(resp.~$\mb(\wt)\ge v(t,\w,\mb(\w))$) for all $t\ge0$ and $\w\in\W$, then $\mb$ is a
$\tau$-{\em subequilibrium} (resp.~$\tau$-{\em superequilibrium});
and it is {\em strong\/} if, in addition, there exists a time $s_*>0$ such that $\mb(\ws_*)<v(s_*,\w,\mb(\w))$
(resp.~$\mb(\ws_*)>v(s_*,\w,\mb(\w))$) for all $\w\in\W$.
A {\em $\tau$-equilibrium} is a map $\mb\colon\W\to\R$ with $\tau$-invariant graph; i.e.,
with $\mb(\wt)= v(t,\w,\mb(\w))$ for all $(t,\w)\in\R\times\W$.
The graph of a continuous $\tau$-equilibrium $\mb$, represented by $\{\mb\}$,
is a {\em $\tau$-copy of the base\/} or {\em $\tau$-copy of $\W$}.
The prefix $\tau$ will be often omitted.

The concepts of superequilibria and subequilibria are strongly related to those of global
upper and lower solutions.
A {\em bounded global lower solution\/} of the family \eqref{eq:2fam} is a bounded map
$\mb\colon\W\to\R$ such that $t\mapsto\mb(\wt)$ is $C^1$ and $\mb'(\w)\le\mh(\w,\mb(\w))$ for all
$\w\in\W$, where $\mb'(\w)=(d/dt)\,\mb(\wt)|_{t=0}$, and it is {\em strict} if the inequality is
strict for all $\w\in\W$. By changing the sign of the inequalities, we obtain the definition of
({\em strict}) {\em bounded global upper solution\/}.
If every forward $\tau$-semiorbit is globally defined,
then $\mb$ is a superequilibrium (resp.~subequilibrium)
if and only if it is a global upper (resp.~lower) solution, and
it is a strong superequilibrium (resp.~subequilibrium) if it is
a strict global upper (resp.~lower) solution (see Sections 3 and 4 of \cite{nono}).
\subsection{Some properties of compact $\tau$-invariant sets}\label{subsec:22}
Let $\mB\subset\WR$ be a bounded $\tau$-invariant set, and let us
denote $\mB_\w:=\{x\in\R\,|\;(\w,x)\in\mB\}$ and call it
{\em the fiber of\/} $\mB$ over $\w$.
If $\mB$ {\em projects onto $\W$}
(i.e., if $\mB_\w$ is nonempty for all $\w\in\W$),
then the maps $\ml_\mB(\w):=\inf\mB_\w$ and
$\muk_\mB(\w):=\sup\mB_\w$ define $\tau$-equilibria, called
the {\em lower\/} and {\em upper equilibria of $\mB$}. If $\mK\subset\WR$ is compact, then
$\ml_\mK$ and $\muk_\mK$ are respectively lower and upper semicontinuous, and hence
they are $m$-measurable for any ergodic measure in $\W$, as well as continuous at the
points of a $\sigma$-invariant residual subset of $\W$.

Two compact subsets $\mK_1$ and $\mK_2$ of $\W\times\R$ are {\em ordered\/} with
$\mK_1<\mK_2$ or $\mK_2>\mK_1$ if $x_1<x_2$ whenever there exists $\w\in\W$ such that
$(\w,x_1)\in\mK_1$ and $(\w,x_2)\in\mK_2$. In this case,
we say that $\mK_1$ {\em is below\/} $\mK_2$ or that $\mK_2$ {\em is above\/} $\mK_1$.
The compact sets
$\mK_1$ and $\mK_2$ {\em are (at least) at a positive distance $\delta>0$} if
$x_1+\delta<x_2$ whenever $(\w,x_1)\in\mK_1$ and $(\w,x_2)\in\mK_2$.
Two (ordered) bounded equilibria $\mb_1,\mb_2$ with $\mb_1<\mb_2$
(i.e., with $\mb_1(\w)<\mb_2(\w)$ for all $\w\in\W$) are {\em uniformly separated\/} if
$\inf_{\w\in\W}(\mb_2(\w)-\mb_1(\w))>0$.

\begin{defi}\label{def:unifexpstable}
A compact $\tau$-invariant set $\mK\subset\W\times\R$ which projects onto $\W$
is \emph{uniformly exponentially stable at $+\infty$} (resp. at $-\infty$) if there
exist a \emph{radius of uniform stability} $\rho>0$ and
a \emph{dichotomy constant pair} $(k,\gamma)$ with $k\ge1$ and $\gamma>0$
such that, if $(\w,x)\in\mK$ and $(\w,y)\in\W\times\R$ satisfy $|x-y|<\rho$,
then $v(t,\w,y)$ is defined for all $t\ge0$ (resp. $t\le0$) and
\[
\begin{split}
|v(t,\w,x)-v(t,\w,y)|&\le k\,e^{-\gamma\,t}|x-y|\quad\text{for all }t\ge0\,,\\
(\text{resp. }|v(t,\w,x)-v(t,\w,y)|&\le\, k\,e^{\,\gamma\,t}\,|x-y|\quad\text{for all }t\le0)\,.
\end{split}
\]
\end{defi}
\begin{defi}\label{def:hcotb}
A compact $\tau$-invariant set is {\em hyperbolic attractive\/} (resp.~{\em hyperbolic repulsive})
if it is uniformly exponentially stable at $+\infty$ (resp.~at $-\infty$), and
is {\em nonhyperbolic\/} if it is neither hyperbolic attractive nor hyperbolic repulsive.
\end{defi}
We do not recall here the
well-known definition and properties of the {\em Lyapunov exponents\/} of a compact $\tau$-invariant
set: the reader can find in \cite[Section 2.1]{dno5} a summary of these properties,
which we will use. In particular, if a compact $\tau$-invariant set is hyperbolic attractive
(resp.~repulsive), then its upper (resp.~lower) Lyapunov exponent is strictly  negative
(resp.~positive). In this case, it is ``often" a copy of the base.
The first of the following two nontrivial results, which are essential
in the proofs of the main results of Section \ref{sec:6},
illustrates the scope of this assertion:
\begin{teor}\label{th:3copia}
Let $\mh\in C^{0,1}(\WR,\R)$.
Let $\mK\subset\WR$ be a compact $\tau$-invariant set projecting onto $\W$.
Assume that one of the following conditions holds:
\begin{itemize}[leftmargin=20pt]
\item[\rm (a)] $\ml_\mK(\w)$ and $\muk_\mK(\w)$ coincide (at least) at a point $\w$ of each
minimal subset $\mM\subseteq\W$;
\item[\rm (b)] $\mK=\{(\w,x)\in\WR\,|\;\ml_\mK(\w)\le x\le\muk_\mK(\w)\}$.
\end{itemize}
Then, $\mK$ is hyperbolic if and only if it is a hyperbolic copy of $\W$.
More precisely, the upper (resp. lower) Lyapunov
exponent of $\mK$ is strictly negative (resp.~positive) if and only if $\mK$ is
an attractive (resp. repulsive) hyperbolic copy of $\W$.
\par
In addition, if either $\mK$ (and hence $\W$) is minimal or its upper and
lower equilibria coincide on a $\tau$-invariant subset $\W_0\subseteq\W$
with $m(\W_0)=1$ for all $m\in\merg$, then condition {\rm (a)} holds.
\end{teor}
\begin{proof}
Observe that the ``if" condition is trivial.
Under condition (a), the characterization and the last assertion
are the statements of \cite[Theorem 2.8]{dno4},
based on \cite[Proposition 2.8]{cano}.
Let us check that, if (b) holds and the upper Lyapunov exponent of $\mK$ is strictly negative,
then (a) holds.  Given a minimal subset $\mM\subseteq\W$, we define
$\mK^\mM:=\{(\w,x)\in\mK\,|\;\w\in\mM\}=\{(\w,x)\in\mM\times\R\,|\;\ml_\mK(\w)\le x\le\muk_\mK(\w)\}$.
Any $\sigma$-invariant measure $m_\mM$ on $\mM$ can be extended in a trivial way to a
$\sigma$-invariant measure $m_\W$ on $\W$ concentrated on $\mM$. Hence, the upper Lyapunov
exponent of $\mK^\mM$ is
strictly negative: see, e.g., \cite[Section 2.1]{dno5}.
The minimality of $\mM$ and the connectedness of $\mK^\mM$ allow us to repeat
the arguments of the proof of \cite[Theorem 3.4]{cano}, leading us to conclude
that $\mK^\mM$ is the graph of a continuous map on $\mM$, and clearly this ensures (a).
The case of strictly positive lower Lyapunov exponent can be treated in the same way.
\end{proof}
\begin{prop}\label{prop:2bocata}
Let $\mb_1\colon\W\to\R$ be a lower (resp.~upper) semicontinuous $\tau$-subequilibrium and
let $\mb_2\colon\W\to\R$ be an upper (resp.~lower) semicontinuous $\tau$-superequilibrium with
$\mb_1\le\mb_2$ (resp.~$\mb_1\ge\mb_2$).
Then, there exists a maximal $\tau$-invariant
compact set $\mK$ contained between the graphs of $\mb_1$ and $\mb_2$
which, in addition, projects onto $\W$ and
takes the form of Theorem {\rm\ref{th:3copia}(b)}.
Moreover, $\mK$ does not intersect the graph of $\mb_i$ if $\mb_i$
is strong;
and, if $\mb_1\le\mb_2-\delta$ (resp.~$\mb_1\ge\mb_2+\delta$) for a constant $\delta>0$
and any of the two semiequilibria, $\mb_i$, is a repulsive (resp.~attractive)
hyperbolic copy of $\W$, then $\mK$ contains a $\tau$-invariant compact set $\tilde\mK$
which does not intersect the graph of $\mb_i$.
\end{prop}
\begin{proof}
Let us sketch the proof in the less standard case $\mb_1\ge\mb_2$. All the involved ideas
are detailed in \cite[Sections 3 and 4]{nono}.
The results there are
formulated for a minimal base flow, but this property is not used in the proofs.
We check that: the region
delimited by the graphs is negatively time-invariant; for $s<0$, the maps
\begin{equation}\label{def:3b}
 (\mb_i)_s(\w):=v(s,\w{\cdot}(-s),\mb_i(\w{\cdot}(-s)))
\end{equation}
are also $\tau$-semiequilibria inheriting the semicontinuity property of the generators
$\mb_i$,
satisfy $\mb_2\le(\mb_2)_s\le(\mb_1)_s\le\mb_1$ for all $s<0$,
and behave monotonically with respect to $s$ (decreasing as $s$ decreases in the case
of $(\mb_1)_s$, increasing as $s$ decreases in the case of $(\mb_2)_s$); and
their limits $\mc_i:=\lim_{s\to-\infty}(\mb_i)_s$
define two $\tau$-equilibria with $\mb_2\le\mc_2\le\mc_1\le\mb_1$,
and with $\mc_1$ upper semicontinuous and $\mc_2$ lower semicontinuous. So,
$\mK=\{(\w,x)\,|\;\mc_2(\w)\le x\le\mc_1(\w)\}$ is a compact
$\tau$-invariant set projecting onto $\W$ and taking the form of
Theorem {\rm\ref{th:3copia}(b)}.

Let us show the maximality of $\mK$, i.e,
that it contains any other compact $\tau$-invariant set $\mM$
lying between the graphs of $\mb_1$ and $\mb_2$: if,
$\mb_1(\w)\ge\muk_\mM(\w)$ for all $\w\in\W$, then
$(\mb_1)_s(\w)\ge v(s,\w{\cdot}(-s),\muk_\mM(\w{\cdot}(-s)))=\muk_\mM(\w)$, and hence
$\mc_1(\w)\ge \muk_\mM(\w)$; and, analogously, $\mc_2(\w)\le\ml_\mM(\w)$.

If, in addition,
$\mb_1$ (resp.~$\mb_2$) is strong, then there exist $s_0<0$ and $\delta>0$
such that $(\mb_1)_{s_0}\le\mb_1-\delta$ (resp.~$(\mb_2)_{s_0}\ge\mb_2+\delta$),
which proves that $\mK$ does not intersect the graph of $\mb_1$ (resp.~$\mb_2$).

We will prove the last assertion in the case that $\mb_1$ is a hyperbolic attractive
hyperbolic copy of $\W$ with radius of uniform stability $\rho$. We assume for contradiction that
there exists $s<0$ and $\w\in\W$ such that $(\mb_2)_s(\w)\ge\mb_1(\w)-\rho$, which
combined with $(\mb_2)_s(\w)\le\mb_1(\w)$ ensures that
$\lim_{t\to\infty}(\mb_1(\wt)-v(t,\w,(\mb_2)_s(\w)))=0$. But
$v(t,\w,(\mb_2)_s(\w))=v(t+s,\w{\cdot}(-s),\mb_2(\w{\cdot}(-s)))\le \mb_2(\wt)$ for $t\ge-s$,
and hence $\mb_1(\wt)-v(t,\w,(\mb_2)_s(\w))\ge\mb_1(\wt)-\mb_2(\wt)\ge\delta$, which
provides the contradiction. Hence $\mc_2\le\mb_1-\rho$,
and the closure of its graph is the set $\tilde\mK$ of the statement.
%
%
\end{proof}
Finally, if a compact $\tau$-invariant set $\mA\subset\WR$ satisfies
$\lim_{t\to\infty} \text{dist}(\mC{\cdot}t,\mA)=0$ for every bounded set $\mC$,
where $\mC{\cdot}t=\{(\wt,v(t,\w,x))\,|\;(\w,x)\in\mC\}$ and
\[
\text{dist}(\mC_1,\mC_2)=\sup_{(\w_1,x_1)\in\mC_1}\left(\inf_{(\w_2,x_2)\in\mC_2}
\big(\mathrm{dist}_{\WR}((\w_1,x_1),(\w_2,x_2))\big)\right)\,,
\]
then $\mA$ is the (unique) {\em global attractor for $\tau$}.
\subsection{The hull construction}\label{subsec:2hull}
A continuous map $h\colon\RR\to\R$ is {\em admissible\/} if the restriction of
$h$ to $\R\times\mJ$ is bounded and uniformly continuous for any compact set $\mJ\subset\R$.
Let the map $h\colon\RR\to\R$ be admissible and have admissible first and second derivatives
with respect to its second variable, $h_x,\,h_{xx}\colon\RR\to\R$. We define $h{\cdot}t(s,x):=h(t+s,x)$.
The {\em hull\/} $\W_h$ of $h$ is the closure of
$\{h{\cdot}t\,|\;t\in\R\}$ on the set $C(\RR,\R)$ provided with the compact-open
topology. The set $\W_h$ is a compact metric space,
the time-shift map $\sigma_h\colon\R\times\W_h\to\W_h,(t,\w)\mapsto\w{\cdot}t$
defines a global continuous flow, and the map $\mh$ given by $\mh(\w,x)=\w(0,x)$
is continuous on $\W_h\times\R$ and has continuous first and second derivative
with respect to its second variable, $\mh_x,\,\mh_{xx}\colon\W_h\times\R\to\R$.
The proof of these properties can be found in \cite[Theorem~I.3.1]{shyi4} and
\cite[Theorem IV.3]{sell2}.
Note that $(\W_h,\sigma_h)$ is a transitive flow, since the $\sigma_h$-orbit of
$h\in\W_h$ is dense in $\W_h$.
More precisely, if $\W_h^\upalpha$ and $\W_h^\upomega$ are the \upalfa-limit set
and \upomeg-limit set of the element $h\in\W_h$, then
$\W_h=\W_h^\upalpha\cup\{h{\cdot}t\,|\;t\in\R\}\cup\W_h^\upomega$, as proved in \cite[Lemma 2.4]{dno4}.
The map $h$ is {\em recurrent} if $(\W_h,\sigma_h)$ is a {\em minimal flow}, i.e., if every
$\sigma_h$-orbit is dense in~$\W_h$.

Observe that the initial equation $x'=h(t,x)$ is included in the family $x'=\mh(\wt,x)$:
it corresponds to $\w=h$. Thus, a global description of the dynamics of the associated skewproduct
flow yields substantial information about the dynamics of the process induced by the equation.
\section{D-concave framework and loss-of-hyperbolicity bifurcation points}\label{sec:3}
Let $(\W,\sigma)$ be a global continuous flow on a connected compact metric space.
The connectedness is assumed throughout the paper just to slightly simplify the
dynamical description and the comprehension of the results: the compactness and connectedness
of $\W$ ensure that two copies of $\W$
are disjoint if and only if they are strictly ordered,
and if and only if the continuous equilibria which determine them are uniformly separated.

Our first goal in this section is to provide a suitable concept of bifurcation point
for a d-concave $\lb$-parametric family
\begin{equation}\label{eq:3hlb}
 x'=\mh(\wt,x)+\lb\,,\quad\w\in\W
\end{equation}
under certain assumptions about $\mh$, which are described below. The second objective,
addressed in Subsection \ref{subsec:31}, is to recall the description of one of the
possible global bifurcation diagrams, referred to as ``S-shaped'',
and to delve deeper into the dynamical circumstances that give rise to it.
This type of bifurcation diagram and its causes
are particularly relevant to the rest of the article.

The hypotheses on $\mh\colon\WR\to\R$ are
\begin{enumerate}[leftmargin=30pt,label=\rm{\bf{d\arabic*}}]
\item\label{d1} $\mh\in C^{0,2}(\WR,\R)$,
\item\label{d2} $\lim_{x\to\pm\infty}(\pm\mh(\w,x))=-\infty$ uniformly on $\W$,
\item\label{d3} $m(\{\w\in\W\,|\;x\mapsto \mh_x(\w,x) \text{ is concave}\})=1$
for all $m\in\merg$,
\item\label{d4} $m(\{\w\in\W\,|\; x\mapsto \mh_{xx}(\w,x)$ is strictly decreasing
on $\mJ\})>0$ for all compact interval $\mJ\subset\R$ and all $m\in\merg$,
\end{enumerate}
and they are obviously inherited by the map $\mh+\lb\colon\WR\to\R$ for all $\lb\in\R$.
The concept of bifurcation point must capture the idea of a significant change in
the global dynamics induced on $\W\times\R$ by the flow $\tau_\lb$
determined by the family.

Conditions \ref{d1} and \ref{d3} are the framework of what we call the
{\em d-concave in measure case}, \ref{d4} is a hypothesis on local strict
d-concavity which precludes the {\em quadratic case\/} $\mh(\w,x)=
\mc(\w)\,x^2+\md(\w)\,x+\mathfrak e(\w)$, and \ref{d2} ensures
(or, more precisely, is equivalent to) the classical coercivity property
``$\limsup_{x\to\pm\infty}(\pm\mh(\w,x)+\lb)<0$ uniformly on $\W$" for all $\lb\in\R$,
as proved in \cite[Proposition 3.3]{duen}.

\begin{notas}\label{rm:3pre}
1.~Conditions \ref{d1} and \ref{d2} ensure the existence of the global
attractor $\mA_\lb$ of the skewproduct flow $\tau_\lb$. This assertion is proved in
\cite[Proposition 5.5]{dno4}, which also shows that any forward $\tau_\lb$-semiorbit
is globally defined and bounded, and that $\mA_\lb$
is the set of all the globally bounded $\tau$-orbits. As said in Section \ref{subsec:22},
its upper and lower equilibria (i.e.,
the maps $\ml_\lb\colon\W\to\R$ and $\muk_\lb\colon\W\to\R$ for which
\begin{equation}\label{def:3ul}
 \mA_\lb=\{(\w,x)\,|\;\w\in\W\text{ and }\ml_\lb(\w)\le x\le\muk_\lb(\w)\})
\end{equation}
are respectively lower and upper semicontinuous $\tau_\lb$-equilibria.
In addition, the map $\R\to\R,\,\lb\mapsto\muk_\lb(\w)$ (resp.~$\R\to\R,\,\lb\mapsto\ml_\lb(\w)$)
is right-continuous (resp.~left-continuous) and strictly increasing for all $\w\in\W$: see
\cite[Proposition 4.2]{dno5}.
\par
2.~Let us assume \ref{d1} and \ref{d2} and define
\[
\begin{split}
 \mK^l_\lb&:=\mathrm{closure}_{\WR}\{(\w,\ml_\lb(\w))\,|\;\w\in\W\}\,,\\
 \mK^u_\lb&:=\mathrm{closure}_{\WR}\{(\w,\muk_\lb(\w))\,|\;\w\in\W\}\,.
\end{split}
\]
It is easy to check that they are $\tau_\lb$-invariant compact sets, and that
$(\mK^l_\lb)_\w=\{\ml_\lb(\w)\}$ (resp.~$(\mK^u_\lb)_\w=\{\muk_\lb(\w)\}$)
at those points $\w$ of the residual subset of continuity points of $\ml_\lb$
(resp.~$\muk_\lb$). So, none of these sets can contain two disjoint
compact $\tau_\lb$-invariant sets projecting onto $\W$.
\par
3.~On the other hand, \cite[Theorem 5.3 and Proposition 5.5]{dno4} show
that, under conditions \ref{d1}, \ref{d2}, \ref{d3} and \ref{d4}, there exist
at most three disjoint and ordered $\tau$-invariant compact sets $\mK_1<\mK_2<\mK_3$
projecting onto $\W$, and that whenever all three exist they are hyperbolic copies of $\W$
with $\mK_1=\{\ml_\lb\}$ and $\mK_3=\{\muk_\lb\}$ attractive, and with $\mK_2$ repulsive
and given by the graph of a continuous $\tau_\lb$-equilibrium $\mm_\lb$,
which separates the domains of attraction of the upper and lower copies of $\W$.
\end{notas}

As explained in the Introduction, providing a definition of bifurcation point $\lb_0$ valid
for a one-parametric family of scalar nonautonomous dynamic equations
is not a trivial task. Our approach in this work and in previous ones suggests that
what is presented below, in Definition \ref{def:3bifloss}, is a promising candidate
(not only for d-concave equations). That definition requires
Theorem~\ref{th:3persistence}, which
recalls a well-known result on the persistence of hyperbolic
copies of $\W$ for which the interested
reader can find a detailed in proof in \cite[Theorem 1.39]{duen}.
Both Definition \ref{def:3bifloss} and Theorem~\ref{th:3persistence} require Definition
\ref{def:3branch}. From now on, when talking about a {\em continuous\/} map
$\mI\to C(\W,\R)$ for $\mI\subseteq\R$,
we refer to the topology of uniform convergence on $C(\W,\R)$.
\begin{defi}\label{def:3branch}
Let $\mI\subseteq\R$ be an interval. A continuous map $\mI\to C(\W,\R),\,\lb\mapsto\mb_\lb$
such that $\mb_\lb$ is an attractive (resp.~repulsive) hyperbolic $\tau_\lb$-copy of $\W$
for all $\lb\in\mI$ is a {\em continuous branch of attractive\/} (resp.~{\em repulsive})
{\em copies of $\W$} for the family \eqref{eq:3hlb}.
\end{defi}
\begin{teor}\label{th:3persistence}
Let us assume that $\mh\in C^{0,1}(\WR,\R)$.
If there exists an attractive (resp.~repulsive) $\tau_{\bar\lb}$-copy of $\W$
for $\bar\lb\in\R$, $\{\bar\mb_{\bar\lb}\}$, then there exists $\delta>0$, $\rho>0$,
and a branch of attractive (resp.~repulsive) copies of $\W$,
$(\bar\lb-\delta,\bar\lb+\delta)\to C(\W,\R),\,\lb\mapsto\mb_\lb$, with $\mb_{\bar\lb}=
\bar\mb_{\bar\lb}$, such that $\rho$ is a radius of uniform stability for $\mb_\lb$ for
all $\lb\in(\bar\lb-\delta,\bar\lb+\delta)$.
\end{teor}
In particular, if there is a continuous branch
of hyperbolic copies of $\W$ on an interval $\mI$, and if $\mI$ contains
any of its endpoints, then the branch can be extended through this endpoint.
\begin{defi}\label{def:3bifloss}
A {\em bifurcation due loss of hyperbolicity\/} or a {\em loss-of-hyperbolicity
bifurcation} occurs at $\lb_0$,
or {\em $\lb_0$ is a loss-of-hyperbolicity bifurcation point},
if it is an endpoint of an interval $\mI$ on which there exists
a branch of hyperbolic copies of $\W$ that cannot be extended
through $\lb_0$.
\end{defi}

\subsection{S-shaped bifurcation diagrams}\label{subsec:31}
For the reader's convenience, we reproduce here the information of \cite[Theorem 4.4]{dno5},
which describes the global bifurcation diagram in the case of existence of a value
$\lb_0$ of the parameter for which the situation is that described in Remark \ref{rm:3pre}.3.
This result is not only interesting for the information it provides, but also serves as a
key tool for proving the results in the following sections.
In particular, in the described situation,
there are at least two loss-of-hyperbolicity bifurcation points for the family
\eqref{eq:3hlb} (see Definition \ref{def:3bifloss}),
which are denoted in the theorem as $\lb_-$ and $\lb^+$,
and which are the unique ones if $\lb_\diamond=\lb_-$ and
$\lb^\diamond=\lb^+$.
\begin{teor}\label{th:3Dbifur}
Let $\mh\colon\WR\to\R$ satisfy \ref{d1}, \ref{d2}, \ref{d3} and \ref{d4}.
Assume that there exists $\lb_0\in\R$ such that there are three disjoint hyperbolic
$\tau_{\lb_0}$-copies of $\W$. Then, there exists a bounded interval $\mI=(\lb_-,\lb^+)$ with
$\lb_0\in \mI$ and two real values $\lb_\diamond$ and $\lb^\diamond$ with $\lb_\diamond\le\lb_-$ and
$\lb^+\le\lb^\diamond$ such that
\begin{itemize}[leftmargin=20pt]
\item[\rm(i)] for every $\lb\in\mI$, there are three hyperbolic $\tau_\lb$-copies of $\W$, $\{\ml_\lb\}<
    \{\mm_{\lb}\}<\{\muk_{\lb}\}$, with
    $\{\ml_\lb\}$ and $\{\muk_\lb\}$ attractive and $\{\mm_\lb\}$ repulsive. In addition, the maps
    $\mI\to C(\W,\R),\,\lb\mapsto\ml_\lb,-\mm_\lb,\muk_\lb$ are continuous and strictly increasing on $\mI$.
    In particular, $\mm_{\lb_-}(\w):=\lim_{\lb\to(\lb_-)^+}\mm_\lb(\w)$ and
    $\mm_{\lb^+}(\w):=\lim_{\lb\to(\lb^+)^-}\mm_\lb(\w)$ exist for all $\w\in\W$
    and determine (lower and upper) semicontinuous equilibria for the corresponding flows.
\item[\rm(ii)] $\{\ml_\lb\}$ is an attractive hyperbolic $\tau_\lb$-copy of $\W$
    for all $\lb<\lb^+$ and not for $\lb=\lb^+$, the map $(-\infty,\lb_+)\to
    C(\W,\R),\,\lb\mapsto\ml_\lb$ is continuous and strictly increasing,
    $\lim_{\lb\to-\infty}\ml_\lb=
    -\infty$ uniformly on $\W$, $\lim_{\lb\to(\lb^+)^-}\ml_\lb=\ml_{\lb^+}$ pointwise on $\W$,
    and $\inf_{\w\in\W}\big(\mm_{\lb^+}(\w)-\ml_{\lb^+}(\w)\big)=0$.
    Analogously, $\{\muk_\lb\}$ is an attractive hyperbolic $\tau_\lb$-copy
    of the base for all $\lb>\lb_-$ and not for $\lb=\lb_-$,
    $(\lb_-,\infty)\to C(\W,\R),\,\lb\mapsto\muk_\lb$ is continuous and strictly increasing,
    $\lim_{\lb\to\infty}\muk_\lb=\infty$ uniformly on $\W$,
    $\lim_{\lb\to(\lb_-)^+}\muk_\lb=\muk_{\lb_-}$ pointwise on $\W$, and
    $\inf_{\w\in\W}\big(\muk_{\lb_-}(\w)-\mm_{\lb_-}(\w)\big)=0$.
\item[\rm(iii)] $\mA_{\lb}$ is a hyperbolic copy of $\W$ if and only if $\lb\notin[\lb_\diamond,\lb^\diamond]$.
\item[\rm(iv)] If $\lb_\diamond<\lb_-$ and $\lb\in[\lb_\diamond,\lb_-)$, or if $\lb^+<\lb^\diamond$
    and $\lb\in(\lb^+,\lb^\diamond]$,
    then $\inf_{\w\in\W}\big(\muk_\lb(\w)-\ml_{\lb}(\w)\big)=0$
    and there exists at least a
    measure $m_0\in\merg$ such that $m_0(\{\w\in\W\,|\;\ml_\lb(\w)<\muk_\lb(\w)\})=1$.
\item[\rm(v)] If $\W_\w$ is the closure of the $\sigma$-orbit of $\w\in\W$, and
    $\mI_\w=((\lb_-)_\w,(\lb^+)_\w)$ is the interval associated by {\rm (i)} to
    the restriction of the $\lb$-parametric family \eqref{eq:3hlb} to $\W_\w$,
    then $\lb_\diamond=\inf_{\w\in\W}(\lb_-)_\w$ and $\lb_-=\sup_{\w\in\W}(\lb_-)_\w$.
    Analogously, $\lb^+=\inf_{\w\in\W}(\lb^+)_\w$ and $\lb^\diamond=\sup_{\w\in\W}(\lb^+)_\w$.
\item[\rm(vi)] If $\W$ is minimal, then $\lb_-=\lb_\diamond$ and $\lb^+=\lb^\diamond$.
\end{itemize}
\end{teor}
The information given by this theorem is partly depicted in Figure \ref{fig:S}.
Our previous assertion about $\lb_-$ and $\lb^+$ satisfying
Definition \ref{def:3bifloss} follows from (ii).
Note also that, by definition (see (v)),
$\lb_-=(\lb_-)_{\bar \w}$ and $\lb^+=(\lb^+)_{\bar \w}$ if
the $\sigma$-orbit of ${\bar \w}$ is dense in $\W$.
\begin{figure}[ht]
\includegraphics[width=0.9\textwidth]{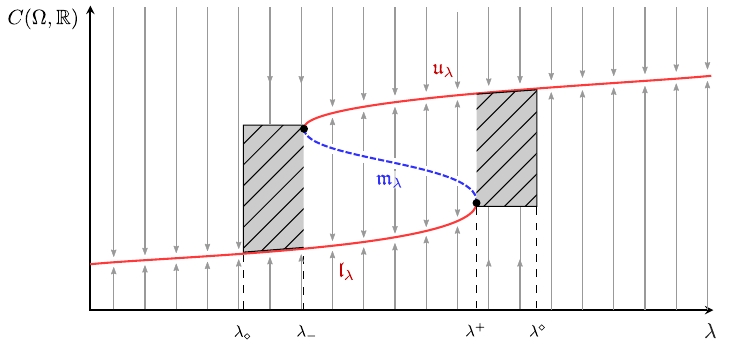}
\caption{A simple representation of the S-shaped bifurcation diagram described in Theorem \ref{th:3Dbifur}.
The solid red curves represent the two strictly increasing branches of
attractive hyperbolic copies of $\W$, $\lb\mapsto\ml_\lb,\muk_\lb$ ,
which cannot be continued through the end point of
the corresponding halfline. The dashed blue curve depicts the
the strictly decreasing branch of repulsive hyperbolic copies of $\W$, $\lb\mapsto\mm_\lb$.
At $\lb_-$ (resp.~$\lb^+$), the large black points represent the pointwise limits of $\muk_\lb$ and $\mm_\lb$
(resp.~$\ml_\lb$ and $\mm_\lb$), which are not uniformly separated and do not define hyperbolic copies of the base:
$\lb_-$ and $\lb^+$ are loss-of-hyperbolicity bifurcation points.
The drawing assumes $\lb_\diamond<\lb_-$ and $\lb^\diamond>\lb^+$ and depicts by gray regions
over $[\lb_\diamond,\lb_-)$ and $(\lb^+,\lb^\diamond]$ the special areas described in Theorem \ref{th:3Dbifur}(iv).
The grey light arrows partly depict the dynamics of the rest of orbits.}
\label{fig:S}
\end{figure}
\begin{defi}\label{def:3S}
The global bifurcation diagram described in Theorem \ref{th:3Dbifur} is called
{\em S-shaped\/} or a {\em double saddle-node global
bifurcation diagram}.
\end{defi}
We complete this section by explaining the cases that lead to the most specific hypothesis of
Theorem~\ref{th:3Dbifur}: the existence of such a value $\lb_0$ of the parameter.
\begin{teor}\label{th:3casosS}
Let $\mh\colon\WR\to\R$ satisfy \ref{d1}, \ref{d2}, \ref{d3} and \ref{d4}.
The following assertions are equivalent:
\begin{itemize}[leftmargin=20pt]
\item[\rm(a)] There exists $\lb_0\in\R$ for which there are three hyperbolic
    $\tau_{\lb_0}$-copies of $\W$. (And, if so,
    the upper and lower ones are $\{\ml_{\lb_0}\}$ and $\{\muk_{\lb_0}\}$, with domains
    of attraction separated by the middle one, $\{\mm_{\lb_0}\}$.)
\item[\rm(b)] There exists $\lb_0\in\R$ for which there is a repulsive hyperbolic
    $\tau_{\lb_0}$-copy of $\W$, $\{\mn_{\lb_0}\}$.
    (And, if so, $\mn_{\lb_0}=\mm_{\lb_0}$.)
\item[\rm(c)] There exists $\lb_0\in\R$ for which there are two disjoint
    hyperbolic $\tau_{\lb_0}$-copies of $\W$. (And, if so, there are three.)
\item[\rm(d)] There exists $\lb_-\in\R$ for which there is an attractive hyperbolic
    $\tau_{\lb_-}$-copy of $\W$ $\{\mb_{\lb_-}\}$ and a $\tau_{\lb_-}$-invariant
    compact set $\mK_-$ projecting onto $\W$  which does not contain any hyperbolic
    $\tau_{\lb_-}$-copy of $\W$ and
    with $\{\mb_{\lb_-}\}<\mK_-$. (And, if so, $\mb_{\lb_-}=\ml_{\lb_-}$,
    $\lb_-$ is the lower bifurcation point described by Theorem {\rm\ref{th:3Dbifur}},
    and $\mK_-:=\bigcup_{\w\in\W}\big(\{\w\}\times
    [\mm_{\lb_-}(\w),\muk_{\lb_-}(\w)]\big)$ satisfies the conditions.)
\item[\rm(e)] There exists $\lb^+\in\R$ for which there is an attractive hyperbolic
    $\tau_{\lb_+}$-copy of $\W$ $\{\mb_{\lb^+}\}$ and a $\tau_{\lb^+}$-invariant
    compact set $\mK^+$ projecting onto $\W$
    which does not contain any hyperbolic $\tau_{\lb^+}$-copy of $\W$ and
    with $\{\mb_{\lb^+}\}>\mK^+$. (And, if so, $\mb_{\lb^+}=\muk_{\lb^+}$,
    $\lb^+$ is the upper bifurcation point described by Theorem {\rm\ref{th:3Dbifur}},
    and $\mK^+:=\bigcup_{\w\in\W}\big(\{\w\}\times
    [\ml_{\lb^+}(\w),\mm_{\lb^+}(\w)]\big)$ satisfies the conditions.)
\end{itemize}
\end{teor}
\begin{proof}
Theorem \ref{th:3Dbifur} and Remark \ref{rm:3pre}.3 show that (a) implies (b), (c), (d) and (e).
Let us check the converse assertions.
\smallskip\par
(b)$\;\Rightarrow\;$(a) Recall that conditions \ref{d1} and \ref{d2} ensure that
any forward $\tau_\lb$-semiorbit is globally defined and bounded (see Remark \ref{rm:3pre}.1).
According to Definition \ref{def:unifexpstable},
there exists $\delta>0$ such that if $(\w,x)\in\WR$ and
$|x-\mn_{\lb_0}(\w)|\le\delta$, then also the backward semiorbit
$\{(\wt,v_{\lb_0}(t,\w,x))\,|\;t\le 0\}$ is globally defined and bounded. So,
$\ml_{\lb_0}(\w)+\delta\le\mn_{\lb_0}(\w)\le\muk_{\lb_0}(\w)-\delta$ for all $\w\in\W$. It follows
easily that the compact sets $\mK^l_{\lb_ 0}$, $\{\mn_{\lb_0}\}$ and $\mK^u_{\lb_ 0}$ (see
Remark \ref{rm:3pre}.2) are disjoint, with $\mK^l_{\lb_ 0}<\{\mn_{\lb_0}\}<\mK^u_{\lb_ 0}$. So,
(a) follows from Remark \ref{rm:3pre}.3, which also ensures that $\mn_{\lb_0}=\mm_{\lb_0}$.
\smallskip\par
(c)$\;\Rightarrow\;$(a) If one of the hyperbolic copies of $\W$ is repulsive, then (b) holds,
and therefore (a) also holds. So, we assume that both of them, say $\{\tml_{\lb_0}\}$ and
$\{\tmuk_{\lb_0}\}$, are attractive, observe that there is no restriction in assuming
that $\tml_{\lb_0}<\tmuk_{\lb_0}$, and call
$\mm_-(\w):=\sup\{x\in\R\,|\;\lim_{t\to\infty}|v_{\lb_0}(t,\w,x)-\tml_{\lb_0}(\wt)|=0\}$ and
$\mm^+(\w):=\inf\{x\in\R\,|\;\lim_{t\to\infty}|\tmuk_{\lb_0}(\wt)-v_{\lb_0}(t,\w,x)|=0\}$.
Clearly, $\tml<\mm_-\le \mm^+<\tmuk$. By reasoning as in, e.g., \cite[Lemma 4.5]{dno6},
we prove that these maps are bounded $\tau_{\lb_0}$-equilibria,
that $\mm_-$ is lower semicontinuous and that $\mm^+$ is upper semicontinuous.
So, $\mK:=\{(\w,x)\,|\;\mm_-(\w)\le x\le\,\mm^+(\w)\}$ is a compact $\tau_{\lb_0}$-invariant set.
Since $\{\tml_{\lb_0}\}<\mK<\{\tmuk_{\lb_0}\}$, Remark \ref{rm:3pre}.3 shows that (a) holds.
\smallskip\par
(d)$\;\Rightarrow\;$(a) It is easy to check that the assertions in parentheses in (d) follow
from Theorem \ref{th:3Dbifur}, once proved that the described situation ensures (a).
\par
{\sc Step 1}. Let us first prove the assertions assuming that $\mK_-$ is a
$\tau_{\lb_-}$-copy of $\W$; i.e., the graph of a continuous $\tau_{\lb_-}$-equilibrium
$\mc\colon\W\to\R$ with $\mc>\mb_{\lb_-}$.
The lack of hyperbolicity of $\{\mc\}$ will not play a role in this proof.
\par
If $\lb_0>\lb_-$, then: $\mc'(\w)<\mh(\w,\mc(\w))+\lb_0$, and hence
$\mc$ is a continuous strong $\tau_{\lb_0}$-subequilibrium (see Section \ref{sec:2});
and $\muk_{\lb_-}<\muk_{\lb_0}$ (see Remark \ref{rm:3pre}.1),
and hence $\mc\le\muk_{\lb_-}<\muk_{\lb_0}$. On the other hand, Theorem \ref{th:3persistence}
allows us to take $\lb_0>\lb_-$ close enough to guarantee the existence of
an attractive hyperbolic
$\tau_{\lb_0}$-copy of $\W$ $\{\mb_{\lb_0}\}$ close enough to $\{\mb_{\lb_-}\}$
as to ensure that it is below and at strictly positive distance
$\delta_1$ from the compact (not $\tau_{\lb_0}$-invariant) set given by the graph of $\mc$.
So, we have $\mb_{\lb_0}<\mc<\muk_{\lb_0}$.
Proposition \ref{prop:2bocata}
ensures the existence of two $\tau_{\lb_0}$-invariant compact sets $\mK_{\lb_0}^+$ and
$\mK_{\lb_0}^-$ projecting onto $\W$, with $\mK_{\lb_0}^+$ strictly above the graph of $\mc$,
and with $\mK_{\lb_0}^-$ strictly between the graphs of $\mb_{\lb_0}$ and $\mc$.
So, $\{\mb_{\lb_0}\}<\mK_{\lb_0}^-<\mK_{\lb_0}^+$, and hence
Remark \ref{rm:3pre}.3 ensures that (a) holds.
\par
{\sc Step 2}. Now, we work in the general case. For each $\lb\in\R$, we call $\hat\tau_\lb$
the skewproduct flow with base $(\mK_-,\tau_{\lb_-})$ defined on $\mK_-\times\R$ by the family of equations
\[
 x'=\hat\mh(\tau_{\lb_-}(t,\w,y),x)+\lb\,,
\]
with $\hat\mh((\w,y),x):=\mh(\w,x)$. That is,
\[
 \hat\tau_\lb(t,(\w,y),x)=(\tau_{\lb_-}(t,\w,y),\hat v_{\lb}(t,(\w,y),x))=((\wt,v_{\lb_-}(t,\w,y)),v_\lb(t,\w,x))\,.
\]
We also define $\hat\mb_{\lb_-}\colon\mK_-\to\R$ by $\hat\mb_{\lb_-}(\w,y):=\mb_{\lb_-}(\w)$.
Let us check that $\{\hat\mb_{\lb_-}\}$ is an attractive $\hat\tau_{\lb_-}$-hyperbolic
copy of $\W$: $\hat\mb_{\lb_-}$ is obviously continuous; $\hat\mb_{\lb_-}(\tau_{\lb_-}(t,\w,y))=
\mb_{\lb_-}(\wt)=v_{\lb_-}(t,\w,\mb_{\lb_-}(\w))=\hat v_{\lb_-}(t,(\w,y),\hat\mb_{\lb_-}(\w,y))$, so it is
a $\hat\tau_{\lb_-}$-equilibrium; and the equalities $|x-\hat\mb_{\lb_-}(\w,y)|=|x-\mb_{\lb_-}(\w)|$ and
$|\hat v_{\lb_-}(t,(\w,y),x)-\hat\mb_{\lb_-}(\tau_{\lb_-}(t,\w,y)|=|v_{\lb_-}(t,\w,x)-\mb_{\lb_-}(\wt)|$
allow us to deduce the uniform exponential stability at $+\infty$
of $\{\hat\mb_{\lb_-}\}$ from that of $\{\mb_{\lb_-}\}$.
It is also clear that $\hat\mK_-:=\{((\w,y),y)\,|\;(\w,y)\in\mK_-\}$ is a $\hat\tau_{\lb_-}$-copy of $\W$
and that $\hat\mK_->\{\hat\mb_{\lb_-}\}$.
So, we are in the situation of {\sc step 1}, which in particular proves that if $\lb_0>\lb_-$ is
close enough, then the upper and lower equilibria of the global $\hat\tau_{\lb_0}$-attractor,
$\{\hat\muk_{\lb_0}\}$ and $\{\hat\ml_{\lb_0}\}$ are hyperbolic $\hat\tau_{\lb_0}$-copies of
$\mK_-$. The same arguments used for $\mb_{\lb_-}$ show that
$\hat\muk_{\lb_0}(\w,y)=\muk_{\lb_0}(\w)$ and $\hat\ml_{\lb_0}(\w,y)=\ml_{\lb_0}(\w)$
for all $(\w,y)\in\mK_-$,
and that $\{\muk_{\lb_0}\}$ and $\{\ml_{\lb_0}\}$ are attractive hyperbolic
$\tau_{\lb_0}$-copies of $\W$. Hence, (c) holds, and so (a) follows.
\smallskip\par
(e)$\;\Rightarrow\;$(a) Arguments similar to the previous ones complete the proof.
\end{proof}
An easy consequence of the previous theorem will be often used:
\begin{coro}\label{coro:3noS}
Assume that the family \eqref{eq:3hlb} does not define an S-shaped global bifurcation
diagram, and take any $\lb\in\R$. If a hyperbolic $\tau_\lb$-copy of $\W$ exists, then it is
attractive and not strictly above or below any other compact
invariant set projecting onto $\W$.
\end{coro}
\section{Loss-of-hyperbolicity and jump bifurcation points}\label{sec:4}
From now on, we also assume that the flow $(\W,\sigma)$ is transitive:
see Section \ref{sec:2}.
As seen in Section \ref{subsec:2hull}, this condition, which slightly simplifies the description
that follows, is fulfilled when $\W$ is the hull of the map defining the
right-hand side of the differential equation (see Section \ref{subsec:2hull}).
So, as said in the Introduction,
the assumption is consistent with our main objective,
which is the analysis of a particular nonautonomous process,
and for which the hull construction is an essential tool.

Recall that we are working with the parametric family of ODEs
\begin{equation}\label{eq:4hlb-w_general}
 x'=\mh(\wt,x)+\lb\,,\quad\w\in\W
\end{equation}
and with the parametric family of skewproduct flows $\tau_\lb$ on $\W\times\R$
under conditions \ref{d1}, \ref{d2}, \ref{d3} and \ref{d4} on $\mh$.
Recall also that our bifurcation analysis is based on the existence of endpoints for
branches of hyperbolic copies of $\W$: see Definition \ref{def:3bifloss}.
As we will see below, an additional condition on $\mh$ ensures that, if $|\lb|$ is large enough,
then the maps $\ml_\lb$ and $\muk_\lb$ delimiting the attractor
$\mA_\lb$ (see \eqref{def:3ul}) coincide and define an attractive hyperbolic $\tau_\lb$-copy
of $\W$ (which is hence the whole global attractor). There are cases
(including very simple autonomous examples) in which $\{\ml_\lb\}$ (or $\{\muk_\lb\}$)
is still a copy of the base beyond the first (or last) value of the parameter
at which $\ml_\lb\ne\muk_\lb$. The relative position of the values $\lb_u\ge-\infty$
and $\lb_l\le\infty$ delimiting the maximal positive and negative halflines (maybe the
whole line) at which $\{\ml_\lb\}$ and $\{\muk_\lb\}$ are hyperbolic attractive copies of $\W$
will determine three quite different possibilities, and describing some
of their characteristics is the first goal of this section. The second one, carried
out in Subsection \ref{subsec:jumpbifurcations}, is
to define \emph{jump bifurcation} and to establish conditions under which
$\lb_l$ and $\lb_u$ are the only bifurcation points at which such bifurcations can occur.
Finally, in Section \ref{subsec:4critical}, we will describe the first example
of a critical transition in this article, which, like the others, is related to a jump bifurcation.
In this case, the bifurcation is one of the simplest that can occur.

The additional condition on $\mh$ is
\begin{enumerate}[leftmargin=30pt,label=\rm{\bf{d\arabic*}}]
\addtocounter{enumi}{4}
\item\label{d5} $\lim_{x\to\pm\infty}\mh_x(\w,x)=-\infty$ uniformly on a Borel subset
$\W_0\subseteq\W$ with $m(\W_0)=1$ for all $m\in\merg$.
\end{enumerate}
The five conditions are assumed without further mention in this section.
Note that the restriction of $\mh$ to any set
$\bar\W\times\R$, where $\bar\W$ is a connected compact $\sigma$-invariant
subset of $\W$, also satisfies \ref{d1}, \ref{d2}, \ref{d3}, \ref{d4} and \ref{d5}.
Note also that if $\{\mb_\lb\}$ is a hyperbolic $\tau_\lb$-copy of $\W$,
then $\{\mb_\lb|_{\bar\W}\}$ is a hyperbolic $\tau_\lb|_{\bar\W\times\R}$-copy of
$\bar\W$. Of course, the opposite is not true.
We will use these facts often, without mention.

We point out here that
the classical examples of maps satisfying \ref{d2}, \ref{d3} and \ref{d4}
also satisfy \ref{d5}. In fact, \ref{d5} holds if we replace \ref{d2} with
$\lim_{x\to\pm\infty}\mh(\w,x)/x=-\infty$ uniformly on $\W$, as proved in
\cite[Proposition 3.5]{duen}. Condition \ref{d5} is added to formulate the next
result, which justifies the
first statement in this section, and can be proved reasoning as in \cite[Theorem 5.5]{dno1}.
Transitivity is not required. Recall the shape \eqref{def:3ul}
of the global attractor $\mA_\lb$, and that the closure of the $\sigma$-orbit of $\w$
is denoted by $\W_\w$.
\begin{teor}\label{th:4shape}
If $|\lb|$ is large enough, then $\mA_\lb$ is an attractive hyperbolic $\tau_\lb$-copy of $\W$:
$\mA_\lb=\{\ml_\lb\}=\{\muk_\lb\}$.
\end{teor}
So, for each $\w\in\W$, the values
\begin{equation}\label{def:4lb-ul}
\begin{split}
\lb_l(\w)&:=\sup\{\lb\in\R\,|\;\{\ml_{\rho}|_{\W_{\w}}\}\text{ is a $\tau_\rho$-hyperbolic
copy of }\W_{\w}\text{ for all }\rho<\lb\}\,,\\
\lb_u(\w)&:=\inf\{\lb\in\R\,|\;\{\muk_{\rho}|_{\W_{\w}}\}\text{ is a $\tau_\rho$-hyperbolic
copy of }\W_{\w}\text{ for all }\rho>\lb\}
\end{split}
\end{equation}
are well defined and satisfy $-\infty<\lb_l(\w)\le\infty$ and
$-\infty\le \lb_u(\w)<\infty$, with
$\lb_l(\w_1)=\lb_l(\w_2)$ and $\lb_u(\w_1)=\lb_u(\w_2)$ if $\W_{\w_1}=\W_{\w_2}$.
It is important to recall that,
if $\{\ml_\lb\}$ or $\{\muk_{\lb}\}$ is a hyperbolic $\tau_\lb$-copy of $\W$,
then it is attractive. This assertion follows easily from Theorem \ref{th:3casosS}(b).
The values corresponding to any $\bar\w$ in the residual subset
of $\W$ of points with dense $\sigma$-orbit (see again Section \ref{sec:2})
are those special values we referred to at the beginning of this section.
We denote them as
\begin{equation}\label{def:4lb*}
 \text{$\lb_l^*:=\lb_l(\bar\w)\;$ and $\;\lb_u^*:=\lb_u(\bar\w)\;$ for any $\;\bar\w\in\W\;$
 with dense $\sigma$-orbit}\,.
\end{equation}
Proposition \ref{prop:4loss} shows their relevance in the global dynamics.
We point out that
all the results established from now on for $\W$, $\lb_l^*$ and $\lb_u^*$ can be
rewritten for $\W_\w$, $\lb_l(\w)$ and $\lb_u(\w)$ for any $\w\in\W$.
\begin{prop}\label{prop:4loss}
The points $\lb_l^*$ and $\lb_u^*$ are loss-of-hyperbolicity bifurcation points
for \eqref{eq:4hlb-w_general} if they belong to $\R$.
\end{prop}
We postpone the proof, which relies in the first one of the following two lemmas.
Lemma \ref{lema:hypconttapas} shows that $\muk_\lb$ gives a branch of attractive hyperbolic copies
of $\W$ on $(\lb_u^*,\infty)$ and $\ml_\lb$ gives a branch of attractive
hyperbolic copies of $\W$ on $(-\infty,\lb_l^*)$, and implies the lack
of hyperbolicity of $\muk_{\lb_u^*}$ and $\ml_{\lb_l^*}$ if $\lb_u^*$ and $\lb_l^*$ are finite.
Lemma \ref{lema:infinity} shows that the two values are simultaneously finite or infinite.
We represent by $\W^{\upalpha}_\w\subseteq\W$ and $\W^{\upomega}_\w\subseteq\W$
the \upalfa-limit set and \upomeg-limit set of each point $\w\in\W$ for $\sigma$.
\begin{lema}\label{lema:hypconttapas}
Let $\bar\lb\in\R$ be such that $\{\muk_{\bar\lb}\}$ (resp.~$\{\ml_{\bar\lb}\}$) is an attractive
hyperbolic $\tau_{\bar\lb}$-copy of $\W$. Let $\delta>0$ and
$\mb\colon(\bar\lb-\delta,\bar\lb+\delta)\to C(\W,\R),\,\lb\mapsto\mb_\lb$ be the continuous map
provided by Theorem {\rm\ref{th:3persistence}},
with $\mb_{\bar\lb}=\muk_{\bar\lb}$ (resp.~$\mb_{\bar\lb}=\ml_{\bar\lb}$).
Then, there exists $\delta_1\in(0,\delta]$ such that
$\mb_\lb=\muk_\lb$ (resp.~$\mb_\lb=\ml_\lb$)
for all $\lb\in(\bar\lb-\delta,\bar\lb+\delta_1)$
(resp.~$\lb\in(\bar\lb-\delta_1,\bar\lb+\delta)$. In particular,
the maps $(\lb_u^*,\infty)\to C(\W,\R),\,\lb\mapsto\muk_\lb$ and
$(-\infty,\lb_l^*)\to C(\W,\R),\,\lb\mapsto\ml_\lb$ are continuous and
strictly increasing.
\end{lema}
\begin{proof}
Recall that $\R\to\R,\,\lb\mapsto\muk_\lb(\w)$ (resp.~$\R\to\R,\,\lb\mapsto\ml_\lb(\w)$)
is right-continuous
(resp.~left-continuous) and strictly increasing for all $\w\in\W$ (see Remark \ref{rm:3pre}.1).

Let us work in the case of $\muk_\lb$. Note that, by definition of $\muk_\lb$, we have
$\mb_\lb\le \muk_\lb$.
We begin by checking the assertion for $\lb>\bar\lb$. Recall that $\muk_\lb$ is upper
semicontinuous on $\W$, which is compact. Since
$\lb\mapsto\muk_\lb(\w)$ is right-continuous and strictly increasing for all $\w\in\W$,
the pointwise limit as $\lb\to\bar\lb^+$ of $\muk_\lb$ is $\muk_{\bar\lb}$, which
in our case is continuous. So, the extension of Dini's Theorem to the upper semicontinuous case
ensures that $\lim_{\lb\to\bar\lb^+}\muk_\lb(\w)=\muk_{\bar\lb}(\w)$ uniformly on $\W$. Hence,
$\lim_{\lb\to\bar\lb^+}(\muk_\lb(\w)-\mb_\lb(\w))=0$ uniformly on $\W$. We choose $\delta_1\le\delta$
such that $\muk_\lb(\w)-\mb_\lb(\w)<\rho$ for all $\w\in\W$ if $\lb\in[\bar\lb,\bar\lb+\delta_1)$,
with $\rho$ also provided by Theorem \ref{th:3persistence}.
So, for any $\lb\in[\bar\lb,\bar\lb+\delta_1)$, there exists $\gamma_\lb>0$ and $K_\lb\ge 1$ such that
$|\mb_\lb(\w)-\muk_\lb(\w)|\le K_\lb e^{-\gamma_\lb t}|\mb_\lb(\w{\cdot}(-t))-\muk_\lb(\w{\cdot}(-t))|
\le K_\lb e^{-\gamma_\lb t}\rho$ for all $\w\in\W$ and $t\ge 0$, and therefore
$\mb_\lb(\w)=\muk_\lb(\w)$.

Now let us work to the left of $\bar\lb$. Let $\mM\subseteq\W$ be a $\sigma$-minimal set.
Assume for contradiction that there exists $\lb_0\in(\bar\lb-\delta,\bar\lb]$ and $\w_0\in\mM$ such that
$\muk_{\lb_0}(\w_0)>\mb_{\lb_0}(\w_0)$. According to \cite[Propositions 2.7 and 2.6(i)]{dno4},
$\inf_{t\le 0}\big(\muk_{\lb_0}(\w_0{\cdot}t)-\mb_{\lb_0}(\w_0{\cdot}t)\big)>0$. This means
the existence of two separated compact $\tau_{\lb_0}$-invariant sets projecting onto the
\upalfa-limit set of $\w_0$, which is $\mM$, being the
lower one $\{\mb_{\lb_0}|_\mM\}$. So, Theorems \ref{th:3casosS} and \ref{th:3Dbifur}
combined with the minimality of $\mM$ show that the global bifurcation diagram for the family
\eqref{eq:4hlb-w_general} on $\mM$ is S-shaped, with bifurcation points
$(\lb_-)_\mM$ and $(\lb^+)_\mM$ depending only on $\mM$.
Since $\muk_{\bar\lb}|_\mM$ is a hyperbolic copy of $\mM$, there are two possibilities:
either $\bar\lb>(\lb_-)_\mM$, and then the unique continuous branch of
hyperbolic copies of $\mM$ containing $\muk_{\bar\lb}|_\mM$ is
$((\lb_-)_\mM,\infty),\,\lb\mapsto\muk_\lb|_\mM$, which means that
$\muk_\lb|_\mM=\mb_\lb|_\mM$ for all $\lb\in(\bar\lb-\delta,\bar\lb]$;
or $\bar\lb<(\lb_-)_\mM$, and hence $\muk_\lb|_\mM$ coincides with
$\mb_\lb|_\mM$ (and with $\ml_\lb|_\mM$) for all
$\lb\in(\bar\lb-\delta,\bar\lb]$. Both situations contradict
$\muk_{\lb_0}(\w_0)>\mb_{\lb_0}(\w_0)$, so the maps coincide
at the elements of any minimal set.

Finally, let us fix $\lb_0\in(\bar\lb-\delta,\bar\lb]$.
We take any $\w_1\in\W$, a $\sigma$-minimal set $\mM\subseteq\W^{\upalpha}_{\w_1}$, and
a sequence $(t_n)\downarrow-\infty$ such that $\w_1{\cdot}t_n\to\w_0\in\mM$.
The upper semicontinuity of $\muk_{\lb_0}$ and the continuity of $\mb_{\lb_0}$ ensure that
$\muk_{\lb_0}(\w_0)\ge\lim_{n\to\infty}\muk_{\lb_0}(\w_1{\cdot}t_n)\ge
\lim_{n\to\infty}\mb_{\lb_0}(\w_1{\cdot}t_n)
=\mb_{\lb_0}(\w_0)=\muk_{\lb_0}(\w_0)$. So, $\lim_{n\to\infty}\big(\muk_{\lb_0}(\w_1{\cdot}t_n)-
\mb_{\lb_0}(\w_1{\cdot}t_n)\big)=0$, and \cite[Proposition 2.6(i)]{dno4}
yields $\muk_{\lb_0}(\w_1)=\mb_{\lb_0}(\w_1)$. That is, $\muk_{\lb_0}=\mb_{\lb_0}$.

The proof for $\ml_\lb$ is analogous, and the last assertion follows from
the previous ones and Remark \ref{rm:3pre}.1.
\end{proof}

\begin{lema}\label{lema:infinity}
\begin{itemize}[leftmargin=20pt]
\item[{\rm(i)}] $\lb_u^*=-\infty$ if and only if $\lb_l^*=+\infty$, in which case
$\ml_\lb=\muk_\lb$ for all $\lb\in\R$.
\item[{\rm(ii)}] $\lb_u(\w)\le\lb_u^*$ and $\lb_l^*\le\lb_l(\w)$ for
any $\w\in\W$.
\end{itemize}
\end{lema}
\begin{proof} (i) If $\lb_u^*=-\infty$,
then $\{\muk_\lb\}$ is an attractive hyperbolic copy of $\W$ for all $\lb\in\R$,
and Lemma \ref{lema:hypconttapas} shows that
$\R\to C(\W,\R),\,\lb\mapsto\muk_\lb$ is continuous. In addition,
Theorem \ref{th:4shape} shows that
$\bar\lb:=\sup\{\lb\in\R\,|\;\ml_\rho=\muk_\rho\text{ for all }\rho<\lb\}$
satisfies $\bar\lb>-\infty$. Clearly, $\bar\lb\le \lb_l^*$.
We assume for contradiction that $\bar\lb<\infty$. Recall that $\ml_{\bar\lb}$
is the pointwise limit of $\ml_\lb$ as $\lb\uparrow\bar\lb$ (see Remark~\ref{rm:3pre}.1),
and hence $\ml_{\bar\lb}=\muk_{\bar\lb}$. In particular, $\{\ml_{\bar\lb}\}$ is a
hyperbolic copy of $\W$, and hence $\bar\lb<\lb_l^*$. But this
contradicts the definition of $\bar\lb$ and Lemma \ref{lema:hypconttapas}. So,
$\lb_l^*=\bar\lb=\infty$ and hence $\ml_\lb=\muk_\lb$ for all $\lb\in\R$.
The argument is analogous if $\lb_l^*=+\infty$.
\smallskip\par
(ii) As already mentioned, it is clear that
$\mb_\lb|_{\W_\w}$ is a hyperbolic copy of $\W_\w$ if
$\mb_\lb$ is a hyperbolic copy of $\W$. Properties (ii) follow
immediately from this fact and the definitions of $\lb_l(\w)$, $\lb_u(\w)$,
$\lb_l^*$ and $\lb_u^*$.
\end{proof}
\noindent{\em Proof of Proposition {\rm\ref{prop:4loss}}.}~It is
an easy consequence of Theorem \ref{th:3persistence},
Lemma \ref{lema:hypconttapas}, and definitions \eqref{def:4lb-ul}.
\hfill$\Box$\medskip

As mentioned earlier, we classify the bifurcations according
to the relative ordering of $\lb_l^*$ and $\lb_u^*$.
There are three possibilities:
\begin{enumerate}[leftmargin=30pt,label=\rm{\bf{p\arabic*}}]
\item\label{p1} $\lb_u^*<\lb_l^*$\,,
\item\label{p2} $\lb_u^*=\lb_l^*$\,,
\item\label{p3} $\lb_u^*>\lb_l^*$\,.
\end{enumerate}
As said at the beginning of the section, now we will describe some of the general
properties that arise or are precluded in these three cases.
\begin{nota}\label{rm:4p1}
The case $\lb_u^*=-\infty=-\lb_l^*$ fits \ref{p1}, as well as
the S-shaped bifurcation diagram described by Theorem \ref{th:3Dbifur} (with
$-\infty<\lb_u^*=\lb_-<\lb^+=\lb_l^*<\infty$).
The remaining possibilities for case \ref{p1} are described in Theorem \ref{teor:4p1}.
\end{nota}

There are trivial autonomous examples (which will be mentioned later)
of cases \ref{p1} and \ref{p2} and, as we will see in
Section \ref{sec:5}, case \ref{p3} also occurs (although not in the autonomous case, as
the results of \cite[Section 5]{dno1} ensure).
Recall that, in any of the cases,
$\ml_{\lb_l^*}$ and $\muk_{\lb_u^*}$ are not hyperbolic copies of $\W$
when $\lb_l^*,\lb_u^*\in\R$. But, of course, these are not the unique values of $\lb$
at which this happens. In fact, $\ml_\lb$ and $\muk_\lb$ are not necessarily continuous maps on
$\W$. Despite this lack of continuity or hyperbolicity,
the union of their graphs contains any attractive hyperbolic copy of $\W$, if this
one exists:
\begin{prop}\label{prop:4disyuntiva}
Let $\lb\in\R$ and let $\mb_\lb\colon\W\to\R$ define an attractive hyperbolic
$\tau_\lb$-copy of $\W$. Then, for every $\w\in\W$, either $\mb_\lb(\w)=\ml_\lb(\w)$ or
$\mb_\lb(\w)=\muk_\lb(\w)$.
\end{prop}
\begin{proof}
Let $\rho$ be the radius of uniform stability of $\mb_\lb$, and let $\mM$ be a
minimal subset of $\W$. Let us check that either $\ml_\lb|_\mM<\mb_\lb|_\mM$
or $\ml_\lb|_\mM=\mb_\lb|_\mM$. We assume the existence $\w_0\in\mM$ such that
$\ml_\lb(\w_0)=\mb_\lb(\w_0)$, take $\bar\w$ in the residual subset of $\mM$
of continuity points of $\ml_\lb$, write $\bar\w=\lim_{n\to\infty}\w_0{\cdot}t_n$
for a suitable sequence, and deduce that $\ml_\lb(\bar\w)=
\lim_{n\to\infty}\ml_\lb(\w_0{\cdot}t_n)=\lim_{n\to\infty}\mb_\lb(\w_0{\cdot}t_n)
=\mb_\lb(\bar\w)$. Now, we take any other $\w\in\mM$, look for $(s_n)\downarrow-\infty$
such that $\bar\w=\lim_{n\to\infty}\ws_n$, deduce that $\mb_\lb(\ws_n)-\ml_\lb(\ws_n)<\rho$
for large enough $n$, and conclude that for these values of $n$ and certain constants
$k\ge 1$ and $\gamma>0$ (see Definition \ref{def:hcotb}),
$\mb_\lb(\w)-\ml_\lb(\w)=v_\lb(-s_n,\ws_n,\mb_\lb(\ws_n))-v_\lb(-s_n,\ws_n,\ml_\lb(\ws_n))
\le k\,e^{\gamma s_n}(\mb_\lb(\ws_n)-\ml_\lb(\ws_n))\le k\,e^{\gamma s_n}\rho$, which
ensures that $\mb_\lb(\w)=\ml_\lb(\w)$. This proves the assertion.
Similarly, we conclude that either $\mb_\lb|_\mM<\muk_\lb|_\mM$ or $\mb_\lb|_\mM=\muk_\lb|_\mM$.
Note also that the combination $\ml_\lb|_\mM<\mb_\lb|_\mM<\muk_\lb|_\mM$ is not possible: if so,
\cite[Proposition 5.2]{dno4} would ensure that $\mb_\lb|_\mM$ is hyperbolic repulsive,
which is not the case.

Now, we fix $\w\in\W$, take a minimal subset $\mM$ of $\W^{\upalpha}_\w$, and assume that
$\ml_\lb|_\mM=\mb_\lb|_\mM$.
So, $\mb_\lb(\w{\cdot}t_n)-\ml_\lb(\w{\cdot}t_n)<\rho$ for a suitable sequence
$(t_n)\downarrow-\infty$, and hence
$|\mb_\lb(\w)-\ml_\lb(\w)|\leq k\, e^{\gamma\,t_n}|b_\lb(\w{\cdot}t_n)-\ml_\lb(\w{\cdot}t_n)|\leq
k\,e^{\gamma\,t_n}\rho$. Thus, $\mb_\lb(\w)=\ml_\lb(\w)$.
The conclusion is the same if $\muk_\lb|_\mM=\mb_\lb|_\mM$. This completes the proof.
\end{proof}
The following result describes a collection of dynamical properties that are common to
the cases \ref{p1}, \ref{p2} and \ref{p3}:
\begin{teor}\label{teor:4p123}
Assume that ${\bar\lb}>\max\{\lb_u^*,\lb_l^*\}$
(resp.~${\bar\lb}<\min\{\lb_u^*,\lb_l^*\}$), and take any $\bar\w\in\W$  with dense $\sigma$-orbit.
Then,
\begin{itemize}[leftmargin=20pt]
\item[{\rm(i)}] $\lim_{t\to\infty}\big(\muk_{\bar\lb}(\bwt)-\ml_{\bar\lb}(\bwt)\big)=0$\,.
\item[{\rm(ii)}] One of the following statements holds:
\begin{itemize}[leftmargin=20pt]
\item[{\rm(a)}] $\muk_{\bar\lb}(\bwt)=\ml_{\bar\lb}(\bwt)$ for all $t\in\R$,
\item[{\rm(b)}] there exists $\delta>0$ such that
$\inf_{t\le 0}\big(\muk_{\bar\lb}(\bwt)-\ml_{\bar\lb}(\bwt)\big)=\delta$,
and the bifurcation diagram of \eqref{eq:4hlb-w_general}
restricted to $\W^{\upalpha}_{\bar\w}$ is S-shaped. In this case, if
$(\lb_-)_{\w_0}$ and $(\lb^+)_{\w_0}$ are defined as in Theorem {\rm\ref{th:3Dbifur}(v)},
then $(\lb_-)_{\w_0}=\lb_u(\w_0)$, $(\lb^+)_{\w_0}=\lb_l(\w_0)$ and
$\lb_u(\w_0)\le\lb_u^*<{\bar\lb}\le\lb_l(\w_0)$
(resp. $\lb_u(\w_0)\le{\bar\lb}<\lb_l^*\le\lb_l(\w_0)$) for all
$\w_0\in\W^{\upalpha}_{\bar\w}$.
\end{itemize}
In particular, if $\W=\W^{\upalpha}_{\bar\w}$, then {\rm(a)} holds.
\end{itemize}
\end{teor}
\begin{proof}
We work in the case ${\bar\lb}>\max\{\lb_u^*,\lb_l^*\}$, proving first that
\begin{equation}\label{eq:apartirdeaqui}
 \inf_{t\in\R}\big(\muk_{\bar\lb}(\bwt)-\ml_{\bar\lb}(\bwt)\big)=0\,.
\end{equation}
Let us assume for contradiction that this is not the case.
Since $\muk_{\bar\lb}$ provides a copy of $\W$, it follows easily that $\{\muk_{\bar\lb}\}$ and
$\mathrm{closure}_{\W\times\R}\{(\bwt,\ml_{\bar\lb}(\bwt))\,|\;t\in\R\}$ are disjoint
ordered compact $\tau_{\bar\lb}$-invariant sets projecting onto $\W$.
So, Corollary \ref{coro:3noS} ensures that the bifurcation diagram is that
described by Theorem \ref{th:3Dbifur}, and this result ensures that ${\bar\lb}\in(\lb_-,\lb^+]$.
But in this case, $(\lb_-,\lb^+]=(\lb_u^*,\lb_l^*]$ (see Remark \ref{rm:4p1}),
which contradicts the choice of $\bar\lb$. So, \eqref{eq:apartirdeaqui} holds.

According to \cite[Propositions 2.7 and 2.6(i)]{dno4}, \eqref{eq:apartirdeaqui} guarantees
that either
$\ml_{\bar\lb}(\bwt)=\muk_{\bar\lb}(\bwt)$ for all $t\in\R$ (and hence (a) of (ii) holds) or
\begin{equation}\label{eq:solshyp_haciatras}
\delta:=\inf_{t\le 0}\big(\muk_{\bar\lb}(\bwt)-\ml_{\bar\lb}(\bwt)\big)>0\,.
\end{equation}
In both cases, $\inf_{t\ge 0}\big(\muk_{\bar\lb}(\bwt)-\ml_{\bar\lb}(\bwt)\big)=0$,
and the hyperbolicity of $\muk_{\bar\lb}$ yields
\begin{equation}\label{eq:solshyp_haciadelante}
\lim_{t\to\infty}\big(\muk_{\bar\lb}(\bwt)-\ml_{\bar\lb}(\bwt)\big)=0\,,
\end{equation}
that is, (i) holds.

Let us complete the proof of (ii). As just seen, if
we are not in case (a), then \eqref{eq:solshyp_haciatras} holds, and this ensures
that $\{\muk_{\bar\lb}|_{\W^{\upalpha}_{\bar\w}}\}$
and the \upalfa-limit set for $\tau_{\bar\lb}$ of
$(\bar\w,\ml_{\bar\lb}(\bar\w))$ are disjoint ordered
compact $\tau_{\bar\lb}$-invariant sets projecting onto $\W^{\upalpha}_{\bar\w}$,
being $\{\muk_{\bar\lb}|_{\W^{\upalpha}_{\bar\w}}\}$ an attractive hyperbolic copy of
$\W^{\upalpha}_{\bar\w}$.
So, when restricting \eqref{eq:4hlb-w_general} to $\W^{\upalpha}_{\bar\w}$,
Corollary \ref{coro:3noS} ensures that we are in case (c)
or (e) of Theorem \ref{th:3casosS}, and hence the bifurcation diagram is that
of Theorem \ref{th:3Dbifur}. Let us take
$\w_0\in\W^{\upalpha}_{\bar\w}$ and observe that
$\W_{\w_0}\subseteq\W^{\upalpha}_{\bar\w}$, so that
the bifurcation diagram for the restriction to $\W_{\w_0}$ is also S-shaped.
Combining the information of Theorems \ref{th:3casosS} and \ref{th:3Dbifur}
with that of Remark \ref{rm:4p1}, we conclude that
${\bar\lb}\in((\lb_-)_{\w_0},(\lb^+)_{\w_0}]=
(\lb_u(\w_0),\lb_l(\w_0)]$.
Finally, our initial assumption ensures that $\lb_u^*<{\bar\lb}$,
and Lemma \ref{lema:infinity}(ii) ensures that $\lb_u(\w_0)\le\lb_u^*$.

To complete the proof, we check that (a) holds if $\W=\W^{\upalpha}_{\bar\w}$;
i.e., if $\bar\w\in\W^{\upalpha}_{\bar\w}$. For contradiction, we assume
\eqref{eq:solshyp_haciatras}. Since $\bwt\in \W^{\upalpha}_{\bar\w}$ for all $t\in\R$,
\eqref{eq:solshyp_haciatras} combined with the continuity of $\muk_{\bar\lb}$
and the lower semicontinuity of $\ml_{\bar\lb}$ ensures that
$(\muk_{\bar\lb}(\bwt)-\ml_{\bar\lb}(\bwt)\big)\ge\delta$ for all $t\in\R$,
which contradicts (i).

The proof for ${\bar\lb}<\min\{\lb_u^*,\lb_l^*\}$ is analogous.
\end{proof}
Note that the last assertion restricts the type
of base spaces $\W$ for which situation
(b) may arise. For instance, it cannot occur if $\W$ is minimal.

Our next result completes the information given in Remark \ref{rm:4p1}.
Three different situations may occur within \ref{p1}: either  the absence of bifurcation points,
or a global S-shaped bifurcation diagram for the whole family, or an S-shaped bifurcation
diagram at the \upalfa-limit sets of all the points with dense orbit.
\begin{teor}\label{teor:4p1}
Assume that \ref{p1} holds.
Then, exactly one of the following statements holds:
\begin{itemize}[leftmargin=20pt]
\item[\rm{(a)}] $\lb_u^*=-\infty$ and $\lb_l^*=\infty$. If so, the global attractor
$\mA_\lb$ is an attractive hyperbolic $\tau_\lb$-copy of $\W$ for all $\lb\in\R$,
and hence there are no loss-of-hyperbolicity bifurcation points.
\item[\rm{(b)}] $-\infty<\lb_u^*<\lb_l^*<\infty$ and,
for all $\lb\in(\lb_u^*,\lb_l^*)$, there exist
three disjoint $\tau_\lb$-copies of $\W$. If so,
the bifurcation diagram of \eqref{eq:4hlb-w_general}
is S-shaped and, if $\lb_-$ and $\lb^+$ are defined in Theorem
{\rm\ref{th:3Dbifur}}, then $\lb_-=\lb_u^*$ and $\lb^+=\lb_l^*$.
\item[\rm{(c)}] $-\infty<\lb_u^*<\lb_l^*<\infty$ and
there are no three
$\tau_\lb$-disjoint copies of $\W$ for any $\lb\in(\lb_u^*,\lb_l^*)$.
In this case, if $\bar\w$ is any point with dense $\sigma$-orbit, then
there exist three $\tau_{\bar\lb}$-copies of $\W^{\upalpha}_{\bar\w}$
for any $\bar\lb\in(\lb_u^*,\lb_l^*)$,
the bifurcation diagram of \eqref{eq:4hlb-w_general} restricted to
$\W^{\upalpha}_{\bar\w}$ is S-shaped, and if $(\lb_-)_{\w_0}$ and $(\lb^+)_{\w_0}$
are defined as in Theorem {\rm\ref{th:3Dbifur}(v)},
then $(\lb_u^*,\lb_l^*)\subseteq((\lb_-)_{\w_0},(\lb^+)_{\w_0})$
for all $\w_0\in\W^{\upalpha}_{\bar\w}$.
Moreover, for any $\bar\lb\in(\lb_u^*,\lb_l^*)$,
$\lim_{t\to\infty}\big(\muk_{\bar\lb}(\bwt)-\ml_{\bar\lb}(\bwt)\big)=0$
and $\muk_{\bar\lb}|_{\W^{\upomega}_{\bar\w}}=
\ml_{\bar\lb}|_{\W^{\upomega}_{\bar\w}}$.
\end{itemize}
Furthermore, if there exist $\w_1,\w_2\in\W$ (equal or different) with dense orbit and
with $\W^{\upalpha}_{\w_1}\cap\W^{\upomega}_{\w_2}$ non-empty, then either {\rm (a)} or
{\rm (b)} holds.
\end{teor}
\begin{proof}
In case (a), for all $\lb\in\R$, $\ml_\lb=\muk_\lb$ (see Lemma \ref{lema:infinity}(i)),
which ensures that $\mA_\lb=\{\muk_\lb\}$, a copy of $\W$. The definition of $\lb_u^*$
ensures that it is hyperbolic, and hence it is attractive. The final conclusion in (a) is obvious.

Let us assume that this is not the case and fix
any point $\bar\w\in\W$ with dense $\sigma$-orbit.
Lemma \ref{lema:infinity}(i) ensures that $-\infty<\lb_u^*<\lb_l^*<\infty$.
If there exists $\lb_0\in(\lb_u^*,\lb_l^*)$ such that
$\inf_{t\in\R}\big(\muk_{\lb_0}(\bwt)-\ml_{\lb_0}(\bwt)\big)=\rho>0$, then
the semicontinuity properties of $\muk_{\lb_0}$ and $\ml_{\lb_0}$ (see Remark \ref{rm:3pre}.1)
ensure that $\inf_{\w\in\W}\big(\muk_{\lb_0}(\w)-\ml_{\lb_0}(\w)\big)>0$, and hence
the hyperbolic copies of $\W$ $\{\muk_{\lb_0}\}$ and $\{\ml_{\lb_0}\}$ are disjoint.
So, Corollary \ref{coro:3noS} and the description made in Theorem
\ref{th:3Dbifur} prove that the situation is that described in (b).

Let us check that the assertions given in (c) hold
if $\inf_{t\in\R}\big(\muk_\lb(\bwt)-\ml_\lb(\bwt)\big)=0$ for all
$\lb\in(\lb_u^*,\lb_l^*)$.
If $\muk_{\lb_1}=\ml_{\lb_1}$ held for some $\lb_1\in(\lb_u^*,\lb_l^*)$,
we could reason as in the proof of Lemma \ref{lema:hypconttapas}
to check that $\muk_\lb=\ml_\lb$ for all
$\lb\in(\lb_u^*,\lb_l^*)$. Therefore,
$\muk_{\lb_u^*}=\lim_{\lb\to(\lb_u^*)^+}\muk_\lb=
\lim_{\lb\to(\lb_u^*)^+}\ml_\lb=\ml_{\lb_u^*}$, due to the right continuity
of $\lb\mapsto\muk_\lb$ (see Remark~\ref{rm:3pre}.1)
and the continuity of $\lb\mapsto\ml_\lb$ at $\lb=\lb_u^*$.
But this means that $\muk_{\lb_u^*}$ is a hyperbolic copy of $\W$, since
$\lb_u^*<\lb_l^*$, which contradicts the definition of $\lb_u^*$.
We fix any ${\bar\lb}\in(\lb_u^*,\lb_l^*)$, and observe that we have
just checked that $\muk_{\bar\lb}(\bwt)\neq\ml_{\bar\lb}(\bwt)$ for all $t\in\R$,
since both maps are continuous and $\{\bwt\,|\;t\in\R\}$ is dense.
Reasoning as in the proof of Theorem \ref{teor:4p123} (after proving \eqref{eq:apartirdeaqui}),
we deduce that \eqref{eq:solshyp_haciatras} and hence
\eqref{eq:solshyp_haciadelante} hold
for $\bar\w$ and that $\ml_{\bar\lb}|_{\W^{\upalpha}_{\bar\w}}$ and
$\muk_{\bar\lb}|_{\W^{\upalpha}_{\bar\w}}$ are two different hyperbolic copies of
$\W^{\upalpha}_{\bar\w}$, and conclude that the bifurcation diagram on
$\W^{\upalpha}_{\bar\w}\times\R$
is S-shaped with $(\lb_u^*,\lb_l^*)
\subseteq((\lb_-)_{\w_0},(\lb^+)_{\w_0})$ for all $\w_0\in\W^{\upalpha}_{\bar\w}$.
Finally, the continuity of $\lb\mapsto\ml_\lb$ and $\lb\mapsto\muk_\lb$ at
$\bar\lb\in(\lb_u^*,\lb_l^*)$
and $\lim_{t\to\infty}\big(\muk_{\bar\lb}(\bwt)-\ml_{\bar\lb}(\bwt)\big)=0$
(see \eqref{eq:solshyp_haciadelante}) ensure that
$\muk_{\bar\lb}|_{\W^{\upomega}_{\bar\w}}=\ml_{\bar\lb}|_{\W^{\upomega}_{\bar\w}}$.

To complete the proof, it remains to assume (c) and $\W=\W_{\w_1}=\W_{\w_2}$,
and check that $\W^{\upalpha}_{\w_1}\cap\W^{\upomega}_{\w_2}$ is empty: for contradiction,
we assume the existence of $\w\in \W^{\upalpha}_{\w_1}\cap\W^{\upomega}_{\w_2}$,
take $\bar\lb\in(\lb_u^*,\lb_l^*$), and observe that
$\muk_{\bar\lb}(\w)=\ml_{\bar\lb}(\w)$, since
$\ml_{\bar\lb}|_{\W^{\upomega}_{\w_2}}=\muk_{\bar\lb}|_{\W^{\upomega}_{\w_2}}$,
and $\muk_{\bar\lb}(\w)<\ml_{\bar\lb}(\w)$, since
$\ml_{\bar\lb}|_{\W^{\upalpha}_{\w_1}}<\muk_{\bar\lb}|_{\W^{\upalpha}_{\w_1}}$.
\end{proof}
Note that the last assertion of the previous theorem
restricts the type of base spaces $\W$ for which
\ref{p1} plus (c) may arise. For instance, it cannot occur if $\W$ is minimal, or given
by a single critical point and a homoclinic connection.
Easy autonomous examples of case \ref{p1} fitting
(a) and (b) are $x'=-x^3-x+\lb$ and $x'=-x^3+x+\lb$, in this second case with
$\lb_u^*=-2/(3\sqrt{3})$ and $\lb_l^*=2/(3\sqrt{3})$. As just said, case (c)
requires a nonautonomous case:
\begin{exa}\label{ex:4p1c}
To give a simple example of case \ref{p1}(c), we take $f\colon\R\to[-1,1]$ as
a smooth decreasing map with $\lim_{t\to-\infty}f(t)=1$ and
$\lim_{t\to\infty}f(t)=-1$. The hull of $h(t,x):=-x^3+f(t)\,x$ can be
identified with the hull $\W$ of $f$: the equations of the corresponding
family \eqref{eq:4hlb-w_general} are $x'=-x^3+\mf(\wt)\,x+\lb$ for $\mf\colon\W\to\R,\,
\w\mapsto\w(0)$, and  $\W$ is the union of the orbit $\{f{\cdot}t\,|\;t\in\R\}$
of $f$ and its \upalfa-limit set and its \upomeg-limit set, which we are
the constant maps $\w^\upalpha=1$ and $\w^\upomega=-1$ (see, e.g.,
\cite[Lemma 2.4]{dno4}). Let us take $\lb<2/(3\sqrt3)$.
We can reason as in the proof of \cite[Theorem 3.7]{dno3}
to check that $\lim_{t\to-\infty}\ml_\lb(f{\cdot}t)
=\ml_\lb(\w^\upalpha)$ which is the lower (hyperbolic attractive) equilibrium
point of $x'=-x^3+x+\lb$; and it is easy to check that
$\lim_{t\to\infty}\ml_\lb(f{\cdot}t)=\ml_\lb(\w^\upomega)$, which is the unique
bounded (constant) solution of $x'=-x^3-x+\lb$. So,
$\ml_\lb\colon\W\to\R$ is continuous. On the other hand,
the unique ergodic measures on $\W$,
$m^\upalpha$ and $m^\upomega$, are those concentrated on
$\{\w^\upalpha\}$ and $\{\w^\upomega\}$. The attractive hyperbolicity of the
solution $\ml_\lb(\w^\upalpha)$ of $x'=-x^3+x+\lb$ ensures that
$\int_\W (-3(\ml_\lb(\w))^2+\mf(\w))\,dm^\upalpha=-3(\ml_\lb(\w^\upalpha))^2+1<0$,
and obviously $\int_\W (-3(\ml_\lb(\w))^2+\mf(\w))\,dm^\upomega=
-3(\ml_\lb(\w^\upomega))^2-1<0$. So, all the Lyapunov exponents of the
copy of the base $\{\ml_\lb\}$ are strictly negative (see, e.g.,
\cite[Section 2.1]{dno5}), and hence
Theorem \ref{th:3copia} ensures that $\{\ml_\lb\}$ is hyperbolic
attractive. This means that $\lb_l^*\ge 2/(3\sqrt3)$. A similar argument
shows that $\lb_u^*\le -2/(3\sqrt3)<\lb_l^*$: we are in case \ref{p1}.
(In fact, $\lb_u^*=-2/(3\sqrt3)$ and $\lb_l^*=2/(3\sqrt3)$, as deduced from
Lemma \ref{lema:infinity}(ii) applied to $\w^\upalpha$).
In addition, the equation corresponding $\w^\upomega$, $x'=-x^3-x+\lb$,
does not have three hyperbolic points for any value of $\lb$, and
this precludes (b); and the equation for $\w^\upalpha$, $x'=-x^3+x+\lb$,
has three fixed points at $\lb=0$, which precludes (a). So, the
situation is necessarily (c).
\end{exa}
The construction of examples of case \ref{p2} exhibiting different properties on the
base $\W$ will be addressed in Section \ref{sec:6}.
We close this section on the general classification with some information
about one of the (many) possibilities that may arise when \ref{p3} holds.
\begin{prop}\label{prop:4p3}
Assume that \ref{p3} holds and that,
for a value $\bar\lb\in(\lb_l^*,\lb_u^*)$, either $\{\ml_{\bar\lb}\}$
or $\{\muk_{\bar\lb}\}$ is a hyperbolic $\tau_{\bar\lb}$-copy of $\W$.
Let $\bar\w\in\W$ have dense orbit. Then, exactly one of the following statements holds:
\begin{itemize}[leftmargin=20pt]
\item[{\rm(a)}] $\muk_{\bar\lb}(\bwt)=\ml_{\bar\lb}(\bwt)$ for all $t\in\R$.
\item[{\rm(b)}] $\inf_{t\le 0}\big(\muk_{\bar\lb}(\bwt)-\ml_{\bar\lb}(\bwt)\big)>0$,
$\lim_{t\to\infty}\big(\muk_{\bar\lb}(\bwt)-\ml_{\bar\lb}(\bwt)\big)=0$,
the bifurcation diagram of \eqref{eq:4hlb-w_general} restricted to
$\W_{\bar\w}^\upalpha$ is S-shaped, and if $(\lb_-)_{\w_0}$ and $(\lb^+)_{\w_0}$
are defined as in Theorem {\rm\ref{th:3Dbifur}(v)}, then $\bar\lb\in
((\lb_-)_{\w_0},(\lb^+)_{\w_0})=(\lb_u(\w_0),\lb_l(\w_0))$
for all $\w_0\in\W^{\upalpha}_{\bar\w}$.
\end{itemize}
In particular, if $\W=\W^{\upalpha}_{\bar\w}$, then {\rm(a)} holds.
\end{prop}
\begin{proof} Let us work in the case where
$\{\muk_{\bar\lb}\}$ is an attractive hyperbolic copy of $\W$.
We assume for contradiction that
$\inf_{t\in\R}\big(\muk_{\bar\lb}(\bwt)-\ml_{\bar\lb}(\bwt)\big)>0$.
Then, $\mK_l:=\mathrm{closure}_{\W\times\R}\{(\w,\ml_{\bar\lb}(\w))\,|\;\,
\w\in\W\}$ is a
compact $\tau_{\bar\lb}$-invariant set strictly below $\{\muk_{\bar\lb}\}$. Hence,
Corollary \ref{coro:3noS} ensures that we are in the situation described on
Theorem \ref{th:3Dbifur}, and Remark \ref{rm:4p1} yields $\lb_l^*<\lb_u^*$,
which is not the case if \ref{p3} holds. Therefore,
\[
 \inf_{t\in\R}\big(\muk_{\bar\lb}(\bwt)-\ml_{\bar\lb}(\bwt)\big)=0\,.
\]
Reasoning again as in the proof of Theorem \ref{teor:4p123} (after proving
\eqref{eq:apartirdeaqui}),
we deduce: that either $\ml_{\bar\lb}(\bwt)=\muk_{\bar\lb}(\bwt)$ for all $t\in\R$,
and hence (a) holds; or
$\delta:=\inf_{t\le 0}\big(\muk_{\bar\lb}(\bwt)-\ml_{\bar\lb}(\bwt)\big)>0$ and
$\lim_{t\to\infty}\big(\muk_{\bar\lb}(\bwt)-\ml_{\bar\lb}(\bwt)\big)=0$,
and so the closure of the graph of $\ml_{\bar\lb}|_{\W^{\upalpha}_{\bar\w}}$
is a compact $\tau_{\bar\lb}$-invariant set projecting onto $\W^{\upalpha}_{\bar\w}$
and strictly below and $\muk_{\bar\lb}|_{\W^{\upalpha}_{\bar\w}}$. In this second case,
Theorem \ref{th:3casosS} ensures that the bifurcation diagram on $\W^{\upalpha}_{\bar\w}$
is S-shaped, with $\bar\lb\in ((\lb_-)_{\w_0},(\lb^+)_{\w_0})=(\lb_u(\w_0),\lb_l(\w_0))$
for all $\w_0\in\W^{\upalpha}_{\bar\w}$. Clearly, this last possibility cannot hold
for $\w_0=\bar\w$, which proves the last assertion. The arguments are analogous if
$\{\ml_{\bar\lb}\}$ is an attractive hyperbolic copy of $\W$.
\end{proof}
Again, the last assertion restricts the type of base spaces $\W$ for which
the situation (b) of Proposition~\ref{prop:4p3} may arise: for instance,
it cannot occur if $\W$ is minimal.
We will give in Section \ref{sec:5} an example of case \ref{p3}
fitting the hypotheses of the previous result and in situation (a),
and the same ideas can be adapted to get an example of situation (b).
\subsection{Jump bifurcations}\label{subsec:jumpbifurcations}
We now define \emph{jump bifurcation as $t \to \infty$}, and we use the
dynamical properties previously described to determine when such bifurcations
can occur. Recall that we are assuming that $\W$ is transitive and that this
ensures the existence of a residual $\sigma$-invariant
subset of points $\bar\w\in\W$ with dense $\sigma$-orbit.
\begin{defi} \label{def:4jump}
Let $\lb_0\in\R$ be a loss-of-hyperbolicity bifurcation point
for \eqref{eq:4hlb-w_general}. A {\em jump bifurcation as $t\to\infty$ occurs at\/} $\lb_0$,
or {\em $\lb_0$ is a jump bifurcation point as $t\to\infty$},
if there exist a $\sigma$-invariant residual subset
$\W_0\subseteq\W$ and a constant $\rho>0$ such that at least one of the next conditions holds:
\begin{eqnarray}
 \liminf_{t\to\infty}\big(\muk_{\lb_0+\ep}(\wt)-\muk_{\lb_0-\ep}(\wt)\big)\ge\rho
 &\quad \text{for all $\w\in\W_0$ and $\ep>0$}\,,\label{eq:4u_lbjumps}\\
 \liminf_{t\to\infty}\big(\ml_{\lb_0+\ep}(\wt)-\ml_{\lb_0-\ep}(\wt)\big)\ge\rho
 & \quad \text{for all $\w\in\W_0$ and $\ep>0$}\,.\label{eq:4l_lbjumps}
\end{eqnarray}
\end{defi}
This definition aims to distinguish those loss-of-hyperbolicity bifurcation points which can
which can trigger critical transitions in d-concave equations
when the parameter $\lb$ is substituted
by a parameter-dependent parameter shift $\Lambda_c(t)$: see, e.g.,
\cite{awvc,apw2017,alk2018,dno4}.
At first glance, the definition may appear unnecessarily complicated; however,
it is satisfied precisely in those cases where the notion of a ``jump'' is unambiguous,
unlike other simpler attempts of definition.
The example of case \ref{p3} with a jump bifurcation described in Section \ref{sec:5}
contributes to clarify the complexity of this matter.
\begin{nota}\label{rm:4nojump}
If $\{\muk_{\lb_0}\}$ is an attractive hyperbolic copy of $\W$,
then Lemma \ref{lema:hypconttapas} yields
$\lim_{\lb\to\lb_0}\muk_\lb(\w)=\muk_{\lb_0}(\w)$ uniformly on $\W$, therefore
precluding \eqref{eq:4u_lbjumps}.
Analogously, the hyperbolicity of $\{\ml_{\lb_0}\}$ precludes \eqref{eq:4l_lbjumps}.
\end{nota}

The following theorem shows that, in case \ref{p1}, a jump bifurcation can occur if
and only if the bifurcation diagram for the family \eqref{eq:4hlb-w_general} is S-shaped,
and that, in this situation, it occurs at both classical bifurcation points.
\begin{teor}\label{teor:4jumpbifurcationp1}
Assume that \ref{p1} holds.
Then, \eqref{eq:4hlb-w_general} has a jump bifurcation
as $t\to\infty$ if and only if situation {\rm(b)} of Theorem {\rm\ref{teor:4p1}} holds;
i.e., if and only if the bifurcation diagram is S-shaped.
In this case, the unique two jump bifurcation points are $\lb_l^*$ and $\lb_u^*$.
More precisely, \eqref{eq:4u_lbjumps} holds if and only if $\lb_0=\lb_u^*$ and
\eqref{eq:4l_lbjumps} holds if and only if $\lb_0=\lb_l^*$.
\end{teor}
\begin{proof}
Theorem \ref{teor:4p1}(a) precludes jump bifurcation points,
since there are no loss-of-hyperbolicity bifurcation points.

Assume now that the situation fits (b) or (c) of Theorem \ref{teor:4p1}.
Let us check that \eqref{eq:4u_lbjumps} cannot hold for any $\lb_0\ne\lb_u^*$.
First, if $\lb_0>\lb_u^*$, then \eqref{eq:4u_lbjumps} is precluded by the
hyperbolicity of $\muk_{\lb_0}$: see Remark \ref{rm:4nojump}.
Second, if $\lb_0<\lb_u^*$, then Theorem \ref{teor:4p123}(i) yields
$\lim_{t\to\infty}\big(\muk_{\lb_0+\ep}(\bwt)-\ml_{\lb_0+\ep}(\bwt)\big)=0$ and
$\lim_{t\to\infty}\big(\muk_{\lb_0-\ep}(\bwt)-\ml_{\lb_0-\ep}(\bwt)\big)=0$
for any $\bar\w$ in the residual $\sigma$-invariant set of points of dense orbit on
$\W$ if $\ep>0$ is small enough.
So, $\liminf_{t\to\infty}\big(\muk_{\lb_0+\ep}(\bwt)-\muk_{\lb_0-\ep}(\bwt)\big)=
\liminf_{t\to\infty}\big(\ml_{\lb_0+\ep}(\bwt)-\ml_{\lb_0-\ep}(\bwt)\big)$.
Since the intersection of two residual sets is residual,
\eqref{eq:4u_lbjumps} is precluded by the hyperbolicity of $\{\ml_{\lb_0}\}$:
see again Remark \ref{rm:4nojump}. The same argument shows that
\eqref{eq:4l_lbjumps} cannot hold for any $\lb_0\ne\lb_l^*$.
In addition, if (c) holds and $\lb_0=\lb_u^*$ and $\ep>0$ is small,
the description of Theorem \ref{teor:4p1}(c)
states that $\lim_{t\to\infty}\big(\muk_{\lb_0+\ep}(\bwt)-\ml_{\lb_0+\ep}(\bwt)\big)=0$
for any $\bar\w$ in the residual $\sigma$-invariant set of points of dense orbit on
$\W$, Theorem \ref{teor:4p123}(i)
ensures that $\lim_{t\to\infty}\big(\muk_{\lb_0-\ep}(\bwt)-\ml_{\lb_0-\ep}(\bwt)\big)=0$,
and we conclude as above that \eqref{eq:4u_lbjumps} does not hold. And, also in case (c),
we prove similarly that $\lb_0=\lb_l^*$ does not satisfy \eqref{eq:4l_lbjumps}.

Altogether, we conclude that, in case (c), there are no jump bifurcation
points, and that $\lb_l^*$ and $\lb_u^*$ are the unique possible ones in case (b).
Let us check that \eqref{eq:4u_lbjumps} actually holds in case (b) at $\lb_0=\lb_u^*$.
We take $\bar\w$ with dense orbit, apply Theorem \ref{teor:4p123}(i) to get
$\lim_{t\to\infty}\big(\muk_{\lb_0-\ep}(\bwt)-
\ml_{\lb_0-\ep}(\bwt)\big)=0$, and conclude from the strictly increasing character of
$\lb\mapsto\ml_\lb$ and $\lb\mapsto\muk_\lb$ that
\[
\begin{split}
 \liminf_{t\to\infty}\big(\muk_{\lb_0+\ep}(\bwt)-\muk_{\lb_0-\ep}(\bwt)\big)
 &=\liminf_{t\to\infty}
 \big(\muk_{\lb_0+\ep}(\bwt)-\ml_{\lb_0-\ep}(\bwt)\big)\\
 &\ge\liminf_{t\to\infty}
 \big(\muk_{\lb_0}(\bwt)-\ml_{\lb_0}(\bwt)\big)\ge\rho
\end{split}
\]
for $\rho:=\inf_{\w\in\W}\big(\muk_{\lb_0}(\w)-\ml_{\lb_0}(\w)\big)$,
which is independent of $\bar\w$: \eqref{eq:4u_lbjumps} holds.
An analogous arguments shows that \eqref{eq:4l_lbjumps} holds in case (b) for $\lb_0=\lb_l^*$.
\end{proof}
The next result shows that, in cases \ref{p2} and \ref{p3}, a jump bifurcation
can only occur at $\lb_0\in[\lb_l^*,\lb_u^*]$ (which means for $\lb_0=\lb_l^*=\lb_u^*$
in case \ref{p2}), and gives a condition under which a jump bifurcation actually occurs
at $\lb_l^*$ or at $\lb_u^*$.
\begin{teor}\label{teor:4jumpbifurcationp2p3}
\begin{itemize}
\item[\rm(i)] Assume that \ref{p2} or \ref{p3} holds.
If $\lb_0\in\R$ is a jump bifurcation point for \eqref{eq:4hlb-w_general} as $t\to\infty$,
then $\lb_0\in[\lb_l^*,\lb_u^*]$. In addition, if $\lb_1=\lb_l^*$ or $\lb_1=\lb_u^*$
and there exists a $\sigma$-invariant residual subset
$\W_r\subseteq\W$ and a constant $\rho>0$ with
\begin{equation}\label{eq:4equivp2}
 \liminf_{t\to\infty}\big(\muk_{\lb_1}(\wt)-\ml_{\lb_1}(\wt)\big)\ge\rho
 \quad\text{for all $\w\in\W_r$}\,,
\end{equation}
then $\lb_1$ is a jump bifurcation point as $t\to\infty$.
\item[\rm(ii)]Assume that \ref{p3} holds and that, for every $\lb\in(\lb_l^*,\lb_u^*)$,
either $\{\ml_\lb\}$
or $\{\muk_\lb\}$ is a hyperbolic copy of $\W$.
If $\lb_0\in\R$ is a jump bifurcation point of \eqref{eq:4hlb-w_general}
as $t\to\infty$, then $\lb_0=\lb_l^*$ or $\lb_0=\lb_u^*$.
\end{itemize}
\end{teor}
\begin{proof}
(i) If $\lb_0>\lb_u^*$, the hyperbolicity of $\muk_{\lb_0}$ directly
precludes \eqref{eq:4u_lbjumps}
(see Remark \ref{rm:4nojump}), and combined with Theorem \ref{teor:4p123}(i) precludes
\eqref{eq:4l_lbjumps}, as in the proof of Theorem \ref{teor:4jumpbifurcationp1}.
A similar argument for $\lb_0<\lb_l^*$ proves that
there are no jump bifurcation points outside $[\lb_l^*,\lb_u^*]$.
Let us take $\lb_1=\lb_l^*$ and assume \eqref{eq:4equivp2}.
The monotonicity of $\lb\mapsto\ml_\lb,\muk_\lb$ (see Remark \ref{rm:3pre}.1) yields
\[
 \liminf_{t\to\infty}\big(\muk_{\lb_1+\ep}(\wt)-\ml_{\lb_1-\ep}(\wt)\big)\ge
 \liminf_{t\to\infty}\big(\muk_{\lb_1}(\wt)-\ml_{\lb_1}(\wt)\big)\ge\rho>0
\]
for all $\w\in\W_r$. If $\W_0\subseteq\W_r$ is a residual subset of $\W$ of points
with dense orbit, Theorem \ref{teor:4p123}(i) ensures that
$\lim_{t\to\infty}\big(\muk_{\lb_1-\ep}(\wt)-\ml_{\lb_1-\ep}(\wt)\big)=0$
for all $\w\in\W_0$, and hence
\[
\begin{split}
 &\liminf_{t\to\infty}\big(\muk_{\lb_1+\ep}(\wt)-\muk_{\lb_1-\ep}(\wt)\big)\\
 &\qquad=\liminf_{t\to\infty}\big(\muk_{\lb_1+\ep}(\wt)-\ml_{\lb_1-\ep}(\wt)
 -(\muk_{\lb_1-\ep}(\wt)-\ml_{\lb_1-\ep}(\wt))\big)\ge\rho
\end{split}
\]
for all $\w\in\W_0$ and $\ep>0$. That is, \eqref{eq:4u_lbjumps} holds,
which combined with Proposition \ref{prop:4loss} shows that $\lb_u^*$ is a jump
bifurcation point. The argument is analogous if \eqref{eq:4equivp2} holds
for $\lb_1=\lb_l^*$.
\smallskip

(ii) We only have to check that \eqref{eq:4u_lbjumps} and \eqref{eq:4l_lbjumps} are
precluded for $\lb_0\in(\lb_l^*,\lb_u^*)$, what we do by reasoning as in the proof
of Theorem \ref{teor:4jumpbifurcationp1} using
the information provided by Proposition \ref{prop:4p3}.
\end{proof}

Note that condition \eqref{eq:4equivp2} is satisfied for all $\lb_0\in\R$ for which
there exist two disjoint $\tau_{\lb_0}$-invariant compact sets.
Examples of case \ref{p2} fulfilling this condition will be constructed in Section \ref{sec:6}.
And an example of case \ref{p3} fitting the hypotheses of
Theorem~\ref{teor:4jumpbifurcationp2p3}(ii) for which actually a jump
bifurcation appears, while \eqref{eq:4equivp2} does not hold, is given in
Section \ref{sec:5}.
\subsection{An example of critical transition in case \ref{p1}}\label{subsec:4critical}
We complete this section by describing an example of a critical transition for case \ref{p1},
in the same style as those in \cite{dno3,dno4,dno5}. We discuss it only very briefly, since the simulation
is analogous to those in the mentioned articles.
Examples of critical transitions for cases
\ref{p3} and \ref{p2} will be given in Sections \ref{subsec:5critical} and \ref{subsec:6critical}.
Now, we consider a case of {\em rate-induced tracking} (see \cite{dno3} and \cite{dlo1}):
the critical transition appears as the rate decreases.

We study a cubic equation of the form
\begin{equation}\label{eq:42CT}
 x'=-x^3+x+a(t)+\G(c\,t)\,,
\end{equation}
where $c>0$ is a rate parameter. Clearly, \ref{d1}, \ref{d2}, \ref{d3}, \ref{d4} and \ref{d5}
are satisfied.
This may represent, for example, the standard cubic simplification of the Landau–Khalatnikov
ferroelectric model in the case of a second-order transition \cite{maslov2021}, where $a(t)+\G(c\,t)$
denotes the external electric field applied to the system.
Here, this field will consist of
a quasiperiodic component, $a$, and a transient component, $\G$.
The observed behavior can be analyzed in terms of the nonautonomous bifurcation diagram of
\begin{equation}\label{eq:42CT_frozen}
 x'=-x^3+x+a(t)+\lb\,.
\end{equation}
We choose $a(t):=0.8\cos(\sqrt{2}\,t)+0.5\cos(\pi\,t)-0.1$ and numerically check that \eqref{eq:42CT_frozen}
has three hyperbolic solutions for $\lb=0.1$, the upper and lower ones being attractive; or,
equivalently (see \cite[Theorem 5.6]{dno4}), that there are three hyperbolic copies of
the base for the skewproduct defined on $\W\times\R$, where $\W$ is the (minimal) hull of $a$.
Therefore, the bifurcation diagram of \eqref{eq:42CT_frozen} corresponds to that of Theorem \ref{th:3Dbifur},
and so it is in case \ref{p1} (see Remark \ref{rm:4p1}).
Now, we choose $\G(t):=-0.1\,e^{-t^2}$. Its asymptotic behavior ensures that $\G$
takes values within the {\em safety interval\/} $(\lb_-,\lb^+)\approx(-0.031,0.231)\ni 0$ for
which Theorem \ref{th:3Dbifur} ensures that
\eqref{eq:42CT_frozen} has three hyperbolic solutions, and it takes values outside
it for a bounded interval of $t$ around $0$.
When we introduce the rate $c$ into the argument of $\G$, the length of this last interval for $\G(c\,t)$
depends on $c$: smaller values of $c$ correspond to smaller intervals of time for which $\G(c\,t)$
is outside the safety interval.

In the upper central part of Figure \ref{fig:42CT}, we represent the unique
solution of \eqref{eq:42CT} which is bounded and tends as $t\to-\infty$ to the upper attractive hyperbolic
solution of the past and future equation $x'=x-x^3+a(t)$: it is proved in \cite[Theorem 3.7]{dno3}
that this solution is unique.
Its behavior drastically changes as $c$ crosses certain threshold: for large values of $c>0$ the
aforementioned solution tends to the same solution of the past and future equation as $t\to\infty$
(there is tracking), while it does not for small values of $c>0$ (and tipping occurs).
The uniqueness of such
tipping value of $c$ in analogous scenarios is explained in \cite[Section 5.1]{dno3} and
\cite[Section V]{dlo1}.
Depending on the context, in the ferroelectric model, this transition may be viewed either as a
useful functional switching event, when polarization reversal is the intended operation of the
ferroelectric device, or as an undesirable critical transition, when the same abrupt jump causes
loss of the stored state, spurious switching, or reduced reliability.
\begin{figure}[ht]
\includegraphics[width=\textwidth]{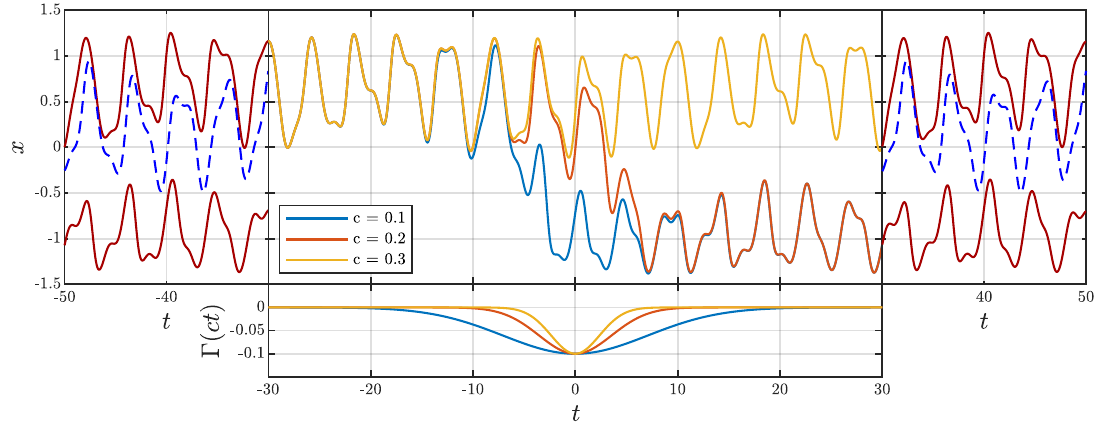}
\caption{In the side panels, approximations to the hyperbolic solutions of the unperturbed
equation \eqref{eq:42CT}: the attractive ones in solid red and the repulsive one in dashed blue.
In the top-center panel, solutions of \eqref{eq:42CT}
with $\mu=1$, $a(t)=0.8\cos(\sqrt{2}\,t)+0.5\cos(\pi t)-0.1$,
$\G(t)=-0.1\,e^{-t^2}$, and $c\in\{0.1,0.2,0.3\}$, obtained from a positive initial condition
after a preliminary integration over a convergence interval and then plotted on $[-30,30]$.
Each one of these solutions approximates the upper bounded solution of \eqref{eq:42CT}
as time decreases, and does the same as time increases if and only if
the rate $c$ is above a certain tipping value. The three upper panels share the same vertical scale.
In the bottom-center panel, the graphs of $\G(ct)$ for the same values of $c$, using the
same color code.}
\label{fig:42CT}
\end{figure}
\section{A non-trivial jump bifurcation in case \ref{p3} and a critical transition}
\label{sec:5}
In continuation with the results of Section \ref{sec:4},
we present an explicit example of case \ref{p3} for which a jump
bifurcation occurs at the point $\lb_l^*$ and for which the ``simpler'' condition
described by \eqref{eq:4equivp2} does not hold.
This example also allows us to present
a numerical simulation that underscores again the close relation between critical
transitions and jump bifurcations.

We study an equation of the type
\begin{equation}\label{eq:5type}
 x'=-\big(x+f_1(t)\big)^3+f_2(t)\,x+f_3(t)+\lb\,,
\end{equation}
where $f_i\colon\R\to\R$ for $i\in\{1,2,3\}$ are the bounded and
uniformly continuous functions depicted in Figure \ref{fig:5Lambdas}
and defined at its caption:
$f_1$ and $f_2$ are monotone functions which are
constant on $(-\infty,0]$ and $[1,\infty)$,
where $f_1$ takes the values $1$ and $0$, and
$f_2$ takes the values $0$ and $3/2^{2/3}$;
$f_3$ takes the value $-4$ on $(-\infty,0]$,
and alternates between $0$ and $2$ on time intervals of increasing length as
$t$ increases, with transition time between both values constantly equal to $1$;
and all of them depend on the map $q(t):=-2t^3+3t^2$.

\begin{figure}[ht]
\includegraphics[width=\textwidth]{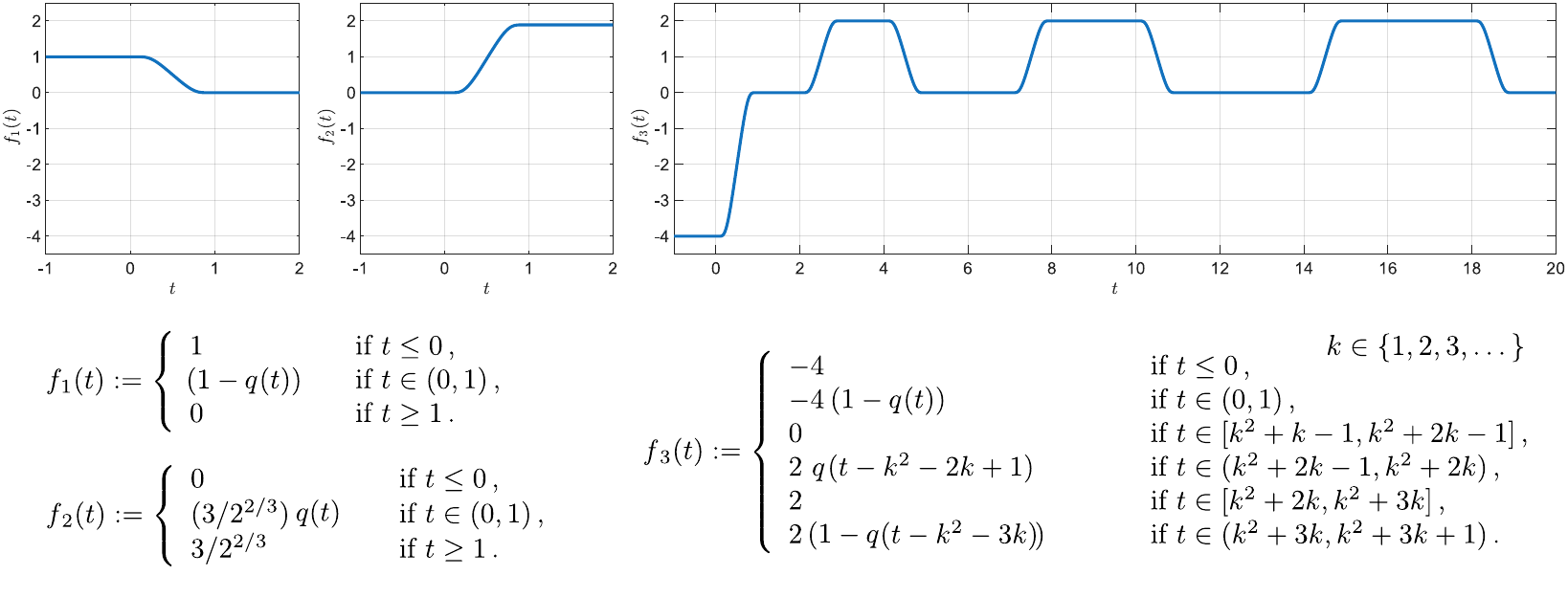}
\caption{Explicit expressions and depiction of $f_1$, $f_2$ and $f_3$.
The map $q(t)$ is the cubic polynomial $q(t):=-2t^3+3t^2$ for $t\in[0,1]$,
for which $q(0)=0$, $q(1)=1$ and $q'(0)=q'(1)=0$, so that $f_i$ are $C^1$.
The results work for any continuous map $q\colon[0,1]\to[0,1]$ with $q(0)=0$ and $q(1)=1$.}
\label{fig:5Lambdas}
\end{figure}

The hull $\W_h\subset C(\R^2,\R)$ of the map
$h(t,x):=-(x+f_1(t))^3+f_2(t)\,x+f_3(t)$ can be
identified with the hull $\W\subset C(\R,\R^3)$
of $f:=(f_1,f_2,f_3)\colon\R
\to\R^3$, given by the closure in the compact-open topology of $C(\R,\R^3)$
of the set $\{f{\cdot}s\,|\;s\in\R\}$, where $f{\cdot}s(t):=f(s+t)$ (see Section \ref{subsec:2hull}).
So, each element of this hull takes the form $\w=(\w_1,\w_2,\w_3)$, and the family of
equations to work with is
\begin{equation}
\label{eq:5typeskewproduct}
 x'=\mh(\wt,x)+\lb\,, \quad\w\in\W
\end{equation}
for $\mh(\w,x)=\mh((\w_1,\w_2,\w_3),x)):=-(x+\w_1(0))^3+\w_2(0)\,x+\w_3(0)$.
It suffices to take $\w=f$ to recover \eqref{eq:5type}.
Recall that $\W$ is a compact metric space,
and  it is transitive by construction.
In addition, conditions \ref{d1}, \ref{d2}, \ref{d3}, \ref{d4} and \ref{d5}
are very easy to verify. As usual, $\tau_\lb$ and $\mA_\lb$ represent the skewproduct
flow induced by \eqref{eq:5typeskewproduct} on $\W\times\R$ and the global attractor.

According to \cite[Lemma 2.4]{dno4}, $\W$ is the union of the orbit
of $f$ and its \upalfa-limit set and its \upomeg-limit set, which we
denote $\W^\upalpha$ and $\W^\upomega$. Let us begin by describing these sets.
From now on we call $a:=3/2^{2/3}$.
The asymptotic properties of $f_1$, $f_2$ and $f_3$
guarantee that
\[
 \W^\upalpha=\{(\bar 1,\bar 0,-\bar4)\}\qquad\text{and}\qquad
 \quad\W^\upomega=\{(\bar 0,\bar a,\w_3)\,|\;\w_3\in\W^\upomega_{f_3}\}\,,
\]
where $\bar r$ stands for the real function which is identically equal to $r$,
and $\W^\upomega_{f_3}$ is the \upomeg-limit set of $f_3$
on its own hull.
Lemma \ref{lema:5omega} will show that $\W^\upomega_{f_3}$ is given by the union of
the fixed points $\bar 0$ and $\bar 2$, and two heteroclinic orbits
connecting the fixed points in both directions. More precisely, they are the orbits
$\{\bar f_{0\mapsto2}{\cdot}t\,|\;t\in\R\}$ and $\{\bar f_{2\mapsto0}{\cdot}t\,|\;t\in\R\}$
of
\[
\bar f_{0\mapsto2}(t):=\left\{\begin{array}{ll}
0\quad&\text{if }t\le0\,,\\
2\,(1-q(t))&\text{if }t\in(0,1)\,,\\
2&\text{if }t\ge1\,.
\end{array}\right.\quad
\bar f_{2\mapsto0}(t):=\left\{\begin{array}{ll}
2\quad&\text{if }t\le0\,,\\
2\,q(t)&\text{if }t\in(0,1)\,,\\
0&\text{if }t\ge1\,.
\end{array}\right.\quad
\]
Consequently, $\W$ contains exactly three fixed points,
which are $\w^1:=(\bar 1,\bar 0,-\bar 4)$, $\w^2:=(\bar 0,\bar a,\bar 0)$
and $\w^3:=(\bar 0,\bar a,\bar 2)$, and exactly three minimal sets, $\{\w^1\}$,
$\{\w^2\}$ and $\{\w^3\}$. We will extract information about
\eqref{eq:5type} by analyzing the equations of the family \eqref{eq:5typeskewproduct}
corresponding to these three fixed points, which are
\begin{equation}\label{eq:5gi}
x'=g_i(x)+\lb\,, \qquad i=1,2,3
\end{equation}
for $g_i(x)=\mh(\w^i{\cdot}t,x)=\mh(\w^i,x)$, i.e.,
\[
 g_1(x):=-(x+1)^3-4\,,\qquad\; g_2(x):=-x^3+a\,x\,,\qquad\;
 g_3(x):=-x^3+a\,x+2\,.
\]

\begin{notas}\label{rm:5equilpoints}
1.~The three bifurcation diagrams of the autonomous equations \eqref{eq:5gi}
are depicted in Figure \ref{fig:5examplefigure}. For each value of $\lb$, the diagrams
represent the solutions of $-\lb=g_i(x)$, which are the fixed points of each equation
\eqref{eq:5gi} and determine their global dynamics. They are hyperbolic attractive
(resp.~repulsive) if and only if, at them, $(g_i)_x$ is strictly negative
(resp.~positive). It is easy to observe and check that $x'=g_1(x)+\lb$
fits case \ref{p2}, and $x'=g_i(x)+\lb$ fits \ref{p1} for $i=2,3$.

2.~Let $\mb_\lb\colon\W\to\R$ be a $\tau_\lb$-equilibrium. Then, for all $t\in\R$,
$\mb_\lb(\w^i)=\mb_\lb(\w^i{\cdot}t)=v_\lb(t,\w^i,\mb_\lb(\w^i))$.
This means that $\mb_\lb(\w^i)$ is a fixed point of $x'=g_i(x)+\lb$.
Note also that, if $x_\lb(\w^i)$ is a fixed point of this equation, then
$t\mapsto (\w^i, x_\lb(\w^i))$ is a constant (and hence bounded)
$\tau_\lb$-orbit, which means that $(\w^i,x_\lb(\w^i))\in\mA_\lb$:
see Remark \ref{rm:3pre}.1. Altogether, we conclude that
$\ml_\lb(\w^i)$ and $\muk_\lb(\w^i)$ are the minimum and maximum of the
(one, two or three) critical points of $x'=g_i(x)+\lb$. In particular,
the fiber of $\mA_\lb$ over the unique point $\w^1$ of $\W^\upalpha$ is a singleton
for all $\lb\in\R$.

3.~In addition to the previous information, we will prove that
$\mA_\lb(f{\cdot}t)$ is a singleton for all $\lb\ne 4$ and all $t\in\R$. Let us take
$(\w,x)$ in the \upalfa-limit set of $(f,\muk_{\lb}(f))$ and observe that
$\w=\w^1$, so that the previous remark ensures that $x=\muk_{\lb}(\w^1)$.
So, if $\lb\ne 4$, the \upalfa-limit set of $(f,\muk_{\lb}(f))$ is
an attractive hyperbolic copy of $\{\w^1\}$ (see Figure \ref{fig:5examplefigure}).
The same argument applies to the \upalfa-limit set of $(f,\ml_{\lb}(f))$. Assuming that
$\ml_{\lb}(f)<\muk_\lb(f)$ gives a contradiction with the information
provided by \cite[Proposition 2.6(ii)]{dno4}. The equality for
all $t\in\R$ follows immediately.
\end{notas}
 \begin{figure}[ht]
\includegraphics[width=0.95\textwidth]{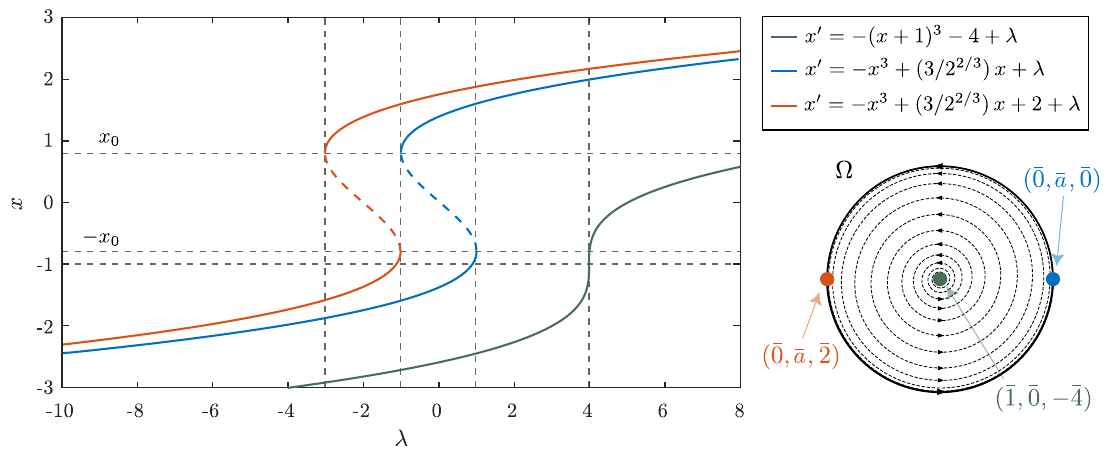}
\caption{On the left panel, the bifurcation diagrams of the three
autonomous equations \eqref{eq:5gi}. The branches of attractive and
repulsive fixed points
are depicted in solid and dashed line, respectively.
At the bifurcation points of $x'=g_2(x)+\lb$ and $x'=g_3(x)+\lb$,
the unique nonhyperbolic fixed points are $\pm x_0$ for $x_0=2^{-1/3}$.
On the right panel, a depiction of the orbits
of the base flow $\W$:
the three fixed points are depicted in colors, and the outer invariant circle
is the \upomeg-limit set of the inner spiraling trajectory,
which represents $\{f{\cdot}t\,|\;t\in\R\}\,$.}
\label{fig:5examplefigure}
\end{figure}
The main goal of this section is to prove the following result.
\begin{teor}\label{teor:5withjump}
The family \eqref{eq:5typeskewproduct} satisfies:
\begin{itemize}
\item[{\rm(i)}] $\lb_l^*=-1<4=\lb_u^*$, and hence \ref{p3} holds.
\item[{\rm(ii)}] There is a jump bifurcation point at $\lb_l^*=-1$.
\item[{\rm(iii)}] The graph $\{\muk_\lb\}$ is a hyperbolic copy of the base for $\lb\in(\lb_l^*,\lb_u^*)
=(-1,4)$.
\end{itemize}
\end{teor}
\begin{proof}
(i) It can be proved that any $\sigma$-ergodic measure on $\W$
is supported either in $\W^\upalpha$ or in $\W^\upomega$:
see e.g. the proof of \cite[Lemma 4.3]{dno4}. Clearly the unique one
concentrated in $\W^\upalpha$ is the atomic measure
$m_1$ concentrated on $\{\w^1\}$.
Now, let $\tilde m$ be any ergodic measure on $\W^\upomega$.
As explained in, e.g., \cite[Remark 1.10]{jonnf}, the
separability of $C(\W^\upomega,\R)$ ensures that $\tilde m$ is concentrated on
\[
 \mS(\tilde m):=\left\{\w\in\W^\upomega\,\bigg|\;\lim_{T\to\infty}\frac{1}{T}\int_0^T f(\wt)\,dt=
 \int_{\W^\upomega} f\,d\tilde m\;\text{ for all }f\in C(\W^\upomega,\R)\right\}\,.
\]
We take $\w\in\mS(\tilde m)$ and deduce from Lemma~\ref{lema:5omega} that $\lim_{t\to\infty}\wt=\w^i$
for $i=2$ or $i=3$. It follows easily that $\lim_{T\to\infty}\frac{1}{T}\int_0^T f(\wt)\,dt=
f(\w^i)$ and hence $\int_{\W^\upomega} f\,d\tilde m=f(\w^i)$ for all $f\in C(\W^\upomega,\R)$,
which means that $\tilde m$ is the atomic measure concentrated on $\{\w^i\}$.
So, there are two ergodic measures on $\W^\upomega$:
$m_2$ and $m_3$, concentrated on $\{\w^2\}$ and $\{\w^3\}$.
In particular, if $\mb_\lb\colon\W\to\R$ is an $m_i$-measurable $\tau_\lb$-equilibrium, then
$\int_{\W}\mh_x(\w,\mb_\lb(\w))\,dm_i=\mh_x(\w^i,\mb_\lb(\w^i))=g_i(\mb_\lb(\w^i))$
for $i=1,2,3$ (see Remark \ref{rm:5equilpoints}.2).

We begin by proving that $\lb_u^*=4$. Let us take $\lb>4$ and any $\tau_\lb$-equilibrium
$\mb_\lb$. As explained in, e.g., \cite[Section 2.1]{dno5}, the upper Lyapunov of
$\mA_\lb$ is given by $\int_\W\mh_x(\w,\mb_\lb(\w))\,dm$ for a $\sigma$-ergodic measure and
an equilibrium $\mb_\lb$. Since $\lb>4$, each one of the three equations \eqref{eq:5gi}
has a unique fixed point, which is hyperbolic attractive
(see Figure \ref{fig:5examplefigure}). Remarks \ref{rm:5equilpoints} explain
that the fixed point is $\mb_\lb(\w^i)=\ml_\lb(\w^i)=\muk_\lb(\w^i)$ for $i=1,2,3$,
and that $(g_i)_x(\mb_\lb(\w^i))<0$. This means that all the Lyapunov exponents of
$\mA_\lb$ are negative. Since the unique minimal sets of are $\{\w^i\}$ for $i=1,2,3$,
the additional condition in Theorem \ref{th:3copia} is also fulfilled, and hence
$\mA_\lb$ is an attractive hyperbolic copy of the base:
$\mA_\lb=\{\ml_\lb\}=\{\muk_\lb\}$. Therefore, $\lb_u^*\le 4$.
We can repeat the argument with $\lb=4$, with a unique important big difference:
now, $(g_1)_x(\mb_4(\w^1))=0$. Since, again,
$\mb_4(\w^1)=\muk_4(\w^1)$, we have $\int_\W\mh_x(\w,\muk_4(\w))\,dm^1=0$,
and Theorem \ref{th:3copia} shows that $\{\muk_{4}\}$ is not
an attractive hyperbolic copy of the base. That is, $\lb_u^*=4$, as asserted.

Now, we will prove that $\lb_l^*=-1$. Let us fix $\lb\le-1$. We begin by
checking that the constant map $-\bar x_0$ with $x_0:=1/2^{1/3}$ is a global upper
solution. That is,
$\mh(\w,-x_0)+\lb\le 0$ for all $\w\in\W$, which due to the continuity of $\mh$
holds if $\mh(f{\cdot}t,-x_0)+\lb\le 0$ for all $t\in\R$; i.e., if
$h(t,-x_0)+\lb=-(-x_0+f_1(t))^3-f_2(t)\,x_0+f_3(t)+\lb\le 0$
for all $t\in\R$. This is easy to check:
\begin{itemize}[leftmargin=25pt]
\item[-] if $t\le0$, then $f_1(t)=1$, $f_2(t)=0$ and $f_3(t)=-4$,
so $h(t,-x_0)+\lb\le-(2^{1/3}-1)^3/2-4-1<0$\,;
\item[-] if $t\in(0,1)$, then $f_1(t)\ge 0$,
$f_2(t)\ge 0$ and $f_3(t)\le 0$, so $h(t,-x_0)+\lambda\le x_0^3+\lb\le 1/2-1<0$\,;
\item[-] if $t\ge 1$, then $f_1(t)=0$, $f_2(t)=a$ and $f_3(t)\le 2$,
so $h(t,-x_0)+\lambda=x_0^3-a\,x_0+f_3(t)+\lb\le 1/2-3/2+2-1=0$\,.
\end{itemize}
According to \cite[Theorem 5.1]{dno1} and \cite[Theorem 2.13(v)]{duen},
this means that $\ml_\lb\le -\bar x_0$. Hence, the
closure $\mK_\lb^l\subset\W\times(-\infty,-x_0]$ of the graph of $\ml_\lb$
is a compact $\tau_\lb$-invariant subset.
Reasoning as above, we check that $\mK_\lb^l$ has at most
three possible different Lyapunov exponents, all
of them strictly negative. Hence, Theorem \ref{th:3copia} shows that
$\mK_\lb^l$ is in fact an attractive hyperbolic copy of the base,
which contains and hence coincides with $\{\ml_\lb\}$.
Thus, $\lb_l^*\ge-1$. In this case, $\ml_{-1}(\w^1)=-x_0$ and
$(g_1)_x(-x_0)=0$ (see Remarks \ref{rm:5equilpoints} and Figure \ref{fig:5examplefigure}).
That is, $\int_\W\mh_x(\w,\ml_{-1}(\w))\,dm^1=0$,
which combined again with Theorem \ref{th:3copia} precludes $\{\ml_{-1}\}$ from
being an attractive hyperbolic copy of the base. The conclusion is that $\lb_l^*=-1$,
and it completes the proof of (i).
\smallskip\par

(ii) Let us prove that $\lb_l^*=-1$ is a jump bifurcation point.
We fix $\ep>0$, and take $(t_n)\uparrow\infty$ and
$\lim_{n\to\infty}f{\cdot}t_n=\w^3$. By taking an adequate subsequence if required,
we assume the existence of $x_{-1+\ep}:=\lim_{n\to\infty}\muk_{-1+\ep}(f{\cdot}t_n)$.
Then, the limit $(\w^3,x_{-1+\ep})$ belongs to $\mA_{-1+\ep}$. The fiber of this attractor
over $\w^3$ reduces to the unique
critical point of $x'=g_3(x)-1+\ep$, which means that $x_{-1+\ep}>x_0$:
see Remarks \ref{rm:5equilpoints} and Figure \ref{fig:5examplefigure}. We take
$t_{n_\ep}$ with $\muk_{-1+\ep}(f{\cdot}(t_{n_\ep}))>x_0$.
On the other hand, $\mh(f{\cdot}t,x_0)-1\ge 0$ for all $t\ge 1$: we must check that
$h(t,x_0)-1=-(x_0+f_1(t))^3+f_2(t)\,x_0+f_3(t)-1\ge 0$
for all $t\ge 1$, and this follows from $f_1(t)=0$, $f_2(t)=a$
and $f_3(t)\ge 0$ if $t\ge 1$, since these equalities yield $h(t,x_0)-1=-x_0^3+a\,x_0+f_3(t)-1
\ge -1/2+3/2-1=0$ for $t\ge 1$. So, $\mh(f{\cdot}t,x_0)-1+\ep\ge\ep$ for all $t\ge 1$.
Since $t\mapsto\muk_{-1+\ep}(f{\cdot}t)$ solves the equation $x'=\mh(f{\cdot}t,x)-1+\ep$,
we conclude that $\muk_{-1+\ep}(f{\cdot}t)\ge x_0$ for all $t\ge t_{n_\ep}$, and hence that
$\liminf_{t\to\infty}\muk_{-1+\ep}(f{\cdot}t)\ge x_0$. On the other hand,
we have checked in the proof of (i) that $\sup_{t\in\R}\ml_{-1-\ep}(f{\cdot}t)\le -x_0$.
Therefore, using the information of Remark \ref{rm:5equilpoints}.3,
we conclude that
\[
 \liminf_{t\to\infty}\big(\muk_{-1+\ep}(f{\cdot}t)-\muk_{-1-\ep}(f{\cdot}t)\big)
 =\liminf_{t\to\infty}\big(\muk_{-1+\ep}(f{\cdot}t)-\ml_{-1-\ep}(f{\cdot}t)\big)
 \ge 2x_0\,,
\]
It is easy to check that this implies
$\liminf_{t\to\infty}\big(\muk_{-1+\ep}((f{\cdot}s){\cdot}t)-
\muk_{-1-\ep}((f{\cdot}s){\cdot}t)\big)\ge 2x_0$ for all $s\in\R$, and to deduce
from the semicontinuity properties of $\muk_\lb$ and $\ml_\lb$ that
$\liminf_{t\to\infty}\big(\muk_{-1+\ep}(\wt)-\muk_{-1-\ep}(\wt)\big)\ge 2x_0$ for all
$\w\in\W^\upomega$. So, no matter what happens with $\w=\w^1$ (the
unique element of $\W^\upalpha$), Definition \ref{def:4jump} shows
the occurrence of a jump bifurcation at $\lb_l^*=-1$.
\smallskip\par

(iii) Let us take $\lb\in(\lb_l(\bar\w),\lb_u(\bar\w))=(-1,4)$, and
check that $\muk_\lb$ is a hyperbolic copy of the base.
We define $\mK_\lb^u:=\mathrm{cls}\{(f{\cdot}t,\muk_\lb(f{\cdot}t))\,|\; t\in\R\}$.
Let us check that $(\mK_\lb^u)_{\w^i}=\{\muk_\lb(\w^i)\}$ for $i\in\{1,2,3\}$.
Since $\mK_\lb^u$ is a $\tau_\lb$-invariant compact set,
$(\mK_\lb^u)_{\w^i}\subseteq[\ml_\lb(\w^i),\muk_\lb(\w^i)]$.
So, $\mK(\w^1)=\{\muk_\lb(\w^1)\}$, since $\ml_\lb(\w^1)=\muk_\lb(\w^1)$
(see Remark \ref{rm:5equilpoints}).
Now, for $i=2$ or $i=3$, we take $(\w^i,x^i)\in\mK_\lb^u$ and
assume without restriction that
$(\w^i,x^i)=\lim_{n\to\infty}(f{\cdot}t_n,\muk_\lb(f{\cdot}t_n))$
for $(t_n)\uparrow\infty$. So,
for all $t\in\R$, $v_\lb(t,\w^i,x^0)=\lim_{n\to\infty}\muk_\lb(f{\cdot}(t+t_n))\ge
\liminf_{t\to\infty}\muk_\lb(f{\cdot}t)\ge x_0$: the last inequality
has been checked in the proof of (ii).
Hence, $t\mapsto v_\lb(t,\w^i,x^i)$ solves $x'=g_3(x)+\lb$ and is bounded and above $x_0$.
The bifurcation diagram of this equation (see Figure \ref{fig:5examplefigure} and
Theorem \ref{th:3Dbifur}) ensures that $v(t,\w^i,x^i)$ is constantly equal to
$\muk_\lb(\w^i)$, which is the upper equilibrium point of the equation. In
particular, $x^i=\muk_\lb(\w^i)$, as asserted.

Now we reason as in the second paragraph of the proof of (i) for
$\lb\in(-1,4)$ instead of $\lb>4$ to conclude that all the Lyapunov exponents
of $\mK_\lb^u$ are strictly negative and to deduce that $\mK_\lb^u$ is an
attractive hyperbolic copy of the base. Since $\mK_\lb^u$ contains
$\{\muk_\lb\}$, both sets coincide, which completes the proof.
\end{proof}
Note that, as said after Theorem \ref{teor:4jumpbifurcationp2p3}, the sufficient
condition \eqref{eq:4equivp2} does not hold in this case of
jump bifurcation, since Remark \ref{rm:5equilpoints}.3 shows that
$\muk_{-1}(f{\cdot}t)=\ml_{-1}(f{\cdot}t)$ for all $t\in\R$.

To complete the previous proof requires to prove the next lema.
\begin{lema}\label{lema:5omega}
$\W^\upomega_{f_3}=
\{\bar 0\}\cup\{\bar 2\}\cup\{\bar f_{0\mapsto2}{\cdot}t\,|\;t\in\R\}\cup
\{\bar f_{2\mapsto0}{\cdot}t\,|\;t\in\R\}$.
\end{lema}
\begin{proof}
For each $k\in\N$, $k\ge 1$, we call
$a_k:=k^2+k-1$, $b_k:=k^2+2k-1=a_k+k$, $c_k:=k^2+2k=b_k+1$ and $d_k:=k^2+3k=c_k+k=a_{k+1}-1$.
We take $\w_3\in\W^\upomega_{f_3}$ and $(t_n)\uparrow\infty$ such that
$\lim_{n\to\infty}f_3{\cdot}t_n=\w_3$ in $\W_{f_3}$; that is,
$\lim_{n\to\infty}(f_3{\cdot}t_n)(t)=\lim_{n\to\infty}f_3(t+t_n)=\w_3(t)$
uniformly on the compact subsets of $\R$.
For each $n\in\N$, let $k_n\in\N$ be such that $t_n\in[a_{k_n},a_{k_n+1})$,
so that $\lim_{n\to\infty}k_n=\infty$, and let
\[
 l_n:=\min\{t_n-a_{k_n},\,|t_n-b_{k_n}|,\,|t_n-c_{k_n}|,\,|t_n-d_{k_n}|\}\,.
\]
If $(l_n)$ is bounded, then there exist $m>0$ and, for every $n\ge 1$, a point
$p_n\in\{a_{k_n},b_{k_n},c_{k_n},d_{k_n}\}$ such that
$|t_n-p_n|\le m$. Clearly, by changing to a suitable subsequence if needed,
there is no loss of generality in assuming from the beginning
that either $p_n=a_{k_n}$ for all $n\ge 1$, or $p_n=b_{k_n}$ for all $n\ge 1$,
or $p_n=c_{k_n}$ for all $n\ge 1$, or $p_n=d_{k_n}$ for all $n\ge 1$,
and that there exists $\lim_{n\to\infty}(t_n-p_n)=:s\in[-m,m]$. Let us check that
$\lim_{n\to\infty}f_3{\cdot}t_n=(\bar f_{2\mapsto0}){\cdot}s$ in $\W_{f_3}$
in the cases $p_n=a_{k_n}$ and $p_n=d_{k_n}$. We only detail the
proof for $p_n=d_{k_n}$ for all $n\in\N$, as the other case is analogous.
We have $\lim_{n\to\infty}(t_n-d_{k_n})=s\in[-m,m]$.
We take any compact set $[-r,r]\subset\R$ with $r>m+1$, and look for
$n$ large enough to ensure that $-k_n<t_n-d_{k_n}-r<0<1<t_n-d_{k_n}+r$,
which is possible since $-\infty<s-r<0<1<s+r$. Then, for these values of $n$,
\[
 c_{k_n}=d_{k_n}-k_n<t_n-r<d_{k_n}<d_{k_n}+1<t_n+r\,.
\]
In particular, $(d_{k_n},d_{k_n}+1)\subset [t_n-r,t_n+r]$ and $c_{k_n}<t_n-r$.
Therefore, $(f_3{\cdot}t_n)(t)=f_3(t+t_n)=0$ for $t\in[-r,d_{k_n}-t_n]$,
$(f_3{\cdot}t_n)(t)=2\,(1-q(t+t_n-d_{k_n}))$ for $t\in(d_{k_n}-t_n,d_{k_n}+1-t_n)$,
and $(f_3{\cdot}t_n)(t)=2$ for $t\in[d_{k_n}+1-t_n,r]$.
The claim follows easily from here and the expression of $(\bar f_{2\mapsto0}){\cdot}s$.
Analogously, we can prove that $\lim_{n\to\infty}(f_3{\cdot}t_n)=
(\bar f_{0\mapsto2}){\cdot}s$ in $\W_{f_3}$ in the cases $p=b_{k_n}$ and $p=c_{k_n}$.

If $(l_n)$ is not bounded, then it has a subsequence with limit $\infty$. Since
$c_k-b_k=a_{k+1}-d_k=1$, we can choose such a subsequence (and rename it as $(l_n)$)
with the additional property that either $t_n\in[a_{k_n},b_{k_n})$ for all $n$ or that
$t_n\in[c_{k_n},d_{k_n})$ for all $n$. Assume that $t_n\in[a_{k_n},b_{k_n})$ for all $n$, so that
$\lim_{n\to\infty}(t_n-a_{k_n})=\lim_{n\to\infty}(b_{k_n}-t_n)\ge\lim_{n\to\infty}l_n=\infty$
take $r>0$; look for $n_0$ such that $[t_n-t,t_n+r]\subset[a_{k_n},b_{k_n}]$ for all $n\ge n_0$,
and deduce that $(f_3{\cdot}t_n)(t)=0$ for all $t\in[-r,r]$, which means that $\w_3=\bar 0$.
Similarly, $\w_3=\bar 2$ if $t_n\in[c_{k_n},d_{k_n})$ for all $n$.

Altogether, $\W^\upomega_{f_3}=\{\bar 0\}\cup\{\bar 2\}\cup\{\bar f_{0\mapsto2}{\cdot}t\,|\;t\in\R\}\cup
\{\bar f_{2\mapsto0}{\cdot}t\,|\;t\in\R\}\}$,
as asserted.
\end{proof}
\subsection{An example of critical transition in case \ref{p3}}\label{subsec:5critical}
The relevant information about the bifurcation diagram of \eqref{eq:5typeskewproduct}
given so far allows us to present an example of a critical transition, again
related to the occurrence of a jump bifurcation point.
Cubic equations like
\eqref{eq:5type} have been widely used as models for populations subject to the
so-called Allee effect (see, e.g., \cite{courchamp08}, \cite{fmrh}), a
biological phenomenon that describes a positive correlation between a population’s
growth rate and its size. This effect helps prevent the extinction of sufficiently
large populations, thanks to cooperative mechanisms such as mating, defense, and
foraging or hunting in groups. In the example that we present, we can interpret
the sudden increase in the steady state of a population (that is intended to be kept under
strict control) upon a slight modification of hunting, as a critical transition.

Let us choose $f_1,f_2$ and $f_3$ as in Figure \ref{fig:5Lambdas} and $q$ as a $C^1$ continuation
of the map there considered taking the value $0$ for $t\le0$ and the value $1$ for $t\ge1$.
We model a size-induced tipping phenomena (see, e.g., \cite{dno4}) through a simple
parameter shift function of the form
\[
 \Lambda_d(t):=d\,q(t)\,,
\]
which takes the value $0$ for $t\le0$ and the value $d\ge 0$ for $t\ge1$.

Replacing the parameter $\lb$ with $\Lambda_d(t)$ in \eqref{eq:5type}
gives the $d$-parametric family
\begin{equation}\label{eq:5criticaltransition}
 x'=-\big(x+f_1(t)\big)^3+f_2(t)\,x+f_3(t)-\Lambda_d(t)
\end{equation}
for $d\ge 0$. We denote $h_d(t,x):=-\big(x+f_1(t)\big)^3+f_2(t)\,x+f_3(t)-\Lambda_d(t)$.

Note that, for all $d\ge0$, \eqref{eq:5criticaltransition} coincides with $x'=-(x+1)^3-4$
for all $t\leq0$, and hence the dissipative equation \eqref{eq:5criticaltransition} has a unique
globally bounded solution, $u_d(t)$ which takes the value $x_*$ for all $t\le 0$,
where $x_*:=-\sqrt[3]{4}-1$ is the unique fixed point of the autonomous equation
(see again Figure \ref{fig:5examplefigure}).
\begin{figure}[ht]
\includegraphics[width=0.7\textwidth]{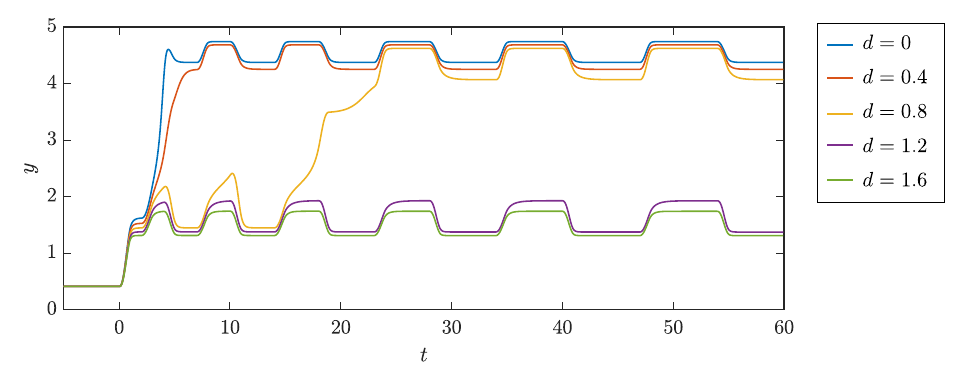}
\caption{A numerical approximation of the unique bounded solution $u_d$
of \eqref{eq:criticaltransition} is shown for several values of the parameter $d$,
written in the right panel.
The solutions are plotted after the change of variables $y=x+3$, which
ensures that $y$ remains positive, so that they can have biological
significance in population dynamics.
As $d$ crosses the threshold value $d=1$, the system undergoes a critical transition:
for $d\in[0,1)$ the solutions eventually remain above $3+2^{-1/3}$,
whereas for $d\ge1$ they eventually fall under $3-2^{-1/3}$.}
\label{fig:5criticaltran}
\end{figure}

Since $u_1(1)\approx-1.6967\le -1/2^{1/3}:=-x_0$ (as easy to check numerically),
and since $h_1(t,-x_0)\le 0$ for all $t\ge 1$ (see the proof of
Theorem \ref{teor:5withjump}(i)), we conclude that $u_1(t)\leq-x_0$ for all $t\ge1$.
It follows from the decreasing character of
$d\mapsto h_d$ that $u_{d_1}\ge u_{d_2}$ whenever $0\le d_1\le d_2$,
and hence $u_d(t)\leq-x_0$ for all $t\ge 1$ if $d\ge 1$.
On the contrary, when $d\in[0,1)$, $u_d(t)\geq x_0$ for all $t\ge t_d>0$,
as we can check by reasoning as in the proof of Theorem \ref{teor:5withjump}(ii).

So, if $d\in[0,1)$, the unique bounded solution
eventually gets trapped above $x_0$, while, if $d\ge1$, the unique bounded solution
is eventually trapped below $-x_0$.
The jump bifurcation present in the underlying bifurcation diagram \eqref{eq:5type}
gives rise to this abrupt change of behavior as the parameter $d$ crosses the
threshold $d=1$.
Figure \ref{fig:5criticaltran} represents the dynamical behavior
of these solutions, which may mean a critical transition in population dynamics.
\section{A non-trivial jump bifurcation in case \ref{p2} over a minimal base}\label{sec:6}
In this section, we will construct a simple example of a coercive nonautonomous cubic
polynomial (and hence $d$-concave) ODE which is recurrent in time,
so that its hull supports a minimal flow, and which defines a skewproduct flow
admitting two separated nonhyperbolic minimal sets. Analyzing the
bifurcation diagram of the family obtained by adding $+\lb$ to the right-hand side,
we will check that this family falls into case \ref{p2}. Therefore,
it exhibits a jump bifurcation at $\lb=0$, in accordance with Theorem
\ref{teor:4jumpbifurcationp2p3}(i). We also include several simulations illustrating
the system’s behavior for a specific base flow provided by the Harper map, as well
as some possible critical jump-type transitions that may occur.

Let $(\W,\sigma)$ be a minimal flow which has exactly two ergodic
measures $m_-$ and $m_+$. The projective flows defined by the pioneer
examples of Million\v{s}\v{c}ikov \cite{milon} and Vinograd \cite{vinograd} of
nonuniformly hyperbolic two-dimensional lineal Hamiltonian systems
contain minimal sets with this property, and we will describe a
flow with the same property in more detail in Section~\ref{subsec:harper}.

Let $\ma_1\colon\W\to\mathbb{R}$ be a continuous map with
$\int_\W\ma_1(\w)\, dm_-<0<\int_\W\ma_1(\w)\,dm_+$. The general existence of such a map follows
from, e.g., Lemma 4.10 of \cite{dno2}, and a particular example will be built
in Section \ref{subsec:harper}. We choose $\ma_0>0$ such that
\begin{equation}\label{eq:checkcondition}
\ma_0+\int_\W\ma_1(\w)\, dm_-<0<\ma_0+\int_\W\ma_1(\w)\,dm_+\,,
\end{equation}
and consider the family of cubic polynomial ODEs
\begin{equation}\label{eq:6cubic}
x'=-x\,(x-1)(\ma_0\,x+\ma_1(\wt))\,,\quad\w\in\W\,.
\end{equation}
We call $\mh(\w,x):=-x\,(x-1)(\ma_0\,x+\ma_1(\w))$, observe that it satisfies
\ref{d1}, \ref{d2}, \ref{d3}, \ref{d4} and \ref{d5}, and denote by $\tau_0$
the skewproduct flow induced by \eqref{eq:6cubic} on $\WR$,
with $\tau_0(t,\w,x)=(\wt,v_0(t,\w,x))$. It is clear that the constant maps
$\bar 0,\,\bar 1\colon\W\to\R,\,\w\mapsto 0,1$ define two $\tau_0$-copies of $\W$, $\{\bar 0\}$
and $\{\bar 1\}$. Since
$\mh_x(\w,0)=\ma_1(\w)$ and $\mh_x(\w,1)=-\ma_0-\ma_1(\w)$, the choices of $\ma_1$ and $\ma_0$ yield
\begin{equation}\label{eq:ergodic_measures_of_0_and_1}
\begin{split}
 \int_\W \mh_x(\w,0)\,dm_-&<0<\int_\W\mh_x(\w,0)\,dm_+\,,\\
 \int_\W\mh_x(\w,1)\,dm_+&<0<\int_\W\mh_x(\w,1)\,dm_-\,.
\end{split}
\end{equation}
So, each one of the copies of $\W$ has exactly two Lyapunov exponents
of different signs: see, e.g., \cite[Section 2.1]{dno5}.
Hence, Theorems \ref{th:3copia}, \ref{th:3casosS} and \ref{th:3Dbifur}
show that
\begin{prop}\label{prop:5dosminimales}
$\{\bar 0\}$ and $\{\bar 1\}$ are two nonhyperbolic $\tau_0$-minimal sets,
and there is no other compact $\tau_0$-invariant set separated from both of them.
\end{prop}
Since $\mh(\w,x)$ satisfies \ref{d1} and \ref{d2}, $\tau_0$ has a global attractor,
$\mA_0=\bigcup_{\w\in\W}\{\w\}\times[\ml_0(\w),\muk_0(\w)]$, where $\ml_0,\muk_0\colon\W\to\R$ are
lower and upper semicontinuous equilibria, respectively: see Remark \ref{rm:3pre}.1. In addition,
\begin{prop}\label{prop:5solotres}
The unique $m_+$-almost always strictly ordered $m_+$-measurable $\tau_0$-equilibria are $\ml_0$,
$\bar 0$ and $\bar 1$, with $\int_\W\mh_x(\w,\ml_0(\w))\,dm_+<0$; and
the unique $m_-$-almost always strictly ordered $m_-$-measurable equilibria are
$\bar 0$, $\bar 1$ and $\muk_0$, with\linebreak $\int_\W\mh_x(\w,\muk_0(\w))\,dm_-<0$.
In particular, $\muk_0=1$ $m_+$-a.a., $\muk_0>1$ $m_-$-a.a., $\ml_0=0$ $m_-$-a.a., and
$\ml_0<0$ $m_-$-a.a.
\end{prop}
\begin{proof} Since $\bar 0$ and $\bar 1$ give constant solutions for all $\w$, $\ml_0\le 0<1\le\muk_0$.
Moreover, according to \cite[Proposition 5.2]{dno1},
$\int_\W\mh_x(\w,\ml_0(\w))\,dm_\pm\le0$ and $\int_\W\mh_x(\w,\muk_0(\w))\,dm_\pm\le0$, and
\cite[Proposition 5.2]{dno4} establishes that there exist at most three bounded $m_\pm$-measurable
$\tau_0$-equilibria which are strictly ordered $m_\pm$-a.e.
Observe that, since $m_\pm$ is ergodic, the $\sigma$-invariant subset of $\W$ where two
$m_\pm$-measurable equilibria are equal has measure $0$ or $1$.
The first assertion follows from these properties, since $\int_\W\mh_x(\w,0)\,dm_+>0$
and $\int_\W\mh_x(\w,\ml_0(\w))\, dm_\pm\le0$. The proof of the second one is analogous.
Once established these properties, \cite[Proposition 5.2]{dno4}
ensures that $\int_\W\mh_x(\w,\ml_0(\w))\,dm_+<0$ and $\int_\W\mh_x(\w,\muk_0(\w))\,dm_-<0$.
The last assertions are easy consequences of the initial ones.
\end{proof}
\subsection{The dynamics within $\W\times[0,1]$}\label{subsec:between0and1}
Our next result shows the complexity of the dynamics in the invariant region $\W\times(0,1)$, where
the flow $\tau_0$ is globally defined: there coexist
three $\sigma$-invariant disjoint sets: $\W_-$ and $\W_+$, with $m_-(\W_-)=m_+(\W_+)=1$, such that if
$\w\in\W_-$ (resp.~$\w\in\W_+$), then $(\wt,v_0(t,\w,x))$ travels from $\{\bar 1\}$ to $\{\bar 0\}$
(resp.~and from $\{\bar 0\}$ to $\{\bar 1\}$) as time increases for all $x\in(0,1)$;
and $\mR$, $\sigma$-invariant and residual with $m_-(\mR)=m_+(\mR)=0$,
for which all the corresponding orbits strongly oscillate between both copies of $\W$
as time increases and also as time decreases.
\begin{teor}\label{th:5banda}
There exist
\begin{itemize}
\item[\rm(i)] two $\tau_0$-invariant disjoint subsets $\W_-,\W_+\subseteq\W$
with $m_-(\W_-)=m_+(\W_+)=1$ and $m_-(\W_+)=m_+(\W_-)=0$
such that, for all $x\in (0,1)$,
\begin{equation}\label{eq:behavior_downwards}
\begin{split}
 \lim_{t\to-\infty} v_0(t,\w,x)=0\quad\text{and}\quad\lim_{t\to\infty} v_0(t,\w,x)=1
 \quad \text{if $\w\in\W_+$}\,,\\
 \lim_{t\to-\infty} v_0(t,\w,x)=1\quad\text{and}\quad\lim_{t\to\infty} v_0(t,\w,x)=0
 \quad \text{if $\w\in\W_-$}\,;
\end{split}
\end{equation}
\item[\rm(ii)] and a residual $\tau_0$-invariant subset $\mR\subseteq\W\setminus(\W_-\cap\W_+)$
such that, for every $(\w,x)\in\mR\times(0,1)$,
\[
\begin{split}
 \limsup_{t\to-\infty} v_0(t,\w,x)&=\limsup_{t\to\infty}v_0(t,\w,x)=1
 \,,\\
 \liminf_{t\to-\infty} v_0(t,\w,x)&=\liminf_{t\to\infty}v_0(t,\w,x)=0\,.
\end{split}
\]
\end{itemize}
\end{teor}
\begin{proof}
(i) Since $\int_\W\mh_x(\w,0)\,dm_+>0$, Lemma \ref{lema:6medidapositiva} (proved below)
shows the existence of $\W_0\subseteq\W$ with $m_+(\W_0)=1$ such that the map
\[
 \mb_+(\w):=\sup\left\{x\in[0,1]\,\Big|\,\lim_{t\to-\infty} v_0(t,\w,x)=0\right\}
\]
satisfies $\mb_+(\w)>0$ for all $\w\in\W_0$. It is not hard to check that
$\W_0$ is $\tau_0$-invariant: in fact, $\mb_+$ is a Borel measurable
(and hence $m_+$-measurable) $\tau_0$-equilibrium. Since $0<\mb_+\le 1$ in full $m_+$ measure,
Proposition \ref{prop:5solotres} shows that it coincides with the constant equilibrium $\bar 1$
in a $\tau_0$-invariant set $\W_+$ with $m_+(\W_+)=1$.
This implies the first equality for $\W_+$ in \eqref{eq:behavior_downwards}. The second one follows from
an analogous reasoning with forward time and $\int_\W\mh_x(\w,1)\,dm_+<0$, and a possible
refinement of $\W_+$. The remaining ones are obtained with the same arguments applied to $m_-$.
Since the behaviors described in the first and second lines of \eqref{eq:behavior_downwards} are incompatible,
the set  $\W_-\cap\W_+$ is empty, and hence $m_+(\W_-)=0$ and $m_-(\W_+)=0$.
\smallskip

(ii) Note that the \upomeg-limit set $\upomega(\w,x)$ (resp. \upalfa-limit set $\upalpha(\w,x)$)
of any $(\w,x)\in\W\times[0,1]$ is contained in $\W\times[0,1]$ and projects onto $\W$.
If $\upomega(\w,x)\subset\W\times(0,1)$, then the upper $\tau_0$-equilibria
of this compact set would be an $m_\pm$-measurable equilibrium strictly between $\bar 0$ and $\bar 1$,
which does not exist: see Proposition \ref{prop:5solotres}. Thus, $\upomega(\w,x)$
intersects either $\{\bar 0\}$ or $\{\bar 1\}$, and hence it contains at least one of them.
The same applies to $\upalpha(\w,x)$. The set $\mR$ will be composed by those points $\w$ for which
$\upalpha(\w,x)\cap\upomega(\w,x)\supseteq \{\bar 0\}\cup\{\bar 1\}$ for all $x\in(0,1)$.

For each $x\in(0,1)$, we define
\[
\W_{m,n}^x:=\left\{
\w\in\W\,|\;\, v_0(t,\w,x)\ge 1/m\text{ for all }t\ge n
\right\}\,,
\]
which is clearly closed.
According to \eqref{eq:behavior_downwards}, $\W_{m,n}^x\supseteq\W_+$, and hence it is nonempty,
while $\W_{m,n}^x\cap\W_-$ is empty. This last property ensures that the interior of $\W_{m,n}^x$ is empty:
otherwise $m_-(\W_{m,n}^x)>0$, since $m_-$ is ergodic, and this would preclude
$\W_{m,n}^x\subseteq\W\setminus\W_-$.
Therefore, $\W_{m,n}^x$ is nowhere dense for all $x\in(0,1)$ and $m,n\in\N^+$.

Note that
\[
 \bigcup_{m,n\in\mathbb{N}^+}\W_{m,n}^x=\big\{\w\in\W\,|\;\{\bar 0\}\not\subseteq\upomega(\w,x)\big\}\,:
\]
the $\subseteq$ inclusion is clear; and if $\w\in\W$ belongs to the set on the right, then
$\liminf_{n\to\infty} v_0(n,\w,x)>0$ and so $\w\in\W^x_{m,n}$ for some $m,n\in\N^+$. Hence,
\[
 \bigcup_{q\in\Q\,\cap\,(0,1)}\bigcup_{m,n\in\mathbb{N}^+}\W_{m,n}^q=\big\{\w\in\W\;|\;\exists\,q\in\Q\,\cap\,(0,1)\text{ such that }
 \{\bar 0\}\not\subseteq\upomega(\w,q)\big\}
\]
is a first category set.
Analogously,
\[
 \big\{\w\in\W\;|\;\exists\,q\in\Q\,\cap\,(0,1)\text{ such that }\{\bar 1\}\not\subseteq\upomega(\w,q)\big\}
\]
is a first category set and, therefore, the complement of the union,
\[
 \big\{\w\in\W\;|\;\{\bar 0\}\cup\{\bar 1\}\subseteq\upomega(\w,q)\;\,\forall\, q\in(0,1)\cap\Q\big\}
\]
is a residual subset of $\W$. Let us check that the previous set coincides with
\[
 \mR_\upomega:=\big\{\w\in\W\;|\;\{\bar 0\}\cup\{\bar 1\}\subseteq\upomega(\w,x)\;\,\forall\, x\in(0,1)\big\}\,,
\]
so $\mR_\upomega$ is a residual subset of $\W$: given $x\in(0,1)$ we find $q_1\in(0,x)\cap\Q$ and $q_2\in(x,1)\cap\Q$,
as well as sequences $(t_n^1)\uparrow\infty$ and $(t_n^2)\uparrow\infty$ with
$\lim_{n\to\infty}v_0(t_n^1,\w,q_1)=1$ and  $\lim_{n\to\infty}v_0(t_n^2,\w,q_2)=0$; so, since
$v_0(t_n^1,\w,q_1)\le v_0(t_n^1,\w,x)\le 1$ and $v_0(t_n^2,\w,q_2)\ge v_0(t_n^2,\w,x)\ge 0$, we conclude that
$\upomega(\w,x)$ intersects (and hence contains) both $\{\bar 1\}$ and $\{\bar 0\}$.

Reasoning analogously, we prove that
\[
 \mR_\upalpha:=\big\{\w\in\W\;|\;\{\bar 0\}\cup\{\bar 1\}\subseteq\upalpha(\w,x)\;\forall x\in(0,1)\big\}
\]
is a residual subset of $\W$, and therefore $\mR:=\mR_\upomega\cap\mR_\upalpha$
is also residual. Clearly, if $\w\in\mR$ and $x\in(0,1)$, then
$v_0(t,\w,x)$ approaches $\bar 0$ and $\bar 1$ both as $t\to\infty$ and $t\to-\infty$,
since $\upalpha(\w,x)\cap\upomega(\w,x)\supseteq \{\bar 0\}\cup\{\bar 1\}$.
\end{proof}
The previous proof uses the next lemma,
which follows from the general theory of Pesin (see, e.g., \cite[Theorem 7.3.10]{arnold}) and
is similar in spirit to the discrete result given in \cite[Proposition 3.3]{jager1}.
\begin{lema} \label{lema:6medidapositiva}
Let $m\in\merg$ be fixed, and let $\mb\colon\W\to\R$ be a bounded $m$-measurable equilibrium.
If
\[
\int_\W\mh_x(\w,\mb(\w))\,dm>0\qquad\left(\text{resp. }\int_\W\mh_x(\w,\mb(\w))\,dm<0\right)\,,
\]
then there exist a set $\W_0\subseteq\W$ with $m(\W_0)=1$ and, for each $\w\in\W_0$,
a constant $\rho(\w)>0$ such that $\lim_{t\to-\infty} (v(t,\w,x)-\mb(\wt))=0$
(resp. $\lim_{t\to\infty} (v(t,\w,x)-\mb(\wt))=0$) whenever $|x-\mb(\w)|<\rho(\w)$.
\end{lema}
\begin{proof}
We assume that $\int_\W\mh_x(\w,\mb(\w))\,dm>\ep$ for an $\ep>0$ and take $\delta>0$ such that
\begin{equation}\label{eq:2firstbound}
 \mh_x(\w,\mb(\w)+x)\geq \mh_x(\w,\mb(\w))-\ep
\end{equation}
for all $\w\in\W$ and $|x|\leq\delta$.
Birkhoff's Ergodic Theorem yields the existence of $\W_0\subseteq\W$ with $m(\W_0)=1$ such that,
for all $\w\in\W_0$,
\[
 \lim_{t\to-\infty}\frac{-1}{t}\int_t^0\big(\mh_x(\ws,\mb(\ws))-\ep\big)\,ds=
 \int_\W\big(\mh_x(\w,\mb(\w))-\ep\big)\,dm>0
\]
and hence
\begin{equation}\label{eq:2limitproof}
\lim_{t\to-\infty}\left(-\int_t^0\mh_x(\ws,\mb(\ws))\,ds-\ep\,t\right)=-\infty\,.
\end{equation}
So, if $\w\in\W_0$,
\[
 k(\w):=\sup_{t\leq0}\;\exp\left(-\int_t^0\mh_x(\ws,\mb(\ws))\,ds\,-\ep\,t\right)\in(0,\infty)\,.
\]
Let us define $y(t,\w,x):=v(t,\w,x)-\mb(\wt)$. By integrating in $[0,1]$ the derivative with respect to
$r$ of $\mh(\wt,\mb(\wt)+r\,y(t,\w,x))$, we obtain $(d/dt)\,y(t,\w,x)=
y(t,\w,x)\int_0^1 \mh_x(\wt,\mb(\wt)+r\,y(t,\w,x))\,dr$, and so
\[
y(t,\w,x)=(x-\mb(\w))\exp\left(-\int_t^0\int_0^1\mh_x\big(\ws,\mb(\ws)+r\,y(s,\w,x)\big)\big)\,dr\,ds\right).
\]
We take $x$ with $|x-\mb(\w)|<\rho(\w):=\delta/k(\w)<\delta$ and deduce from
\eqref{eq:2firstbound} that, if $|y(t)|=|v(s,\w,x)-\mb(\ws)|\le\delta$ for all $s\in[t,0]$, then
\[
\begin{split}
 &|v(t,\w,x)-\mb(\wt)|\leq|x-\mb(\w)|\,\exp\left(-\!\int_t^0\!\mh_x(\ws,\mb(\ws))\,ds-\ep\,t\!\right)\\
 &\qquad\qquad\le k(\w)|x-\mb(\w)|<\delta\,.
\end{split}
\]
An easy contradiction argument ensures that the inequality holds for all $t\le 0$,
and hence \eqref{eq:2limitproof} proves the assertion of the lemma.
The proof is analogous in the case of negative integral.
\end{proof}
\subsection{The bifurcation diagram}\label{subsec:minimallbbifurcation}
Let us describe the bifurcation diagram of
\begin{equation}\label{eq:familyjump}
 x'=-x\,(x-1)(\ma_0\,x+\ma_1(\wt))+\lb\,,
\end{equation}
with $\ma_0$ and $\ma_1$ described at the beginning of Section \ref{sec:6}.
As usual, $\tau_\lb(t,\w,x)=(\wt,v_\lb(t,\w,x))$ represents the flow on $\WR$, which is globally
forward defined and has a global attractor $\mA_\lb$, since \ref{d1} and \ref{d2} hold: see
Remark \ref{rm:3pre}.1.
\begin{prop}
For $\lb=0$, $\{\bar 0\}$ and $\{\bar 1\}$ are nonhyperbolic
minimal sets and the unique separated compact $\tau_0$-invariant sets.
For any $\lb\ne0$, $\mA_\lb=\{\muk_\lb\}=\{\ml_\lb\}$,
it is an attractive hyperbolic $\tau_\lb$-copy of $\W$,
it is strictly above 1 for $\lb>0$ and strictly below $0$ for $\lb<0$,
and $\lim_{t\to\infty}(v_\lb(t,\w,x)-\muk_\lb(\wt))=0$
for any $(\w,x)\in\W\times\R$. In particular,
$\lb_l^*=\lb_u^*=0$ is a jump bifurcation point for \eqref{eq:familyjump},
which is hence in case \ref{p2}.
\end{prop}
\begin{proof}
The assertion for $\lb=0$ is proved in Proposition \ref{prop:5dosminimales}.

We work for a fixed $\lb>0$. Since $\mh(\w,1)+\lb=\lb>0$, $\bar 1$ is a strong
$\tau_\lb$-subequilibrium, while $\lim_{x\to\infty}\mh(\w,x)=-\infty$ uniformly on $\W$
provides $\rho>1$ with $\mh(\w,\rho)+\lb\le0$, which defines a $\tau_\lb$-superequilibrium.
Let $\mK^u_\lb\subset\W\times(1,\rho]$ be
the $\tau_\lb$-invariant compact set provided by Proposition \ref{prop:2bocata}.
Since $\int_\W\mh_x(\w,1)\,dm_->0$ and $\int_\W\mh_x(\w,0)\,dm_+>0$,
\cite[Proposition 5.9]{dno1} yields $\int_\W \mh_x(\w,\mb_\lb(\w))\,dm_-<0$ and
$\int_\W \mh_x(\w,\mb_\lb(\w))\,dm_+<0$
for any $m_-$-measurable (resp. $m_+$-measurable) $\tau_\lb$-equilibrium $\mb_\lb\colon\W\to\R$
with graph in $\mK_\lb^u$. So, the upper Lyapunov exponent of $\mK^u_\lb$ is strictly negative
(see, e.g., \cite[Section 2.1]{dno5}), and hence Theorem \ref{th:3copia} ensures that $\mK_\lb^u$
is an attractive hyperbolic $\tau_\lb$-copy of $\W$, say $\{\mb_\lb\}$.

Since $\W$ is minimal and the dynamics for $\lb=0$ precludes an S-shaped
bifurcation diagram (see Theorem \ref{th:3Dbifur}), we can deduce from Theorem \ref{th:3casosS}
that there is no compact $\tau_\lb$-invariant set uniformly separated from $\{\mb_\lb\}$:
a hyperbolic copy of $\W$ would fit case (c), and a nonhyperbolic one
would fit cases (d) or (e). (Recall that the minimality of $(\W,\sigma)$ ensures that
any compact $\tau_\lb$-invariant set projects onto $\W$).
So, for any $(\w,x)\in\WR$, $\upomega(\w,x)$ intersects and
therefore contains $\{\mb_\lb\}$. Since $\{\mb_\lb\}$ is uniformly exponentially stable at $+\infty$
(see Definition \ref{def:unifexpstable}),
$\lim_{t\to\infty}(v_\lb(t,\w,x)-\mb_\lb(\wt))=0$, and this ensures that
$\mA_\lb=\{\mb_\lb\}$ is the $\tau_\lb$-global attractor. The proof is analogous for $\lb<0$.
The last assertion follows easily from definitions \eqref{def:4lb*},
the description of case \ref{p2}, and Definition \ref{def:4jump}. (In fact,
condition \eqref{eq:4equivp2} of Theorem \ref{teor:4jumpbifurcationp2p3} holds.)
\end{proof}
We complete the description by saying that
$\R\setminus\{0\}\to C(\W,\R),\;\lb\mapsto\muk_\lb=\ml_\lb$ is continuous on the uniform topology
of $C(\W,\R)$ (see, e.g., \cite[Theorem 2.3]{dno4}); that
\[
\lim_{\lb\to0^+}\muk_\lb(\w)=\lim_{\lb\to0^+}\ml_\lb(\w)=\muk_0(\w)\ge1\,,\quad
\lim_{\lb\to0^-}\ml_\lb(\w)=\lim_{\lb\to0^-}\muk_\lb(\w)=\ml_0(\w)\le0
\]
for all $\w\in\W$ (see \cite[Theorem 5.5]{dno1}); and that
$\lb\mapsto\ml_\lb(\w),\muk_\lb(\w)$ are both strictly increasing on $\R$ for all $\w\in\W$
(see again Remark~\ref{rm:3pre}.1).
\subsection{A numerical example using Harper map}\label{subsec:harper}
Now we present a numerical example of the jump bifurcation just described.
Constructing examples of minimal sets supporting multiple
ergodic measures is generally not an easy task. For this reason, we will use one
of the simplest examples in the literature that exhibits this property:
such a minimal set appears in the discrete skewproduct flow induced by
the Harper map, a model with significant connections to quantum mechanics,
which, moreover, is more suitable for numerical simulations
than the examples already mentioned by Million\v{s}\v{c}ikov \cite{milon} and Vinograd \cite{vinograd}.
We will work with its angular formulation, that is, with
\begin{equation}\label{eq:harper}
\begin{split}
\bar \theta&=\theta+\alpha\,,\\
\bar \phi&=\arccot\left(-\tan(\phi)-\lb-2\,b\cos(2\pi\theta)\right).
\end{split}
\end{equation}
The iterations of the map $T(\theta,\phi):=(\bar\theta,\bar\phi)$ define a discrete (skewproduct)
flow $\varsigma$ on $\To\times\pr$, where $\To\equiv\R/\Z$ is the one-dimensional torus and
$\pr\equiv \R/\pi\Z$ is the one-dimensional projective line: $\varsigma(n,\theta,\phi)=T^n(\theta,\phi)$.
The dependence on $b$ and $\lb$ and $\alpha$ is omitted to simplify the notation.
The parameter $\lb$ plays the role of the energy in the model, that is,
it is the spectral parameter. The value $b$ controls the strength of the quasiperiodic
forcing and determines how strongly the angle dynamics is influenced by the external modulation.
Finally, $\alpha$ is an irrational real number, so the base flow on $\To$ is minimal
and uniquely ergodic.

The map $\bar\phi$ of \eqref{eq:harper} can be obtained by taking
$\phi:=\arccot(z_2/z_1)$ from the map on $\R^2$ given by
$\left[\begin{smallmatrix}\bar z_1\\\bar z_2\end{smallmatrix}\right]
=\left[\begin{smallmatrix}\;\;0&\,1\\-1&\,-\lb-2\,b\cos(2\pi\theta)\end{smallmatrix}\right]
\left[\begin{smallmatrix}z_1\\z_2\end{smallmatrix}\right]$.
The $n$-th iteration of $\left(\theta,\left[\begin{smallmatrix}z_1\\z_2\end{smallmatrix}\right]\right)
\mapsto\left(\bar\theta,\left[\begin{smallmatrix}\bar z_1\\\bar z_2\end{smallmatrix}\right]\right)$ provides
$\left(\theta(n),\left[\begin{smallmatrix}z_1(n)\\z_2(n)\end{smallmatrix}\right]\right)$, with
$z(n+1)=-z(n-1)-2b\cos 2\pi(\theta+n\alpha)\,z(n)-\lb\,z(n)$, and this equality for all $n\in\N$
can be identified with the
spectral problem $\mathcal L_\theta z=\lb z$ for the {\em almost Mathieu operator\/} on
$\ell^2(\mathbb{Z})$, given by
\[
 (\mathcal L_\theta z)(n):=-z(n+1)-z(n-1)-2\,b\cos(2\pi(\theta+n\alpha))\,z(n)\,.
\]
It is well-known that the {\em spectrum\/} of $\mL_\theta$ (i.e.,
the set of $\lb\in\C$ such that $\mL_\theta-\lb$ does not admit a
bounded inverse operator) is a compact subset of $\R$ and is common for all $\theta\in\To$.
Since $\alpha\notin\Q$, this spectrum is a Cantor set: see
\cite{tenmartini}.

It is also known (see, e.g., \cite{haropuig} and \cite{jnot}) that,
whenever $|b|\ge1$ and $\lb$ belongs to the spectrum, there exists
a unique $\varsigma$-minimal subset $\Lambda\subseteq\To\times\pr$
which supports exactly two different $\varsigma$-ergodic measures $\tilde m_\pm$.
They are given by
\[
 \int_{\To\times\pr}\mg\, d\tilde m_\pm=\int_{\To}\mg(\theta,\varphi_\lb^\pm(\theta))\, dl
\]
for $\mg\in C(\To\times\pr,\R)$, where
$\varphi^\pm_\lb\colon\To\to\pr$ are two noncontinuous maps with
$\varsigma$-invariant graph within $\Lambda$ and
$l$ is the normalized Lebesgue measure on $\To$, which is the unique ergodic measure for the flow
given by the iterations of the map $\bar\theta$.
Finally, it turns out that, if $\lb$ is the lowest point of the spectrum, then
$\Lambda\subset\To\times(0,\pi)$ (see e.g.~\cite[Remark 6.5]{jnot}).
In these conditions,
if $g(\theta,\phi)=\bar\phi$ (see \eqref{eq:harper}), then
\begin{equation}\label{def:6beta}
 \int_{\To}\log\left(\frac{\partial g}{\partial\phi}(\theta,\varphi_\lb^+(\theta))\right) dl
 =-\int_{\To}\log\left(\frac{\partial g}{\partial\phi}(\theta,\varphi_\lb^-(\theta))\right)dl
 =:\beta(\lb)>0
\end{equation}
(see, e.g., \cite[Section 6]{jnot}).

Let us explain how the classical {\em suspension}
construction allows us to define, starting from $(\Lambda,\varsigma)$,
a minimal real flow $(\W,\sigma)$ on a compact metric space supporting exactly two
ergodic measures. On $\Lambda\times\R$, we consider the equivalence relation
$(\theta,\phi,s)\sim(T(\theta,\phi),s-1)$, define $\W:=(\Lambda\times\R)/\!\sim\,$,
represent the class of $(\theta,\phi,s)$ by $[\theta,\phi,s]$, and observe that each class has a
unique representative $(\theta,\phi,s)$ with $s\in[0,1)$ which will be chosen whenever
fixing one is needed. Then, $\sigma(t,[\theta,\phi,s]):=[\theta,\phi,s+t]$ defines a continuous
flow on $\W$, which satisfies $\sigma(t,[\theta,\phi,s])=
[\varsigma(n,\theta,\phi),r]$ if $s+t=n+r$ for $n\in\Z$.
So, $\sigma(n,[\theta,\phi,0])=[\varsigma(n,\theta,\phi),0]$ for all $n\in\Z$ and hence,
if we identify the point $(\theta,\phi)\in\Lambda$ with $[\theta,\phi,0]\in\W$, we recover
$\varsigma$ from $\sigma$ by restricting to discrete values of the time. We also define
$i_s\colon\Lambda\to\W,\;(\theta,\phi)\mapsto[\theta,\phi,s]$ for each $s\in[0,1]$.

It is immediate to deduce from the minimality of $(\Lambda,\varsigma)$
that the $\sigma$-orbit of any point $[\theta,\phi,s]\in\W$
is dense in $\W$, and hence $(\W,\sigma)$ is also minimal.
In what follows, we will prove the existence of exactly
two $\sigma$-ergodic measures on $\W$, $m^\pm$, defined as
\[
 \int_\W\mg([\theta,\phi,s])\,dm^\pm:=\int_0^1\int_{\To\times\pr}
 \mg\!\circ\!i_s(\theta,\phi)\,d\tilde m^\pm ds= \int_0^1\int_\To
 \mg(\theta,\varphi_\lb^\pm(\theta),s)\,dl\,ds
\]
for $\mg\in C(\W,\R)$. It is clear that $m^+\ne m^-$. Their
$\sigma$-invariance is a well-known property: see, e.g., \cite[Proposition 3.4.1]{viana}.
To check that they are $\sigma$-ergodic we take a $\sigma$-invariant subset $\mK\subseteq\W$
and observe that: a point $(\theta,\phi,s)$ with $s\in[0,1)$ belongs to $\mK$ if and only if
$(\theta,\phi,0)=\sigma(-s,(\theta,\phi,s))\in\mK$; hence, if
$\tilde\mK:=\{(\theta,\phi)\in\Lambda\,|\;(\theta,\phi,0)\in\mK\}$, we
have $m^\pm(\mK)=\tilde m^\pm(\tilde\mK)$; and $\tilde\mK$ is
$\varsigma$-invariant, which means that $\tilde m^\pm(\tilde\mK)$ is $0$ or $1$.
Conversely, given a $\sigma$-ergodic measure $m$ on $\W$, the
expression $\tilde m(\mK):=m((\mK\times[0,1])/\sim)$ for the Borel subsets of $\Lambda$
defines a $\varsigma$-invariant measure on $\Lambda$ (see, e.g., \cite[Proposition 3.4.3]{viana})
which is clearly also $\varsigma$-ergodic. It only remains to prove that
two different $\sigma$-ergodic measures $m_1$ and $m_2$
define two different $\varsigma$-ergodic measures $\tilde m_1$ and $\tilde m_2$.
We assume for contradiction that $\tilde m_1=\tilde m_2$, deduce that
$m_1(\mB\times[p,r])=m_2(\mB\times[p,r])$ for every $p,r\in\Q\cap[0,1)$ and every Borel
subset $\mB$ of $\Lambda$ (see the proof of \cite[Lemma 3.4.2]{viana}),
and conclude that $m_1=m_2$, as they coincide on a generator class of the Borel
sigma-algebra.
\begin{nota}
As already mentioned, we have chosen this particular example
to perform a numerical simulation.
Interested readers can found in \cite{jnot} a survey on the occurrence of this type
of minimal flows supporting two ergodic measures, that are closely related to the
(unfrequent) appearance of non-uniformly hyperbolic two-dimensional linear dynamics and to the
so-called Strange Nonchaotic Attractors.
\end{nota}

Once described the base $(\W,\sigma)$, we will define the map $\mathfrak{a}_1\colon\W\to\R$
giving rise to our example of equation \eqref{eq:6cubic}.
First, given $(\theta,\phi)\in\Lambda$, we define
$f_{\theta,\phi}\colon[0,1]\to\R$ as the linear interpolator taking values
$f_{\theta,\phi}(0)=\log\big((\partial g/\partial\phi)(\theta,\phi)\big)$ and
$f_{\theta,\phi}(1)=\log\big((\partial g/\partial\phi)(T(\theta,\phi))\big)$.
Second, we define $\ma_1([\theta,\phi,s]):=f_{\theta,\phi}(s)$, where we choose
a representative with $s\in[0,1)$.
\begin{figure}[ht]
\includegraphics[width=\textwidth]{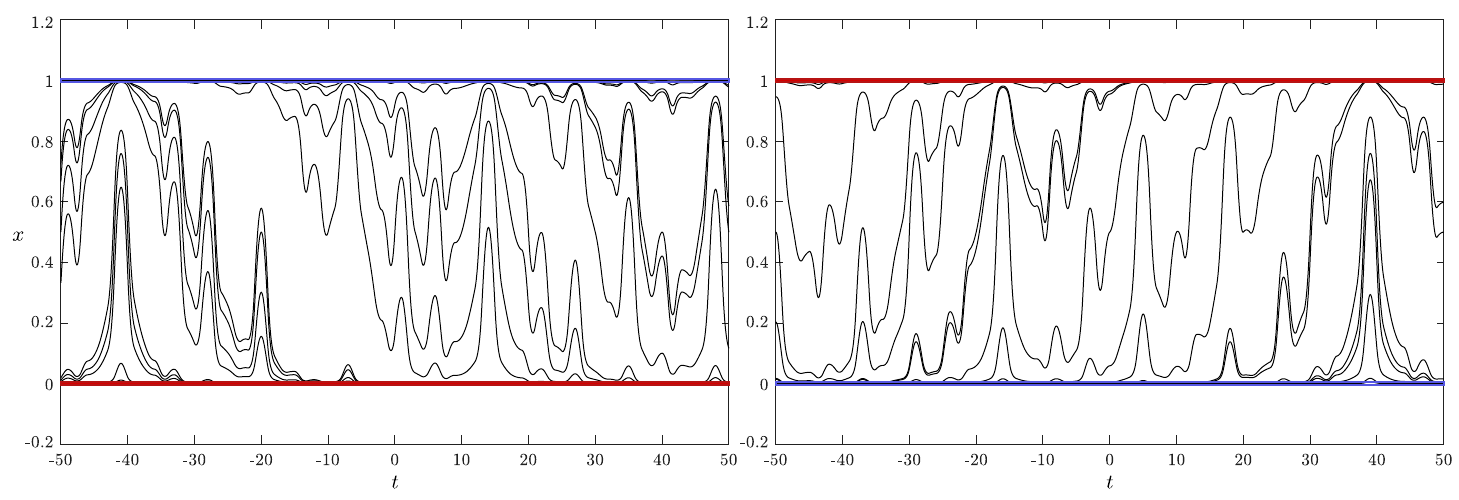}
\caption{Solutions of equation \eqref{eq:6cubic} for $b=1.1$,
$\alpha=(\sqrt{5}-1)/2$, $\lb=1.960148407660$ and $\ma_0=0.05$. On the left, we take
$\w_-:=(\theta^-,\varphi^-_\lb(\theta^-),0)\in\W$ for $\theta^-:=1/2+5200\,\alpha$
and observe that all the solutions converge to the constant solution $0$ as time increases
and to $1$ as time decreases. The map $[0,\infty)\to\R,\;t\mapsto\ma_1((\w_-){\cdot}t)$ is the
continuous piecewise linear map which takes value
$\log((\partial g/\partial\phi)(\theta^-\!+\alpha\,n,\varphi^-_\lb(\theta^-\!+\alpha\,n)))$
at $t=n$ for $n\in\N$. On the right, all the solutions travel from $0$ to $1$, due to
the choice of $\w_+:=(\theta^+,\varphi^+_\lb(\theta^+),0)\in\W$ for $\theta^+:=1/2-5200\,\alpha$,
and $t\mapsto\ma_1((\w_+){\cdot}t)$ interpolates the values
value $\log((\partial g/\partial\phi)(\theta^+\!+\alpha\,n,\varphi^+_\lb(\theta^+\!+\alpha\,n)))$.
The values of $\varphi^-$ (resp.~$\varphi^+$) have been approximated by calculating a single
forward (resp.~backward) $\varsigma$-trajectory departing from $(\theta^0,\phi_0)
=(0.5,0)$ over a long run.}
\label{fig:innerdynamics}
\end{figure}
\begin{figure}[ht]
\includegraphics[width=\textwidth]{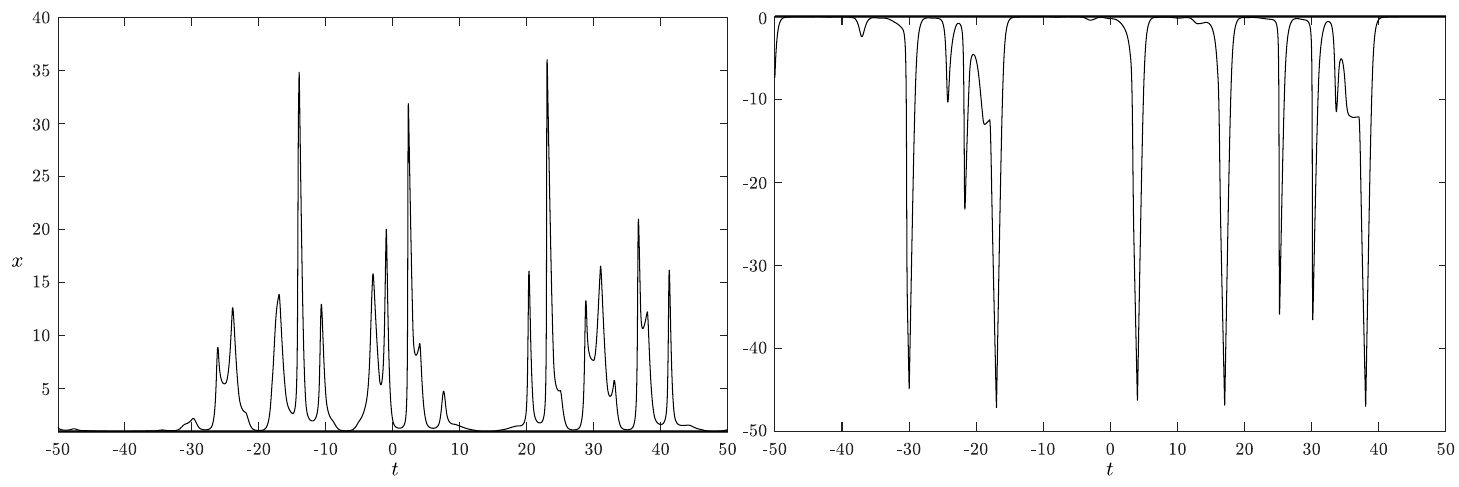}
\caption{Solutions of the same equation \eqref{eq:6cubic} as in Figure
\ref{fig:innerdynamics}.
On the left, a single solution with an initial data above $1$
at large negative initial time, and the same choice of $\w_-$ as
in Figure \ref{fig:innerdynamics}:
this solution approximates the upper bound of the attractor along a $\sigma$-trajectory
contained in the subset where $\muk_0$ differs from $1$.
On the right, a single solution with initial data below $0$ at
large negative initial time, and the same choice of $\w_+$ as in Figure
\ref{fig:innerdynamics}: this solution approximates the lower bound of the attractor along
a $\sigma$-trajectory contained in the subset where $\ml_0$
differs from $1$.}
\label{fig:outerdynamics}
\end{figure}
Since, by construction, $f_{\theta,\phi}(1)=f_{T(\theta,\phi)}(0)$, the map $\ma_1$ is
well-defined and continuous. In addition, for each $s\in[0,1]$,
\[
\begin{split}
 &\lim_{N\to\infty}\frac{1}{N}\sum_{n=0}^{N-1}
 \ma_1\!\circ\!i_s(T^n(\theta,\phi))
 =\lim_{N\to\infty}\frac{1}{N}\sum_{n=0}^{N-1}f_{T^n(\theta,\phi)}(s)\\
 &\qquad=\lim_{N\to\infty}\frac{1}{N}\sum_{n=0}^{N-1}
 \left((1-s)\log\left(\frac{\partial g}{\partial\phi}\big(T^n(\theta,\phi)\big)\right)
 +s\log\left(\frac{\partial g}{\partial\phi}\big(T^{n+1}(\theta,\phi)\big)\right)\!\right)\\
 &\qquad=\lim_{N\to\infty}\frac{1}{N}\sum_{n=0}^{N-1}
 \log\left(\frac{\partial g}{\partial\phi}\big(T^n(\theta,\phi)\big)\right)
\end{split}
\]
for each $(\theta,\phi)\in\To\times\pr$.
So, it follows from Birkhoff's Ergodic Theorem that
$\int_{\To\times\pr}\ma_1\!\circ\!i_s(\theta,\phi)\,d\tilde m^\pm=
\int_{\To\times\pr}\log\big((\partial g/\partial\phi)(\theta,\phi)\big)\,d\tilde m^\pm$
for all $s\in[0,1]$, and hence \eqref{def:6beta} and the definition of the measures $m^\pm$ yield
\begin{equation}\label{eq:5beta-2}
 \pm\int_\W\ma_1(\w)\,dm^\pm=\beta(\lb)\ne 0\,.
\end{equation}
This means that $\ma_1$ fits in the framework described at the beginning of Section \ref{sec:6}.
In addition, it can be numerically
approximated thanks to the definition of $f_{\theta,\phi}$ for $(\theta,\phi)$ in the minimal
set $\Lambda$.

Let us now represent some solutions of a particular equation
\eqref{eq:6cubic} that illustrate the dynamics.
Following \cite{jnot}, we take $b = 1.1$ and $\alpha = (\sqrt{5}-1)/2$ and
the value $\lb = 1.960148407660$ as a suitable approximation to the lowest
point of the spectrum: it is reasonably expected that the exact value
lies between the chosen one and $1.960148407661$.
Our choices determine the minimal set $\W$ and the continuous map
$\ma_1\colon\W\to\R$. We take $\ma_0=0.05$
and use \eqref{def:6beta} and \eqref{eq:5beta-2} to verify
numerically the condition \eqref{eq:checkcondition}.
The numerical evidence of \cite{jnot} also suggests
that the corresponding minimal set $\Lambda$ is contained in $\To\times(0,\pi/2)$,
as it must be if $\lb$ is indeed the lowest point of the spectrum.
Figure \ref{fig:innerdynamics} depicts several solutions of the constructed equation
\eqref{eq:6cubic} starting at the invariant set $\R\times[0,1]$
for two different points $\w_-$ and $\w_+$ of $\W$: the solutions
behave as those described in Theorem \ref{th:5banda} for $\W_-$ and
$\W_+$, respectively. And Figure \ref{fig:outerdynamics} completes the numerical analysis,
showing the behavior of solutions on the (also invariant) complement of $\R\times[0,1]$.
\subsection{An example of critical transition in case \ref{p2}}\label{subsec:6critical}
Now, we can consider and simulate a critical transition of jump-type in the previously
described model. Let $(\W,\sigma)$ be the flow described in Section \ref{subsec:harper}.
We can understand the equation
\begin{equation}
\label{eq:criticaltransition}
 x'=-x\,(x-1)\big(\ma_0\,x+\ma_1(\wt)\big)+\Gamma(t)\,,
\end{equation}
\eqref{eq:6cubic}
as the model for the evolution of a population subject to an Allee effect (due to some cooperative
effect in breeding, foraging, defense, etc.),
and influenced by a small net migration $\Gamma(t)$.
The maps $\ma_0$ and $\ma_1$ and the point $\w:
=\w_-$ are chosen as in Figure \ref{fig:innerdynamics}, and $\Gamma(t):=-\arctan(t/5)/5\pi$.

Since $\lim_{t\to-\infty}\Gamma(t)=0.1$ and $\W$ is minimal,
we can consider the following past equation of \eqref{eq:criticaltransition}:
\begin{equation}
\label{eq:criticaltransition_past}
x'=-x\,(x-1)\big(\ma_0\,x+\ma_1(\wt)\big)+0.1\,,
\end{equation}
whose upper bounded solution can be numerically checked to be
qualitatively similar to that depicted in the left picture of
Figure \ref{fig:outerdynamics}.
The arguments in Section \ref{subsec:minimallbbifurcation} ensure that the global attractor
of \eqref{eq:criticaltransition_past} is an attractive hyperbolic copy $\{\mb_-\}$ of $\W$,
and therefore there is a unique bounded solution $t\mapsto\mb_-(\wt)$ of
\eqref{eq:criticaltransition_past}. We can reason as in the proof of
\cite[Theorem 3.7]{dno3} to deduce that \eqref{eq:criticaltransition} admits exactly a
bounded solution $b(t)$ which, in addition, satisfies $\lim_{t\to-\infty}(b(t)-\mb_-(\wt))=0$.

In consonance with the established jump bifurcation at $\lb=0$ for the family
\eqref{eq:familyjump}, the dynamics of \eqref{eq:criticaltransition} suffers an abrupt change
when $\G$ crosses the value $0$, triggering a critical transition. To get a nice numerical
depiction of this phenomenon, we will not represent $x$ as state variable, but
\[
 y=\frac{\arctan(3\,(x-1/2)/4)}{\arctan(3/8)}\,.
\]
In Figure~\ref{fig:minimaltransition}, we illustrate the transformed by this change of
dependent variable of the solution $b$. The abrupt change of dynamics when the parameter
shift $\G$ crosses the bifurcation threshold $\lb=0$ is notable. The strongly oscillatory
nature of the “stationary” states observed both before and after the transition arises from the
nonuniqueness of ergodic measures on $\W$. From an ecological perspective, this regime shift can
be interpreted as a critical transition, since negative values of $b$ may be understood as population
extinction.
\begin{figure}[ht]
\includegraphics[width=0.8\textwidth]{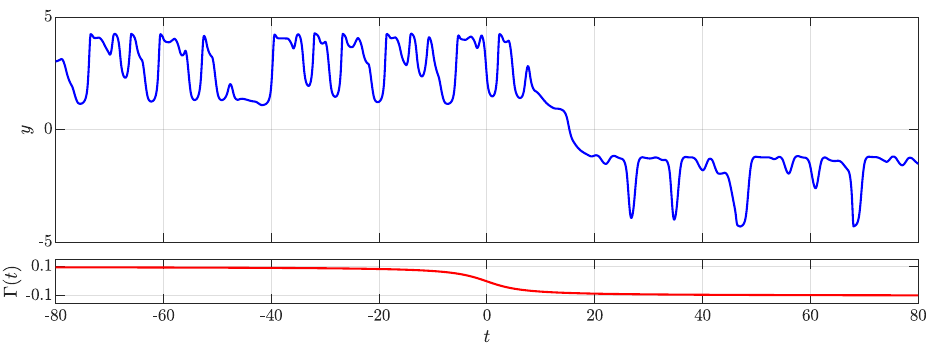}
\caption{On the lower panel, in red, the representation of the function $\G$
defined after \eqref{eq:criticaltransition}.
On the upper panel, in blue, the numerical approximation to the unique bounded
solution of \eqref{eq:criticaltransition}.
To this end, we have integrated from an initial value larger than 1 at a very negative time.
The dynamics of the solution drastically changes after $\G$ changes sign,
as expected from the analysis in the previous sections.
}
\label{fig:minimaltransition}
\end{figure}


\begin{thebibliography}{99}
\bibitem{alk2018} H.~Alkhayuon, P.~Ashwin. Rate-induced tipping from periodic attractors: Partial tipping and connecting orbits. \emph{Chaos} {\bf 28}, 033608 (2018).
\bibitem{AlkhayuonEtAl2019} H.~Alkhayuon, P.~Ashwin, L.C.~Jackson, C.~Quinn, R.A.~Wood. Basin bifurcations, oscillatory instability and rate-induced thresholds for Atlantic meridional overturning circulation in a global oceanic box model. \emph{Trans. R. Soc. A.} {\bf 475}, 20190051 (2019).
\bibitem{arnold} L.~Arnold. \emph{Random Dynamical Systems}. Springer, New York, 1998.
\bibitem{asne} P.~Ashwin, J.~Newman. Physical invariant measures and tipping probabilities for chaotic attractors of asymptotically autonomous systems. The European Physical Journal Special Topics. {\em Phys. J. Spec. Top.} {\bf 230}, 3235--3248 (2021).
\bibitem{apw2017} P.~Ashwin, C.~Perryman, S.~Wieczorek. Parameter shifts for nonautonomous systems in low dimension: bifurcation- and rate-induced tipping. \emph{Nonlinearity} {\bf 30} (6), 2185--2210 (2017).
\bibitem{awvc} P.~Ashwin, S.~Wieczorek, R.~Vitolo, P.~Cox. Tipping points in open systems: bifurcation, noise-induced and rate-dependent examples in the climate system. \emph{Phil. Trans. R. Soc. A} {\bf 370} (1962), 1166--1184 (2012).
\bibitem{tenmartini} A.~Avila, S.~Jitomirskaya. The Ten Martini Problem. \emph{Annals of Mathematics} {\bf 170}, 303--342 (2009).
\bibitem{BaltanasEtAl2003} J.P.~Baltan\'{a}s, L.~L\'{o}pez, I.I.~Blechman, P.S.~Landa, A.~Zaikin, J.~Kurths, M.A.F.~Sanju\'{a}n. Experimental evidence, numerics, and theory of vibrational resonance in bistable systems. \emph{Phys. Rev. E} {\bf 67}, 066119 (2003).
\bibitem{BerglundGentz2002} N.~Berglund, B.~Gentz. Metastability in simple climate models: pathwise analysis of slowly driven Langevin equations. \emph{Stoch. Dyn.} {\bf 2} (3), 327--356 (2002).
\bibitem{cano} J.~Campos, C.~N\'{u}\~{n}ez, R.~Obaya. Uniform stability and chaotic dynamics in nonhomogeneous linear dissipative scalar ordinary differential equations. \emph{J. Differential Equations} {\bf 361}, 248--287 (2023).
\bibitem{Cessi1994} P.~Cessi. A simple box model of stochastically forced thermohaline flow. \emph{J. Phys. Oceanogr.} {\bf 24}, 1911--1920 (1994).
\bibitem{courchamp08} F.~Courchamp, L.~Berec, J.~Gascoigne. \emph{Allee effects in ecology and conservation}. Oxford University Press, New York, 2008.
\bibitem{duen} J.~Due\~{n}as. D-concave nonautonomous differential equations and applications to critical transitions. PhD thesis, Universidad de Valladolid, 2024.
\bibitem{dlo1} J.~Due\~{n}as, I.P.~Longo, R.~Obaya. Rate-induced tracking for concave or d-concave transitions in a time-dependent environment with application in ecology. \emph{Chaos} {\bf 33} (12), 123113 (2023).
\bibitem{dno1} J.~Due\~{n}as, C.~N\'{u}\~{n}ez, R.~Obaya. Bifurcation theory of attractors and minimal sets in d-concave nonautonomous scalar ordinary differential equations. \emph{J. Differential Equations} {\bf 361}, 138--182 (2023).
\bibitem{dno3} J.~Due\~{n}as, C.~N\'{u}\~{n}ez, R.~Obaya. Critical transitions in d-concave nonautonomous scalar ordinary differential equations appearing in population dynamics. \emph{SIAM J. Appl. Dyn. Syst.} {\bf 22} (4), 2649--2692 (2023).
\bibitem{dno2} J.~Due\~{n}as, C.~N\'{u}\~{n}ez, R.~Obaya. Generalized pitchfork bifurcations in d-concave nonautonomous scalar ordinary differential equations. \emph{J. Dynam. Differential Equations} {\bf 36}, 3125--3157 (2024).
\bibitem{dno4} J.~Due\~{n}as, C.~N\'{u}\~{n}ez, R.~Obaya. Critical transitions for asymptotically concave or d-concave nonautonomous differential equations with applications in ecology. \emph{J. Nonlinear Sci.} {\bf 34}, 105 (2024).
\bibitem{dno5} J.~Due\~{n}as, C.~N\'{u}\~{n}ez, R.~Obaya. Saddle–node bifurcations for concave in measure and d-concave in measure skewproduct flows with applications to population dynamics and circuits. \emph{Commun. Nonlinear Sci. Numer. Simul.} {\bf 142}, 108577 (2025).
\bibitem{dno6} J.~Due\~{n}as, C.~N\'{u}\~{n}ez, R.~Obaya. Concave–convex nonautonomous scalar ordinary differential equations: from bifurcation theory to critical transitions. \emph{Nonlinearity} {\bf 38}, 095030 (2025).
\bibitem{feni} N.~Fenichel. Geometric singular perturbation theory for ordinary differential equations. \emph{J. Differential Equations} {\bf 31}, 53--98 (1979).
\bibitem{fmrh} R.~Fern\'{a}ndez-Mart\'{\i}nez, A.~Ruiz-Herrera. Mosquito population suppression models with seasonality and d-concave equations. \emph{J. Dynam. Differential Equations} (2026), 10.1007/s10884-025-10481-z.
\bibitem{gohe} W.H.~Gottschalk, G.A.~Hedlund. Topological Dynamics. \emph{Amer. Math. Soc. Colloq. Publ.} {\bf 36}, Amer. Math. Soc., Providence, 1955.
\bibitem{HankelTziperman2023} C.~Hankel, E.~Tziperman. An approach for projecting the timing of abrupt winter Arctic sea ice loss. \emph{Nonlin. Processes Geophys.} {\bf 30}, 299--309 (2023).
\bibitem{haropuig} \`{A}.~Haro, J.~Puig. Strange nonchaotic attractors in Harper maps. \emph{Chaos} {\bf 16}, 033127 (2006).
\bibitem{Hill1936} A.V.~Hill. Excitation and accommodation in nerve. \emph{Proc. R. Soc. Lond. B Biol. Sci.} {\bf 119} (814), 305--355 (1936).
\bibitem{HohlEtAl1995} A.~Hohl, H.J.C.~van der Linden, R.~Roy, G.~Goldsztein, F.~Broner, S.H.~Strogatz. Scaling laws for dynamical hysteresis in a multidimensional laser system. \emph{Phys. Rev. Lett.} {\bf 74}, 2220--2223 (1995).
\bibitem{jager1} T.H.~J\"{a}ger. Quasiperiodically forced interval maps with negative Schwarzian derivative. \emph{Nonlinearity} {\bf 16} (4), 1239--1255 (2003).
\bibitem{jonnf} R.~Johnson, R.~Obaya, S.~Novo, C.~N\'{u}\~{n}ez, R.~Fabbri. \emph{Nonautonomous Linear Hamiltonian Systems: Oscillation, Spectral Theory and Control}. Developments in Mathematics {\bf 36}, Springer, 2016.
\bibitem{jnot} \`{A}.~Jorba, C.~N\'{u}\~{n}ez, R.~Obaya, J.C.~Tatjer. Old and new results on Strange Nonchaotic Attractors. \emph{Int. J. Bifurcation Chaos} {\bf 17} (11), 3895--3928 (2007).
\bibitem{JungGrayRoy1990} P.~Jung, G.~Gray, R.~Roy. Scaling law for dynamical hysteresis. \emph{Phys. Rev. Lett.} {\bf 65}, 1873--1876 (1990).
\bibitem{lno2} I.P.~Longo, C.~N\'{u}\~{n}ez, R.~Obaya. Critical transitions in piecewise uniformly continuous concave quadratic ordinary differential equations. \emph{J. Dynam. Differential Equations} {\bf 36} (3), 2153--2192 (2024).
\bibitem{lnor} I.P.~Longo, C.~N\'{u}\~{n}ez, R.~Obaya, M.~Rasmussen. Rate-induced tipping and saddle-node bifurcation for quadratic differential equations with nonautonomous asymptotic dynamics. \emph{SIAM J. Appl. Dyn. Syst.} {\bf 20} (1), 500--540 (2021).
\bibitem{loos} I.P.~Longo, R.~Obaya, A.M.~Sanz. Tracking nonautonomous attractors in singularly perturbed systems of ODEs with dependence on the fast time. \emph{J. Differential Equations} {\bf 414}, 609--644 (2025).
\bibitem{LopesAmorGore2024} W.~Lopes, D.R.~Amor, J.~Gore. Cooperative growth in microbial communities is a driver of multistability. \emph{Nature Communications} {\bf 15}, 4709 (2024).
\bibitem{LuchinskyEtAl1999} D.G.~Luchinsky, R.~Mannella, P.V.E.~McClintock, N.G.~Stocks. Stochastic resonance in electrical circuits---I: Conventional stochastic resonance. \emph{IEEE Trans. Circuits Syst. II Analog Digit. Signal Process.} {\bf 46}, 1205--1214 (1999).
\bibitem{maslov2021} A.G.~Maslovskaya, L.I.~Moroz, A.Yu.~Chebotarev, A.E.~Kovtanyuk. Theoretical and numerical analysis of the Landau–Khalatnikov model of ferroelectric hysteresis. \emph{Commun. Nonlinear Sci. Numer. Simul.} {\bf 93}, 105524 (2021).
\bibitem{MayLevinSugihara2008} R.M.~May, S.A.~Levin, G.~Sugihara. Complex systems: Ecology for bankers. \emph{Nature} {\bf 451}, 893--895 (2008).
\bibitem{milon} V.M.~Million\v{s}\v{c}ikov. A proof of the existence of irregular systems of linear equations with quasiperiodic coefficients. \emph{Differ. Uravn.} {\bf 5}, 1979--1983 (1969).
\bibitem{NeneZaikin2010} N.R.~Nene, A.~Zaikin. Gene regulatory network attractor selection and cell fate decision: Insights into cancer multi-targeting. In \emph{Proceedings of Biosignal 2010}, 14--16 (2010).
\bibitem{nono} S.~Novo, C.~N\'{u}\~{n}ez, R.~Obaya. Almost automorphic and almost periodic dynamics for quasimonotone non-autonomous functional differential equations. \emph{J. Dynam. Differential Equations} {\bf 17} (3), 589--619 (2005).
\bibitem{SchefferEtAl2009} M.~Scheffer, J.~Bascompte, W.A.~Brock, V.~Brovkin, S.R.~Carpenter, V.~Dakos, H.~Held, E.H.~van Nes, M.~Rietkerk, G.~Sugihara. Early-warning signals for critical transitions. \emph{Nature} {\bf 461}, 53--59 (2009).
\bibitem{Schellnhuber2009} H.J.~Schellnhuber. Tipping elements in the Earth System. \emph{Proc. Natl. Acad. Sci. U.S.A.} {\bf 106}, 20561--20563 (2009).
\bibitem{sell2} G.R.~Sell. \emph{Topological Dynamics and Ordinary Differential Equations}. Van Nostrand Reinhold, London, 1971.
\bibitem{shyi4} W.~Shen, Y.~Yi. Almost Automorphic and Almost Periodic Dynamics in Skew-Product Semiflows. \emph{Mem. Amer. Math. Soc.} {\bf 647}, Amer. Math. Soc., Providence, 1998.
\bibitem{tikh} A.N.~Tikhonov. Systems of differential equations containing a small parameter multiplying the derivative. \emph{Mat. Sb.} {\bf 31}, 575--586 (1952).
\bibitem{VanselowEtAl2022} A.~Vanselow, L.~Halekotte, U.~Feudel. Evolutionary rescue can prevent rate-induced tipping. \emph{Theor. Ecol.} {\bf 15}, 29--50 (2022).
\bibitem{viana} M.~Viana, K.~Oliveira. \emph{Foundations of Ergodic Theory}. Cambridge Studies in Advanced Mathematics {\bf 151}, Cambridge, 2016.
\bibitem{vinograd} R.E.~Vinograd. On a problem of N.P.~Erugin. \emph{Differ. Uravn.} {\bf 11}, 632--638 (1975).
\bibitem{WackerSchluter2021} B.~Wacker, J.C.~Schl\"{u}ter. A cubic nonlinear population growth model for single species: Theory, an explicit--implicit solution algorithm and applications. \emph{Adv. Difference Equ.} {\bf 2021}, 236 (2021).
\bibitem{WangChenZhengYu2024} Y.~Wang, Y.~Chen, B.~Zheng, J.~Yu. Periodic dynamics of a mosquito population suppression model based on Wolbachia-infected males. \emph{Discrete Contin. Dyn. Syst.} {\bf 44} (8), 2403--2437 (2024).
\bibitem{WieczorekXieAshwin2023} S.~Wieczorek, C.~Xie, P.~Ashwin. Rate-induced tipping: Thresholds, edge states and connecting orbits. \emph{Nonlinearity} {\bf 36}, 3238--3293 (2023).
\bibitem{YukalovEtAl2009} V.I.~Yukalov, D.~Sornette, E.P.~Yukalova. Nonlinear dynamical model of regime switching between conventions and business cycles. \emph{J. Econ. Behav. Organ.} {\bf 70}, 206--230 (2009).
\end{thebibliography}
\end{document}